\documentclass[11pt]{amsart}
\usepackage[utf8]{inputenc}
\usepackage{amsmath, palatino, mathpazo, amsfonts, amssymb, mathtools}
\usepackage[mathscr]{eucal}
\usepackage[all]{xy}
\usepackage{datetime}
\usepackage{amsthm}
\usepackage{tkz-berge}
\usepackage{pgfplots}
\usepackage{tikz-3dplot}
\tdplotsetmaincoords{54.7356}{45}
\pgfplotsset{compat=newest}
\usepackage{comment}
\usepackage{xcolor}
\usepackage{subcaption}
\usepackage{graphicx}
\usepackage[export]{adjustbox}
\usepackage{marginnote}
\usepackage{enumerate}
\usepackage{float}
\usepackage{fourier}

\theoremstyle{definition}
\newtheorem{theorem}{Theorem}[section]
\newtheorem{lemma}[theorem]{Lemma}
\newtheorem{proposition}[theorem]{Proposition}

\theoremstyle{remark}

\newtheorem{remark}[theorem]{Remark}
\newtheorem{example}[theorem]{Example}
\newtheorem{definition}[theorem]{Definition}

\newcommand{\R}{\mathbb{R}}
\newcommand{\Z}{\mathbb{Z}}
\newcommand{\V}{\mathcal{V}}
\newcommand{\E}{\mathcal{E}}
\newcommand{\F}{\mathcal{F}}
\newcommand{\Q}{\mathcal{Q}}
\numberwithin{equation}{section}

\usepackage{todonotes}

\author{You-Cheng Chou}
\address{Institute of Mathematics, Academia Sinica, Taiwan}
\email{bensonchou72@gmail.com}

\author{Chin-Lung Wang}
\address{Department of Mathematics, National Taiwan University, Taipei, Taiwan}
\email{dragon@math.ntu.edu.tw}

\author{Po-Sheng Wu}
\address{Department of Mathematics, National Taiwan University, Taipei Taiwan}
\email{bobson860412@gmail.com}

\title[Lam\'e equations with finite monodromy]{Characterization and enumeration \\ on Lam\'e equations \\ with finite monodromy}

\date{\today}

\begin{document}
\maketitle

\begin{abstract}
We give a complete characterization of the classical Lam\'e equations $y'' = (n(n + 1)\wp(z) + B)y$, $n \in \Bbb R$, $B \in \Bbb C$ on flat tori $E_\tau = \Bbb C/(\Bbb Z + \Bbb Z\,\tau)$ with finite monodromy groups $M$. 

Beuker--Waall had shown that such $n$ must lie in a finite number of arithmetic progressions $n_i + \Bbb N \subset \Bbb Q$ and they determined all corresponding $M$. By combining the theory of dessin d'enfants with the geometry of spherical tori, we prove the existence of $(B, \tau)$ for each such $n$ and provide a description of all such $(n, B, \tau, M)$. 

In particular, for a given $(n, M)$ with $n \not\in \tfrac{1}{2} + \Bbb Z$, we prove the finiteness of $(B, \tau)$ and derive an explicit counting formula of them. (The case $n \in \tfrac{1}{2} + \Bbb Z$ is a classical result due to Brioschi--Halphen--Crawford.)

The main ingredients in this work are (1) the definition and classification of basic spherical triangles with finite monodromy and (2) the process of attaching cells corresponding to $n \mapsto n + 1$ which reduces the problem to the basic case. 
\end{abstract}

\small
\tableofcontents
\normalsize

\section{Introduction and statements of results}

\subsection{Schwarz triangles and Klein's pullback}
Linear differential equations $L\, y = 0$ on $\mathbb{CP}^1$ (ODE) with only algebraic solutions had been studied actively since the nineteenth century. It is equivalent to requiring that $L\, y = 0$ has finite monodromy group $M_L$. 

For Gauss' hypergeometric equations, i.e., second order ODE on $\mathbb{CP}^1$ up to 3 regular singular points, H.A.~Schwarz \cite{Schwarz_1879} classified all possible finite monodromy groups through \emph{spherical triangles} named after him. 
\begin{itemize}
\item[$\bullet$]
The angles correspond to the exponent differences and 
\item[$\bullet$]
the monodromy group $M$ is generated by the reflections across the boundary circles. 
\end{itemize}
Indeed for $\alpha$, $\beta$, $\gamma\in \mathbb{C}$, the normal form of hypergeometric operator is
\[
H_{\alpha,\beta,\gamma} := \frac{d^2}{dx^2} + \Bigg( \frac{1-\alpha^2}{4x^2} + \frac{1-\beta^2}{4(x-1)^2} + \frac{\alpha^2 + \beta^2 -\gamma^2 -1}{4x(x-1)} \Bigg).
\]
It has regular singularities at $0$, $1$, and $\infty$ with exponent difference $\alpha$, $\beta$, and $\gamma$ respectively. The numbers $(\alpha$, $\beta$, $\gamma)$ up to permutations, sign changes and addition of $(l,m,n) \in \mathbb{Z}^3$ with $l+m+n$ even, are of the same type as their monodromy groups are isomorphic. 

For $M$ being finite, it must be cyclic (the case with 2 singular points), dihedral, tetrahedral, octahedral, or icosahedral. The cyclic case is elementary. For the remaining four types of groups, a complete list of 15 types of equations, known as the Schwarz list, was given. 

Meanwhile, F.~Klein \cite{Klein_1877} showed that any 2nd order ODE on $\mathbb{CP}^1$ with finite monodromy must arise from a pullback of a hypergeometric equation in the Schwarz list with the same projective monodromy group $PM$. In particular, under pullbacks, the Schwarz list is reduced to the following basic Schwarz list corresponding uniquely to the five types of groups:
\begin{table}[ht]
\begin{tabular}{p{3cm}  l}
$\big(1, \ \frac{1}{N}, \ \frac{1}{N} \big) $ &  cyclic ($C_N$)    \\\\
$\big(\frac{1}{2}, \ \frac{1}{2}, \ \frac{1}{N}\big) $ &  dihedral ($D_N$)    \\\\
$\big(\frac{1}{2}, \ \frac{1}{3}, \ \frac{1}{3}\big) $ &  tetrahedral ($A_4$)         \\\\
$\big(\frac{1}{2}, \ \frac{1}{3}, \ \frac{1}{4}\big) $ &  octahedral ($S_4$)         \\\\
$\big(\frac{1}{2}, \ \frac{1}{3}, \ \frac{1}{5}\big) $ &  icosahedral ($A_5$)        
\end{tabular}
\end{table}

In practice we need to (i) decide in finite steps whether an ODE has finite monodromy. If it does, we need to (ii) develop an effective method to compute the pullback map $\mathbb{CP}^1 \to \mathbb{CP}^1$. Klein had proposed a program for (i). For (ii), it is largely open if the ODE has four or more singular points. 

Klein's theorem and his program were generalized by Baldassarri and Dwork to second order ODE on compact Riemann surfaces:

\begin{theorem}[\cite{Klein_1877, Baldassarri_Dwork_1979, Baldassarri_1980}] \label{t:KBD}
Let $L$ be a second order differential operator on a Riemann surface $C$ with regular singularities and finite projective monodromy. Then $L$ is the pullback of the unique $H_{\alpha,\beta,\gamma}$ with the same $PM$ in the basic Schwarz list by a map $C \to \mathbb{CP}^1$.
\end{theorem}

\subsection{Lam\'e equations}
The simplest higher genus case is an ODE on a torus with only one regular singular point, namely the Lam\'e equation. The finite monodromy theory of it, including classifications and enumerations, is the main topic of this paper. On $E_{\tau} = \mathbb{C} / (\mathbb{Z} + \mathbb{Z}\tau)$, it is given by  
%
%
%
%
%
\begin{equation} \label{eqn_Lame_ellipic}
L_{n,B}\, w := \frac{d^2 w}{dz^2} - \Big( n(n+1)\wp(z) + B \Big) w =0,
\end{equation}
where $n\in \mathbb{R}_{>0}$ and $B\in \mathbb{C}$. It has a regular singularity at $z=0$ with local exponents $\{-n, n+1\}$. Denote by $M$ (resp. $PM$) the monodromy group (resp. projective monodromy group) of equation \eqref{eqn_Lame_ellipic}.

Under $x=\wp(z)$, which induces a double cover $E \rightarrow \mathbb{P}^1$, we obtain its \emph{algebraic form} on $\mathbb{CP}^1$
\begin{equation} \label{eqn_lame_algebraic}
\widetilde{L}_{n,B}\, w := p(x) \frac{d^2 w}{dx^2} + \frac{1}{2} p'(x) \frac{dw}{dx} - \Big(n(n + 1)x + B \Big) w =0,
\end{equation}
where 
$$
p(x) = 4 x^3 -g_2 x - g_3 = 4 \prod\nolimits_{i=1}^3 (x - e_i).
$$ 
It has regular singular points at $\{e_1,e_2,e_3,\infty \}$ with local exponents $\{0, 1/2\}$ at $e_i$ and $\{ -n/2, (n+1)/2 \}$ at $\infty$. Denote by $\widetilde{M}$ and $P\widetilde{M}$ the corresponding groups of equation \eqref{eqn_lame_algebraic} over $\Bbb P^1$.

We are interested in the following question:
\medskip

\centerline{\emph{Classify all equations \eqref{eqn_Lame_ellipic} and \eqref{eqn_lame_algebraic} with finite monodromy.}}
\medskip

In the literature, many works have been done with a primary focus on the algebraic form (\ref{eqn_lame_algebraic}). The question has been studied in the following three special cases seperately: (1) $n\in \frac{1}{2} + \mathbb{\Bbb Z}_{\ge 0}$, (2) $n\in \mathbb{N}$, and (3) $n \notin \frac{\mathbb{N}}{2}$.

For (1), it was solved in late nineteenth century:
 
\begin{theorem}[Brioschi–Halphen–Crawford \cite{Crawford_1895}] \label{t:BHC}
Let $m \in \Bbb Z_{\ge 0}$. There exists a weighted homogeneous polynomial 
$$
p_m(B) = p_m(B; g_2, g_3)
$$ 
of degree $m + 1$ in $B$ with coefficients in $\mathbb{Z}[\frac{g_2}{4}, \frac{g_3}{4}]$ such that $p_m(B)=0$ if and only if the equation $\widetilde{L}_{m + 1/2,B}\, w = 0$ has finite projective monodromy $K_4$ (Klein-four). 
\end{theorem}

For (2), using Klein's theorem and the combinatorics of dessins d'enfants associated to the pullback (Belyi) map $\mathbb{CP}^1 \to \mathbb{CP}^1$, Dahmen proved a counting formula for equations (\ref{eqn_lame_algebraic}) with $n\in \mathbb{N}$ and $P\widetilde{M} \cong D_N$:

\begin{theorem}[\cite{Dahmen_2007_Dessin}] \label{thm_alg_dahmen} 

Let $n\in \mathbb{N}$. Denote by $P\widetilde{L}(n,N)$ the number of equations (\ref{eqn_lame_algebraic}) with $P\widetilde{M} \cong D_N$. If $N=1$, $P\widetilde{L}(n,1)=0$. If $N\geq 2$, 
\[
P\widetilde{L}(n,M) = \frac{n(n+1)}{12} \Big( \Phi(N) - 3\phi(N) \Big) + \frac{2}{3} \epsilon(n,N),
\]
where $\Phi(N) = \#\{(k_1, k_2) \mid {\rm gcd}(k_1, k_2, N) = 1,\, 0 \le k_i < N\}$ and
\[
\epsilon(n,N)=
\begin{cases}
1 \qquad \mbox{if $N=3$ and $n \equiv 1$ (mod $3$)},
\\
0 \qquad \mbox{otherwise}.
\end{cases}
\]
\end{theorem}

For the full monodromy group, let $\widetilde{L}(n, N)$ denotes the number of equations (\ref{eqn_lame_algebraic}) with $\widetilde{M} \cong D_N$. The following formula, conjectured by Dahmen, has been proved by Chen--Kuo--Lin \cite{Chen_Kuo_Lin_2021} using Painlev\'e equations, and later by Wu in his master thesis \cite{Wu_2022} using spherical structures on tori (cf.~\S \ref{s:SphG}).

\begin{theorem}[\cite{Chen_Kuo_Lin_2021, Wu_2022}] \label{thm_alg_dahmen_II}
With the same notation as in Theorem~\ref{thm_alg_dahmen}, 
\[
\widetilde{L}(n,N) = \frac{1}{2} \Bigg( \frac{n(n+1) \Phi(N) }{24} - \Big( a_n \phi(N) + b_n \phi(\frac{N}{2}) \Big) \Bigg) + \frac{2}{3} \epsilon(n,N),
\]
where $a_{2n}=a_{2n+1}= n(n+1)/2$ and $b_{2n} = b_{2n-1} = n^2$.
\end{theorem}

For the last case (3) $n\notin \mathbb{N}/2$, the list of all possible $n$ is known:

\begin{theorem}[Beukers--Waall \cite{Beukers_Waall_2004}] \label{Thm:B.-W.}
Let $0 < n \notin \mathbb{N}/2$. If $\widetilde{L}_{n,B}\, w = 0$ has finite monodromy group $\widetilde{M}$, then one of the following holds:
\begin{enumerate}
    \item $\widetilde{M}\cong G_{12}$ with $n \in \{ \pm \frac{1}{4} \}+ \mathbb{Z}$;
    \item $\widetilde{M}\cong G_{13}$ with $n \in \{ \pm \frac{1}{6} \} + \mathbb{Z}$;
    \item $\widetilde{M}\cong G_{22}$ with $n \in \{ \pm \frac{1}{10}, \pm \frac{3}{10}, \pm \frac{1}{6}\} + \mathbb{Z}$,
\end{enumerate}
where $G_i$ is the complex reflection group with Shepherd-Todd number $i$ \cite{Shephard_Todd_1954}.
\end{theorem}

In \cite{Beukers_Waall_2004}, an algorithm to construct a Lam\'e equation $\widetilde{L}_{n,B}\,w = 0$ for given $n, \widetilde{M}$ was provided by computing $\widetilde{M}$-invariant polynomials $J(w_1, w_2)$ where $w_1, w_2$ are independent solutions, and a list was computed for small $n$. However the complexity grows fast and it is hard to proceed for large $n$. 

Since the projective groups of $G_{12}$, $G_{13}$, and $G_{22}$ are different (being $A_4$, $S_4$, and $A_5$ respectively), to prove the existence for $n\notin \{\frac{1}{2}\} + \mathbb{Z}$, it suffices to consider $P\widetilde{M}$. Using pullbacks from the (not necessarily basic) Schwarz list, the existence for given finite $P\widetilde{M}$ was shown by Maier \cite{Maier_2004} in several, but not all, arithmetic progressions of the form $n_i + 3 \Bbb N$ or $n_j + 2\Bbb N$. The existence problem was finally solved by Chou in his bachelor thesis \cite{Chou_2016} by constructing the dessin for each possible $n \in \Bbb Q$ (cf.~\S \ref{s:dessin}).

%


\subsection{Summary of results}

The main tasks beyond existence result are (1) the exact counting formula and (2) explicit construction for each solution, given $n$ and a fixed finite ordinary or projective monodromy group on equation (\ref{eqn_Lame_ellipic}) or (\ref{eqn_lame_algebraic}).

In \S \ref{s:SphG} we give a complete answer to all the questions listed above. The main technique is the spherical structure on tori, which is motivated by the recent work of Eremenko et.~al.~\cite{Eremenko_Mondello_Panov_2023}. 


In \S \ref{ss:SphG}, we recall previous results on spherical tori, in particular the correspondence between Lam\'e equations of unitary monodromy and "balanced" spherical triangles (cf. Proposition~\ref{Prop:Sphe_tori_decomp}). The case of finite monodromy should be regarded as a special case of unitary monodromy.

In \S \ref{ss:basic}, we give a classification of "basic" spherical triangles with finite monodromies in {\bf Table~\ref{Table:Sph_triangle}}. (A spherical triangle is said to have finite monodromy means that the corresponding spherical torus has finite monodromy.) It consists of three types
\begin{enumerate}
    \item Non-dihedral type $(n \not\in \Bbb N /2)$: the list is classified by a given Platonic solid, the area $n\pi$, and the graph distance.
    \item Dihedral, non Klein-four type: the list is classified by edge lengths. 
    \item Klein-four type: we have complex 1-dimensional moduli of corresponding spherical tori. This is consistent with Theorem \ref{t:BHC}.
\end{enumerate}

The general balanced spherical triangles could be obtained from attaching digons of the shape of hemispheres to the edges. As a simple application, we count the number of spherical tori of given finite monodromies in {\bf Table~\ref{Table:Counting_sphe_tri}}. The counting results for $n \notin \mathbb{N}/2$ are new.

The list of all possible finite monodromy groups is given in {\bf Table~\ref{Table:List_mono_gps}}. For the type (1), the projective monodromy group uniquely determines the monodromy group. For the type (2), it takes extra care and will be discussed in \S \ref{ss:DF}.

In \S \ref{ss:DF}, we establish relations between edge lengths of the basic triangle and the monodromy group for the case of type (2). In particular, it gives a simple proof of Dahmen's formulas ({\bf Theorem \ref{thm_alg_dahmen}, \ref{thm_alg_dahmen_II}}). 

In \S \ref{s:dessin}, we explain the relation between spherical tori (with finite monodromy) and the corresponding dessin. We also discuss known algorithms to compute the pullback Belyi map via a given dessin. 

With the classification of basic spherical triangles, and hence spherical tori with finite monodromy, we have thus completed the classification of Lam\'e equations with finite monodromy, together with explicit constructions of the Lam\'e equations via dessins. 

\begin{remark}
This paper grew out as an attempt to solve the finite monodromy problem of Lam\'e equations by combining/extending the methods in the first author's bachelor thesis \cite{Chou_2016} (existence and description of conformal structure via dessin for $n \in \Bbb Q$) and the third author's master thesis \cite{Wu_2022} (enumeration for $n, N \in \Bbb N$ via spherical geometry), both supervised by the second author. For the sake of completeness we merge part of them, with extensions, in this paper rather than publishing them separately.      
\end{remark}

%

\section{Spherical geometry aspect} \label{s:SphG}

In this section we deal with the finite monodromy problem of Lam\'e equations \eqref{eqn_Lame_ellipic} and \eqref{eqn_lame_algebraic} within the aspect of spherical geometry.

\subsection{Spherical geometry} \label{ss:SphG}

\begin{definition} A spherical torus (with one conical singularity, omitted for short) $(T,x)$  is an oriented Riemannian surface $T$ of constant curvature $1$ and genus $1$, with a conical singularity $x$ of angle $2\pi\theta$, i.e., there is a local isometry of $T$ at $x$ to a spherical fan with a corner of angle $2\pi\theta$, and with its two edges being identified.
\end{definition}

Given a spherical torus, consider a developing map to the unit sphere
\[
f: \widetilde{T \backslash \{ x \} } \rightarrow S^2,
\]
where \ $ \widetilde{\cdot} $ \ denotes the universal cover. The pullback of the complex structure on $S^2 \cong \mathbb{CP}^1$ gives $T \backslash \{x\}$ a punctured Riemann surface structure. We sometimes identify a point on $T$ with its image on $S^2$ (with some choice of fundamental domain for $T$).

\begin{definition}
We say two spherical tori are projective equivalent if their developing maps $f_1$ and $f_2$ are differ by a M\"obius transformation $\gamma\in \textrm{PGL}(2,\mathbb{C})$, i.e., $f_2 = \gamma \circ f_1$.
\end{definition}

\begin{proposition}
There is a one-to-one correspondence between Lam\'e equations \eqref{eqn_Lame_ellipic} with unitary monodromy, and projective equivalence classes of spherical tori with a conical singularity of angle $(4n+2)\pi$.
\end{proposition}
\begin{proof}
Let $w_1$ and $w_2$ be two linearly independent solutions of equation \eqref{eqn_Lame_ellipic}. Then the pullback of the Fubini-Study metric on $\mathbb{CP}^1$ via $f=w_1/w_2$ gives $E$ a spherical structure with a conical singularity of angle $(4n+2)\pi$. The angle can be computed from the exponent differences of the equation.

Conversely, given a spherical torus with a conical singularity of angle $(4n+2)\pi$ and a developing map $f$, the corresponding Lam\'e equation is given by
\[
\frac{d^2w}{dz^2} + (Sf)(z) w=0,
\]
where $Sf$ is the Schwarz derivative of $f$. 
\end{proof}

\begin{definition}
A \textit{spherical polygon} is an oriented Riemannian surface of constant curvature $1$ with piecewise geodesic boundaries.
\end{definition}

\begin{definition} 
An interior angle $\pi\theta$ of a spherical polygon is said to be \textit{integral} if $\theta$ is integral.

\end{definition}

\begin{definition} A spherical triangle is called \textit{balanced} if the three interior angles $\pi\theta_1,\pi\theta_2,\pi\theta_3$ satisfy the triangle inequalities: 
$$
|\theta_2-\theta_3|\leq\theta_1\leq\theta_2+\theta_3,
$$ 
and is called semibalanced (resp.~strictly balanced) if any (resp.~none) of the equalities holds. 
\end{definition}

We have the following properties for balanced triangles.

\begin{proposition}[{\cite[Theorem 2.9]{Eremenko_Mondello_Panov_2023}}] A spherical triangle is balanced if and only if it contains a circumcenter, i.e.~a point which is equidistant to the three vertices. The circumcenter lies on some edge if and only if the spherical triangle is semibalanced and the circumcenter is the midpoint of the corresponding edge of the vertex with largest angle.
\end{proposition}

\begin{proposition}[{\cite[Corollary 2.14]{Eremenko_Mondello_Panov_2023}}] \label{short} The three sides of a balanced triangle are all less than $2\pi$.
\end{proposition}

\begin{proposition}[{\cite[Theorem~B, E]{Eremenko_Mondello_Panov_2023}}] \label{Prop:Sphe_tori_decomp}  
Let $(T,x)$ be a spherical torus with a conical singularity of angle $(4n+2)\pi > 0$.
\begin{enumerate}
\item If $n\not\in\mathbb Z$, then $T$ can be decomposed into two isometric balanced triangles. Conversely, any balanced triangle uniquely determines a spherical torus. A semibalanced triangle and its mirror image correspond to the same spherical torus. Also any two non-isometric balanced triangles correspond to different spherical tori.

\item If $n\in \mathbb Z$, every projective equivalence class of spherical tori is parametrized by $\mathbb R$ (modding out automorphisms). In each class, there is a unique spherical torus having the isometric spherical triangles decomposition as above.
\end{enumerate}
\end{proposition}

For any spherical torus with such a decomposition, we denote by $\triangle, \triangle'$ the two spherical triangles, and $L_i, P_i, Q_i$ (resp.~$L_i', P_i', Q_i'$) the edges, vertices and midpoint of edges on $\triangle$ (resp.~$\triangle'$), such that $P_1,P_2,P_3$ are ordered \emph{clockwise} on $\partial\triangle$ (so that $L_1,L_2$ correspond to the two periods of the fundamental domain of $E$). Denote $(\pi\theta_1,\pi\theta_2,\pi\theta_3)$ the three interior angles of the spherical triangle. The spherical triangle has area $(2n+1)\pi$ and sum of interior angles $\pi\theta$, where 
$$
\theta = \theta_1 + \theta_2 + \theta_3 = 2n + 1.
$$

To glue $\triangle$ and $\triangle'$ into $(T,x)$, we simply glue $L_i$ along $L_i'$ so that the orientations on both sides are compatible, $i=1,2,3$. All the six vertices of the two spherical triangles are thus identified to form the conical singularity (e.g. Figure~\ref{red_path}).

\subsection{Counting tori with finite monodromy} \label{ss:basic}

There are four kinds of monodromies, namely the ordinary or projective monodromy, with respect to the elliptic curve $E$ (resp. $\mathbb {CP}^1$ for the algebraic form), which are denote by $M$ and $PM$ (resp.~$\widetilde{M}$ and $P\widetilde{M}$). The projective monodromies lie in $\mathrm{SO}(3)\cong\mathrm{PU}(2)$. $M$ lies in $\mathrm{SU}(2)$, while $\widetilde{M} \subset \mathrm{U}(2)$ lies in the group generated by $\mathrm{SU}(2)$ and $iI_2$. 

It is known that the finiteness of the four monodromies are equivalent. Indeed, $M$ is an index $2$ subgroup of $\widetilde{M}$, $P\widetilde{M}$ and $PM$ are quotient images of $\widetilde{M}$ and $M$ with kernels lying in $\langle iI_2\rangle$ and $\langle -I_2\rangle$ respectively.

\begin{definition} 
Denote by $R(l) = R(p)$ the axial reflection along an axis $l = \vec{Op}\subset \R^3$ passing through a point $p\in S^2$.
\end{definition}

Denote by $\gamma_i\in \mathrm{U}(2)$ the local (ordinary) monodromy of the algebraic form \eqref{eqn_lame_algebraic} at $e_i\in \mathbb{CP}^1$, which satisfies $\gamma_i^2=1$ and $\det \gamma_i = -1$ for $i = 1, 2, 3$. 
$$
\widetilde{M} = \langle \gamma_1, \gamma_2, \gamma_3 \rangle, \qquad P\widetilde{M} = \langle R(Q_1), R(Q_2), R(Q_3) \rangle.
$$ 

For the elliptic form \eqref{eqn_Lame_ellipic}, $M$ is an index $2$ subgroup of $\widetilde{M}$ and
$$
M = \langle \gamma_1\gamma_3, \gamma_2\gamma_3 \rangle, \qquad PM = \langle R(Q_1)R(Q_3), R(Q_2)R(Q_3) \rangle.
$$

Given $Q_1,Q_2,Q_3\in S^2$, we can read out the projective monodromy directly from the viewpoint of spherical geometry. For ordinary monodromy this is slightly indirect. Consider the standard double cover $\mathrm{SU}(2)\rightarrow \mathrm{PU}(2)\cong\mathrm{SO}(3)$. We view $R(Q_i)\in \mathrm{SO}(3)$ as rotating an angle of $\pi$ counterclockwise along $Q_i$. As the rotation angle increases from $0$ to $\pi$, it gives a path on $\mathrm{SO}(3)$ from identity to $R(Q_i)$. Then $R(Q_i)$ has a unique lift 
$$
\widetilde{R(Q_i)}\in\mathrm{SU}(2)
$$ 
with the double cover using this path. $\gamma_i$ is then given by 
$$
\gamma_i = i\widetilde{R(Q_i)} \in \mathrm{U}(2).
$$

(The conditions $\gamma_i^2=1$ and $\det \gamma_i=-1$ already implies $
\gamma_i = \pm i\widetilde{R(Q_i)}
$. To determine the sign, we may assume $w_1, w_2$ to be the solutions of exponents $1/2,0$ at $e_i$ respectively, then along a loop around the origin, the ordinary monodromy would be $r_i=\begin{pmatrix}-1 & 0\\0 & 1\end{pmatrix}$. The projective monodromy of rotating an angle of $\delta$ along $0\in \mathbb{CP^1}$ would be parametrized as $z\mapsto e^{i\delta}z$, which lifts to $\begin{pmatrix}e^{i\delta/2} & 0\\0 & e^{-i\delta/2}\end{pmatrix}$ on $\mathrm{SU}(2)$, thus the unique lift is $\widetilde{R(Q_i)}=\begin{pmatrix}i & 0\\0 & -i\end{pmatrix}$ and we have $\gamma_i = i\widetilde{R(Q_i)}$).

The moduli of spherical triangles have been well studied by A. Eremenko et al.~\cite{Eremenko_Gabrielov,Eremenko_Mondello_Panov_2023}. Here we list two more properties of spherical triangles which are direct consequences of their works. 

\begin{proposition} Any spherical triangle $\triangle$ with $n\leq 1$ can be embedded into a hemisphere.
\end{proposition}

\begin{proposition}[{\cite[\S 3]{Eremenko_Mondello_Panov_2023}}] \label{p:EMP}
Let $\pi \theta_i$, $i = 1, 2, 3$ be the interior angles.
\begin{enumerate}
\item If $\theta_i \not\in \mathbb{N}$ for all $i$, then $\triangle$ exists if and only if 
$$
|\theta_1 - n_1| + |\theta_1 - n_2| + |\theta_1 - n_3| > 1
$$ 
for some integers $n_1,n_2,n_3$ with $\sum n_i$ even. Such spherical triangle $\triangle$ is uniquely determined by $\theta_i$'s up to isometry.

\item If there is exactly one of $\theta_i \in \mathbb{N}$, say $\theta_1$, then $\triangle$ exists if and only if either both $
\theta_1 + \theta_2 - \theta_3$ and $\theta_1 + \theta_3 - \theta_2$ are odd positive integers, or $ \theta_1 - (\theta_2 + \theta_3)$ 
is an odd positive integer. (The latter case won't happen when $\triangle$ is balanced.) Such spherical triangle $\triangle$ is uniquely determined by $\theta_1,\theta_2,\theta_3$ and the edge length of $L_2$. 

\item If at least two of $\theta_i$, say $\theta_1$ and $\theta_2$, then $\triangle$ exists if and only if $\theta_i \in \mathbb{N}$ for all $i$ with $\sum \theta_i$ being odd, and $\triangle$ must be balanced. Also $\triangle$ is uniquely determined by $\theta_i$'s and the edge lengths of $L_1, L_2$. 
\end{enumerate}
\end{proposition}

The main operation for the reduction step is the following,

\begin{definition} [Attaching hemispheres] Given a spherical triangle with an edge of length less than $2\pi$, we have a basic operation of attaching a digon of both interior angles $\pi$ (thus has the shape of a hemisphere) to this side. The resulting shape is a spherical triangle with the edge involved becoming the complement of the original edge.
\end{definition}

Notice that attaching a hemisphere on each of the corresponding edges of $\triangle$ and $\triangle'$ results in inserting a full sphere to the corresponding spherical torus. In particular it preserves the projective monodromy  However it changes the sign of ordinary monodromy whenever a path passes the edge involved.

The operation of attaching digons on the sides had appeared in the work of G. Mondello and D. Panov \cite{Mondello_Panov_2016} to show the existence of spherical triangles with given interior angles. In their article they attached digons with integral interior angles to the sides, which is equivalent to consecutively attaching hemispheres to the sides. For the sake of counting, we will use a slightly modified decomposition (into some hemispheres and a basic triangle) as the following.

\begin{definition} [Basic triangles] We call a balanced triangle \textit{basic} if either $n\leq 1$, or $1<n<2$ and it is a complement (w.r.t.~the unit sphere) of some strictly balanced triangle with $n<1$.
\end{definition}


\begin{proposition} Any balanced triangle can be obtained from repeatedly attaching hemispheres on the sides of a basic triangle in a unique way.
\end{proposition}

\begin{proof}
Suppose $\theta_i$ are all non-integers. By Proposition \ref{p:EMP}-(1), let $(\theta_1,\theta_2,\theta_3)$ to be the unique spherical triangle with these interior angles. We can write 
\begin{equation*}
\begin{split}
\theta_1 &= B + C + \phi_1, \\
\theta_2 &= A + C + \phi_2, \\
\theta_3 &= A + B + \phi_3,
\end{split}
\end{equation*}
satisfying one of the following two conditions:
\begin{enumerate}[(i)]
\item $0<\phi_i<1$ for all $i$, and $A,B,C$ are nonnegative integers,
\item $1<\phi_i<2$ for all $i$, and $A,B,C$ are nonnegative integers except possibly one of them being $-1$.
\end{enumerate}

For case (i), if $(\phi_1,\phi_2,\phi_3)$ is balanced, then $(\theta_1,\theta_2,\theta_3)$ is obtained from the basic triangle $(\phi_1,\phi_2,\phi_3)$ by adding $A,B,C$ hemispheres to the three edges respectively. If $(\phi_1,\phi_2,\phi_3)$ is unbalanced, say $\phi_1 > \phi_2 + \phi_3$, then we must have $A\geq 1$, and $(\theta_1,\theta_2,\theta_3)$ is obtained by adding $A-1,B,C$ hemispheres to the three edges of $(\phi_1, \phi_2 + 1, \phi_3 + 1)$, which is the complement of the strictly balanced triangle $(2 - \phi_1, 1 -\phi_2, 1 - \phi_3)$, hence is also basic.

For case (ii), $(\phi_1,\phi_2,\phi_3)$ is always balanced but not necessarily basic. If say $(\phi_1,\phi_2-1,\phi_3-1)$ is also balanced, then $(\theta_1,\theta_2,\theta_3)$ is obtained from the basic triangle $(\phi_1,\phi_2-1,\phi_3-1)$ by adding $A+1,B,C$ hemispheres to the three edges respectively. (Note that if say $A=-1$, then $(\phi_1,\phi_2-1,\phi_3-1)$ must be balanced, so it automatically falls into this case.) If none of 
$$
(\phi_1-1,\phi_2-1,\phi_3), \quad (\phi_1-1,\phi_2,\phi_3-1), \quad (\phi_1,\phi_2-1,\phi_3-1)
$$ 
is balanced, then $(\theta_1,\theta_2,\theta_3)$ is obtained by adding $A,B,C$ hemispheres to the three edges of $(\phi_1,\phi_2,\phi_3)$, which is the complement of the strictly balanced triangle $(2-\phi_1,2-\phi_2,2-\phi_3)$, hence is also basic. 

For the case when one or three of the $\theta_i$ are integral, the full description of balanced triangles with one or three integral interior angles in {\cite[Proposition 3.5, 3.8]{Eremenko_Mondello_Panov_2023}} leads to the result. We can write
\begin{equation*}
\begin{split}
\theta_1 &= B + C + 1, \\
\theta_2 &= A + C + \phi, \\
\theta_3 &= A + B + \phi,
\end{split}
\end{equation*}
where $A$, $B$, $C$ are nonnegative integers, $$\dfrac{1}{2}<\phi<1\qquad\textnormal{or}\qquad 1<\phi<\dfrac{3}{2}$$ for the case of one integral angle (assuming $\theta_1$ is integral), and $\phi=1$ for the case of three integral angles. 
The basic triangle would be a spherical triangle of interior angles $\pi$, $\pi\phi$, $\pi\phi$ which forms the shape of a spherical digon or a hemisphere respectively. $(\theta_1,\theta_2,\theta_3)$ is then obtained from this basic triangle by adding $A,B,C$ hemispheres to the three edges.
\end{proof}

\begin{figure}
\includegraphics[scale=0.5]{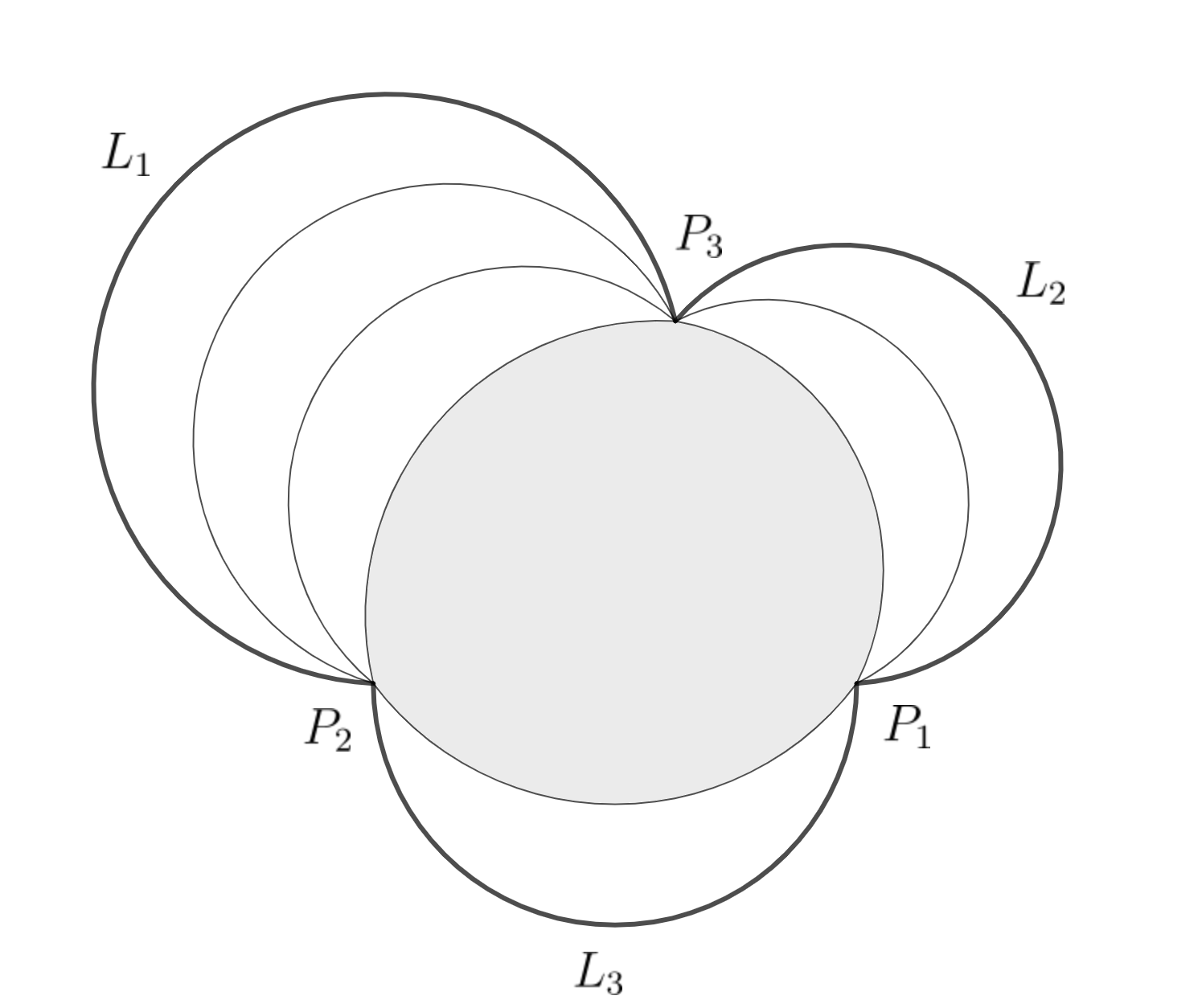}\\
\caption{A demonstration of the decomposition of a balanced spherical triangle into hemispheres and a basic triangle. All the unshaded regions in the figure are hemispheres, and the shaded region is the basic spherical triangle.}
\end{figure}

\begin{remark} 
The above argument for non-integral case is better illustrated by considering the position of the circumcenter of the basic triangle. Note that for any three given points on $S^2$ in general position, there are two choices of circumcenters. With the chosen circumcenter lying on each possible position determined by the extension of the three edges (three great circles), there is a unique basic triangle with this circumcenter. 
See Figure~\ref{Figure:Position_circumcenter}.
\end{remark}

\begin{figure}
\includegraphics[scale=0.1]{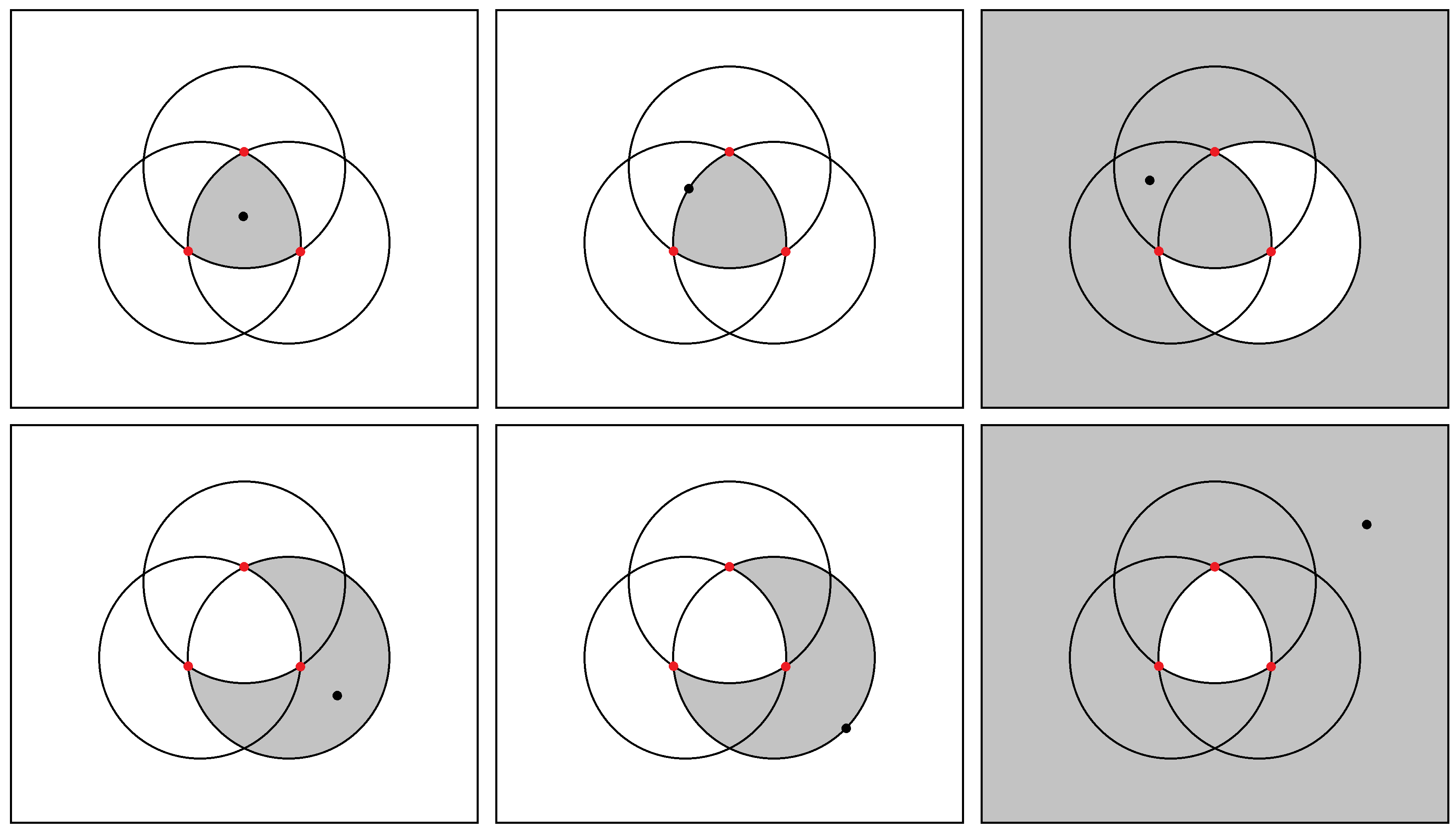}\\
\caption{Demonstration of all possible positions of the circumcenter (black) and its corresponding basic triangle, given three vertices (red) not lying on a great circle.}
\label{Figure:Position_circumcenter}
\end{figure}

Consider a basic triangle $\triangle = P_1P_2P_3$ corresponding to a spherical torus $(T,x)=\triangle \cup \triangle'$ with finite monodromy group. By proposition~\ref{short}, the three vertices of a basic triangle are mapped to three different points on the unit sphere. 

\begin{lemma}\label{Lemma:Klein-four}
Given three lines $l_1$, $l_2$, $l_3$ through the center of $S^2$,
$$
R(l_3)R(l_2)R(l_1) = \mathrm{id}
$$ 
if and only if $l_1$, $l_2$, $l_3$ are mutually orthogonal to each other.
\end{lemma}

\begin{proof} 

Suppose that $R(l_3)R(l_2)R(l_1) = \mathrm{id}$, then $R(l_3) R(l_2) = R(l_1)$ and then 
$$
l_2' := R(l_1) l_2 = R(l_3) R(l_2) l_2 = R(l_3) l_2.
$$ 
Also $l_1 \ne l_3$, as $l_1 = l_3$ implies $R(l_2) = \mathrm{id}$, which is a contradiction. 

Suppose $l_2'\neq l_2$, then $l_1$ and $l_3$ are orthogonal, since they will be the two angle bisectors of $l_2$ and $l_2'$. But then $R(l_2)=R(l_3)R(l_1)$ shows that $l_2$ is orthogonal to both $l_1$ and $l_3$, contradict to that $l_2'\neq l_2$. This shows that $l_2$ must be fixed by $R(l_1)$ and $R(l_3)$, and thus $l_1,l_3$ must be orthogonal to $l_2$. By applying the same argument to 
$$
R(l_2)R(l_1)R(l_3)=\mathrm{id}
$$ 
we obtain the result.
\end{proof}

\begin{definition} \label{d:VEF}
We define the three sets $\V$, $\E$, $\F$ on $S^2$ by\\
$\V$: the set of vertices of a regular $n$-gon on a great circle, or a platonic solid.\\
$\E$: the set of midpoints of the edges of the shape above.\\
$\F$: the set of centers of the faces of the shape above (for regular $n$-gons this means the two poles of rotational axis of the great circle).\\
\end{definition}


\begin{proposition} \label{p:keyP}
For $n\not\in\Z+\frac{1}{2}$, $P_1$, $P_2$, $P_3$ must lie on some choice of $\V$, i.e.~the vertices of either a regular $n$-gon on a great circle or a platonic solid.
\end{proposition}

\begin{proof}
It is well known that any finite subgroup of $\mathrm{SO}(3)$ fixes some finite set $\V$ as in Definition \ref{d:VEF}. Consider the composition $$R(Q_3)R(Q_2)R(Q_1)\in \mathrm{SO}(3).$$ Notice that $$R(
Q_3)R(Q_2)R(Q_1)P_2=P_2,$$ so $P_2$ must lies on the rotational axis of $R(Q_3)R(Q_2)R(Q_1)$ if the composition is not identity. 

The cases where $$R(Q_3)R(Q_2)R(Q_1)=\mathrm{id}$$ would have Klein-four monodromy group by Lemma \ref{Lemma:Klein-four}, and correspond to the case $n\in \Z+\frac{1}{2}$. It is not hard to argue that any axial rotations fixing $\V$ must have axes passing through $\V\cap \E\cap \F$, so it turns out that we only have to rule out the case $P_i\in \E$ for platonic solid cases (If $P_i\in \F$ for all $i$, then replace the solid $\V$ by the dual solid $\F$).

For the case of icosahedron or dodecahedron, note that $\E$ forms the set of vertices of a icosidodecahedron, which consists of $12$ regular pentagons and $20$ regular triangles. As there is a covering of the spherical torus which is also the covering of the sphere, the spherical torus should consists of $3k$ spherical regular pentagons and $5k$ spherical regular triangles for some integer $k$. Thus the area of spherical torus is a multiple of $4\pi/4$ and thus $n\in \frac{1}{2}\Z$. The octahedral and cubical cases are similar, and the set $\E$ for a tetrahedron is actually an octahedron. The proof is complete.
\end{proof}

\begin{definition}\label{Def:Q} Denote by 
$$
\Q = \{ p \in \V\cup \E\cup \F \mid R(p)\V = \V \}.
$$ 
In explicit terms, $\Q$ is given by
\begin{itemize}
\item $\V\cup \E$ for the regular $n$-gonal case when $n$ is odd,
\item $\V\cup \E\cup \F$ for the regular $n$-gonal case when $n$ is even,
\item $\V\cup \E$ for the octahedral case,
\item $\E\cup \F$ for the cubical case,
\item $\E$ for the icosahedral or dodecahedral cases.
\end{itemize}
\end{definition}





We can find all basic triangles with finite monodromy for $n\leq 1$ by exhaustion. We sketch our procedure for determining the table:

By Proposition \ref{p:keyP}, it boils down to determine all $3$-subsets 
$$
\{P_1, P_2, P_3\} \subset \V
$$ 
and the edges $L_i$ between them, such that the midpoints of the edges $Q_i$ lie in $\Q$. The positions of $P_1$, $P_2$ and $P_3$ fall into one of the following three cases.

\begin{enumerate}
\item If the chosen three points on the unit sphere are not lying on the same great circle, then there are four spherical triangles whose vertices maps to the three given points and with $n<1$, illustrated in Figure \ref{small_triangle_fig}. Either one of them is balanced, or two of them are semibalanced, depending on whether the circumcenter of $P_1,P_2,P_3$ lies on a great circle passing any two points.

\begin{figure}
\includegraphics[scale=0.2]{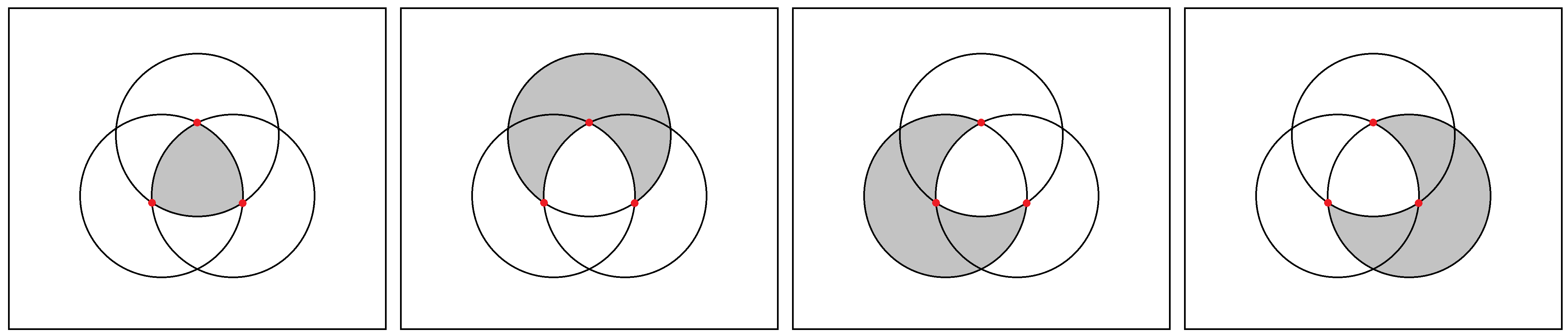}\\
\caption{The shaded regions are all spherical triangles of $n<1$ with three given vertices on $S^2$.}
\label{small_triangle_fig}
\end{figure}

\item If two of the chosen points are antipodal to each other, then the basic triangle must form the shape of a spherical digon. There is a choice for the edge of length $2\pi$. However the midpoint of this edge must lies in $\Q$, so there is a finite number of choices of this edge. Also, with the balancing condition, the angle of the digon should be no less than $\pi/2$. If the angle equals to $\pi/2$, then $n=1/2$ and it belongs to the Klein-four case.

\item If the three chosen points lie on the same great circle and have no antipodal pair of points, then the basic triangle must form the shape of a hemisphere, hence $n=1$ in this case. The three edges should have lengths being some rational multiples of $\pi$ in order to have finite monodromy.
\end{enumerate}

We exhaust all triples of distances for each case in Table \ref{Table:Sph_triangle}. Here "distance" means the graph distance between each two vertices. We use $'$ to indicate that the corresponding edge of the spherical torus is chosen to be the superior arc rather than the inferior one. Note that for the dodecahedral case, we can exclude those triangles with some edges of graph distance $2$, because the midpoints of such edges must not lie in $\Q$.

By appending the complements of all strictly balanced triangles of $n<1$ into the list, we obtain the full list of basic triangles.

\begin{table}[h]
\begin{tabular}{|l|l|l|l|}
\hline
polyhedron   & $n$    & distance                & note\\ \hline
octahedral   & $1/4, 7/4$  & $1,1,1$                 & regular \\ \hline
             & $3/4, 5/4$  & $1,1,2$                 & \\ \hline
cubical      & $1/6$ & $1,1,2$                 & semibalanced \\ \hline
             & $5/6$ & $1,1,2'$                & semibalanced \\ \hline
             & $5/6, 7/6$  & $1,2,3$                 & \\ \hline
             & $5/6, 7/6$  & $1,3,2$                 & \\ \hline
icosahedral  & $1/10, 19/10$ & $1,1,1$                 & regular \\ \hline
             & $3/10, 17/10$ & $1,2,2$                 & \\ \hline
             & $7/10, 13/10$ & $1,2,3$                 & \\ \hline
             & $7/10, 13/10$ & $1,3,2$                 & \\ \hline
             & $7/10, 13/10$ & $2,2,2$                 & regular \\ \hline
             & $9/10, 11/10$ & $1,1,2'$                & \\ \hline
             & $9/10, 11/10$ & $1,2,3$                 & \\ \hline
             & $9/10, 11/10$ & $1,3,2$                 & \\ \hline
dodecahedral & $1/6, 11/6$ & $1,3,3$                 & \\ \hline
             & $5/6, 7/6$ & $1,3,4'$                & \\ \hline
             & $5/6, 7/6$ & $1,4',3$                & \\ \hline
             & $5/6, 7/6$ & $1,4,5$                 & \\ \hline
             & $5/6, 7/6$ & $1,5,4$                 & \\ \hline
             & $5/6, 7/6$ & $3,3,4$                 & \\ \hline
dihedral     & $1$    & $k_1,k_2,k_3$ (coprime) & \\ \hline
Klein-four   & $1/2, 3/2$ & -                       & \\ \hline
\end{tabular}
\caption{Full list of basic spherical triangles.}
\label{Table:Sph_triangle}
\end{table}

\medskip

We proceed to determine the number of spherical tori of each given type. Note that to increase the value of $n$ by $m-1$, where $m\in\Z_{>0}$, we add $m_1$, $m_2$, $m_3$ hemispheres to the three edges of the basic triangle such that 
$$
m_1 + m_2 + m_3 = m - 1.
$$ 
Hence, there are $m(m+1)/2$ such balanced triangles, without removing repetitions. For semibalanced triangles, the mirrored pair of triangle should be counted as the same. Also, for basic triangles with $3$-fold symmetry (regular), cyclic permutations on the vertices should be counted as the same. We can then enumerate such spherical tori as shown in Table $\ref{Table:Counting_sphe_tri}$.

\begin{table}[h]
\begin{tabular}{|l|l|l|l|}
\hline
case         & $n$                                                          & distance                                                          & \#                                                                                                         \\ \hline
octahedral   & $m\pm 3/4$                                                   & $1,1,1$                                                           & $\lceil m(m+1)/6\rceil$                                                                                    \\ \hline
             & $m\pm 1/4$                                                   & $1,1,2$                                                           & $m(m+1)/2$                                                                                                 \\ \hline
cubical      & $m-5/6$                                                      & $1,1,2$                                                           & $m(m+1)/2-\lfloor m/2\rfloor$                                                                              \\ \hline
             & $m-1/6$                                                      & $1,1,2'$                                                          & $m(m+1)/2-\lfloor m/2\rfloor$                                                                              \\ \hline
             & $m\pm 1/6$                                                   & $1,2,3$                                                           & $m(m+1)/2$                                                                                                 \\ \hline
             & $m\pm 1/6$                                                   & $1,3,2$                                                           & $m(m+1)/2$                                                                                                 \\ \hline
icosahedral  & $m\pm 9/10$                                                  & $1,1,1$                                                           & $\lceil m(m+1)/6\rceil$                                                                                    \\ \hline
             & $m\pm 7/10$                                                  & $1,2,2$                                                           & $m(m+1)/2$                                                                                                 \\ \hline
             & $m\pm 3/10$                                                  & $1,2,3$                                                           & $m(m+1)/2$                                                                                                 \\ \hline
             & $m\pm 3/10$                                                  & $1,3,2$                                                           & $m(m+1)/2$                                                                                                 \\ \hline
             & $m\pm 3/10$                                                  & $2,2,2$                                                           & $\lceil m(m+1)/6\rceil$                                                                                                 \\ \hline
             & $m\pm 1/10$                                                  & $1,1,2'$                                                          & $m(m+1)/2$                                                                                                 \\ \hline
             & $m\pm 1/10$                                                  & $1,2,3$                                                           & $m(m+1)/2$                                                                                                 \\ \hline
             & $m\pm 1/10$                                                  & $1,3,2$                                                           & $m(m+1)/2$                                                                                                 \\ \hline
dodecahedral & $m\pm 5/6$                                                   & $1,3,3$                                                           & $m(m+1)/2$                                                                                                 \\ \hline
             & $m\pm 1/6$                                                   & $1,3,4'$                                                          & $m(m+1)/2$                                                                                                 \\ \hline
             & $m\pm 1/6$                                                   & $1,4',3$                                                          & $m(m+1)/2$                                                                                                 \\ \hline
             & $m\pm 1/6$                                                   & $1,4,5$                                                           & $m(m+1)/2$                                                                                                 \\ \hline
             & $m\pm 1/6$                                                   & $1,5,4$                                                           & $m(m+1)/2$                                                                                                 \\ \hline
             & $m\pm 1/6$                                                   & $3,3,4$                                                           & $m(m+1)/2$                                                                                                 \\ \hline
dihedral     & $m$                                                          & \begin{tabular}[c]{@{}l@{}}$k_1,k_2,k_3$\\ (coprime)\end{tabular} & \begin{tabular}[c]{@{}l@{}}$\lceil m(m+1)/6\rceil$ if $k_1=k_2=k_3=1$,\\ $m(m+1)/2$ otherwise\end{tabular} \\ \hline
Klein-four   & $m\pm 1/2$ & -                                                                 & Complex 1-dimensional                                                                          \\ \hline
\end{tabular}
\caption{Counting formulae for spherical triangles.}
\label{Table:Counting_sphe_tri}
\end{table}

\begin{remark}
For $n \in \frac{1}{2} + \Bbb Z_{\ge 0}$, the moduli space of the spherical tori is real $2$-dimensional with a complex structure defined by the Brioschi--Halphen polynomial as mentioned in Theorem \ref{t:BHC}. (See also \cite[\S 5]{Eremenko_Mondello_Panov_2023}).
\end{remark}
\medskip

Finally we can determine all the four types of monodromy groups from the constructions. The results are shown in Table \ref{Table:List_mono_gps} where the following notations are used: 
\begin{itemize}
\item
$G_i$ is the complex reflection group of Shephard–Todd number $i$.
\item
$G_i^+$ is the index $2$ normal subgroup of $G_i$ consisting of elements with determinant $1$. 
\item
In the dihedral case, $k_1$, $k_2$ and $k_3$ sum up to $k$, and $\widetilde M = D_k$ or $D_{2k}$ depending on whether it contains $-I\in \mathrm{U}(2)$. 
\item
In the Klein-four $K_4$ case, $\mathcal P_1=G(4,2,2)$ is the Pauli group.
\end{itemize}

\begin{table}[h]
\begin{tabular}{|l|ll|ll|}
\hline
case         & \multicolumn{2}{l|}{$E$}                         & \multicolumn{2}{l|}{$\mathbb {CP}^1$}             \\ \hline
             & \multicolumn{1}{l|}{$M$}                   & $PM$    & \multicolumn{1}{l|}{$\widetilde{M}$}                 & $P\widetilde{M}$    \\ \hline
octahedral   & \multicolumn{1}{l|}{$G_{12}^+$}          & $A_4$ & \multicolumn{1}{l|}{$G_{12}$}          & $S_4$ \\ \hline
cubical      & \multicolumn{1}{l|}{$G_{13}^+$}          & $S_4$ & \multicolumn{1}{l|}{$G_{13}$}          & $S_4$ \\ \hline
icosahedral  & \multicolumn{1}{l|}{$G_{22}^+$}          & $A_5$ & \multicolumn{1}{l|}{$G_{22}$}          & $A_5$ \\ \hline
dodecahedral & \multicolumn{1}{l|}{\textbf{$G_{22}^+$}} & $A_5$ & \multicolumn{1}{l|}{$G_{22}$}          & $A_5$ \\ \hline
dihedral     & \multicolumn{1}{l|}{$C_k$ or $C_{2k}$}   & $C_k$ & \multicolumn{1}{l|}{$D_k$ or $D_{2k}$} & $D_k$ \\ \hline
Klein-four   & \multicolumn{1}{l|}{$Q_8$}          & $K_4$ & \multicolumn{1}{l|}{$\mathcal P_1$}        & $K_4$ \\ \hline
\end{tabular}
\caption{Full list of all four types of monodromy groups}
\label{Table:List_mono_gps}
\end{table}

\subsection{Dahmen's formulas} \label{ss:DF}

As an application, we provide simple proofs of the counting formulas (Theorem \ref{thm_alg_dahmen}, \ref{thm_alg_dahmen_II}) discovered by S. Dahmen.




We first establish relations between edge lengths and monodromies. For $n\in \mathbb N$, we have seen that a (balanced) spherical triangle $\triangle$ is constructed from a basic triangle of the shape a hemisphere, by repeatedly attaching hemispheres to its edges.

The previous results in \cite{Lin_Wang_2017, Beukers_Waall_2004} already showed that for $n\in\mathbb N$ Lam\'e equation with unitary monodromy has ordinary monodromy of the form
\begin{equation}\label{nonproj_mono_eq}
\begin{cases}
w_{\pm}(z + \omega_1) = e^{\mp 2i\pi s} w_{\pm}(z), \\ 
w_{\pm}(z + \omega_2) = e^{\pm 2i\pi t} w_{\pm}(z),
\end{cases}
\end{equation}
while the projective monodromy has the form
\begin{equation}\label{proj_mono_eq}
\begin{cases}
f(z + \omega_1) = e^{-4i\pi s} f(z), \\ 
f(z + \omega_2) = e^{4i\pi t} f(z).
\end{cases}
\end{equation}
where $f=w_+/w_-$ is the ratio of the two solutions of the Lam\'e equation. Thus the pair $s, t \pmod 1$ determines the monodromy of equation \eqref{eqn_Lame_ellipic}, while the pair $2s, 2t \pmod 1$ determines the projective monodromy. We can also derive this result using spherical geometry, and the domain of possible $s,t\in \mathbb R$ will be determined in the next proposition.

\begin{proposition} 
For fixed parameters $\theta_1, \theta_2, \theta_3\in\{1,2,\dots,n\}$ with $\theta = \theta_1 + \theta_2 + \theta_3 = 2n + 1$, the monodromy parameters $s, t \pmod 1$ for the solutions $w_{\pm}$ satisfy either 
\[
s < \frac{\theta_1}{2}, \quad t < \frac{\theta_2}{2}, \quad s + t > \frac{\theta_1 + \theta_2 - 1}{2},
\] 
or 
\[
s > -\frac{\theta_1}{2}, \quad t > -\frac{\theta_2}{2}, \quad s + t < -\frac{\theta_1 + \theta_2 - 1}{2},
\] 
for some choice of representatives $s, t\in \mathbb R$.
\end{proposition}

\begin{proof} We first consider the case $n=1$ and $\theta_1=\theta_2=\theta_3=1$. Denote $\ell_1,\ell_2,\ell_3$ the edge lengths of this basic triangle. 

Suppose the center of $\triangle$ is mapped to $0\in \mathbb{CP}^1$, then Figure \ref{red_path} shows that the ordinary monodromy has the form \eqref{nonproj_mono_eq}, where
\[
s = (\ell_2 + \ell_3)/4\pi , \quad t = (\ell_1 + \ell_3)/4\pi,
\]
and therefore
\[
s < \frac{1}{2},\quad t < \frac{1}{2},\quad s + t > \frac{1}{2},
\] 
which follow from 
\[
\ell_2 + \ell_3 < 2\pi, \quad \ell_1 + \ell_3 < 2\pi,\quad \ell_1 + 2\ell_2 + \ell_3 = 2\pi + \ell_2 > 2\pi.
\] 

(The projective monodromy corresponding to $z\mapsto z+\omega_1$ should be regarded as the composition $R(Q_3)R(Q_2)^{-1}$. It lifts to the ordinary monodromy $$(i\widetilde{R(Q_3)})(-i\widetilde{R(Q_2)}^{-1})=\widetilde{R(Q_3)}\widetilde{R(Q_2)}^{-1},$$ which is the lift of a rotation of angle $\ell_2+\ell_3$, hence the result above.)

\begin{figure}
\includegraphics[scale=0.5]{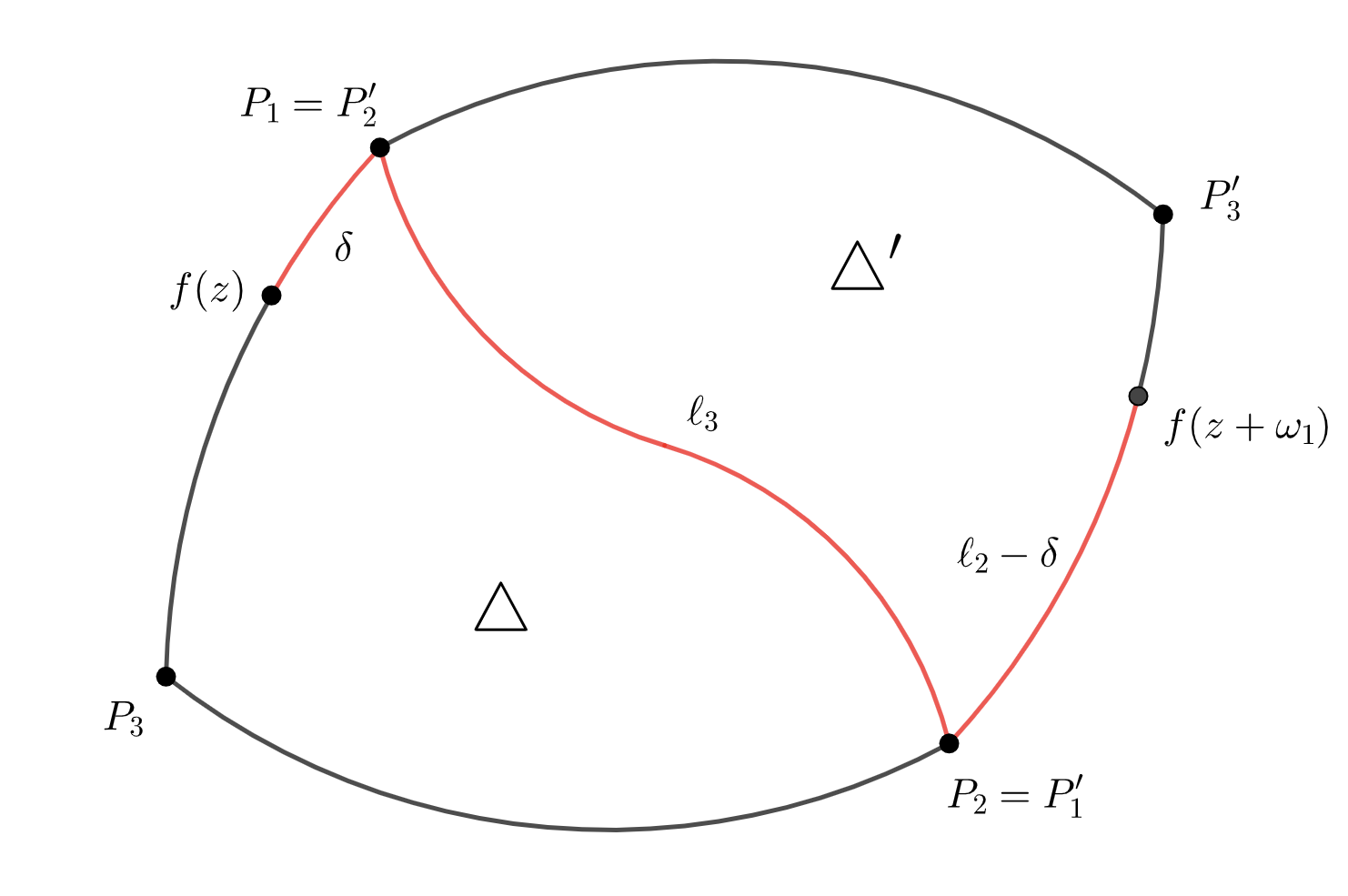}\\
\caption{A fundamental domain of $(T,x)$ with $n=1$. The image of the red path under the map $f$ is actually a smooth arc of length $\ell_2+\ell_3$ on the unit circle.}
\label{red_path}
\end{figure}

Conversely, if $s$ and $t$ satisfies the inequalities, then we can solve 
\[
(\ell_1,\ell_2,\ell_3) = (2\pi - 4\pi t, 2\pi - 4\pi s, 4\pi t + 4\pi s - 2\pi).
\]

If the center of $\triangle$ is mapped to $\infty\in \mathbb{CP}^1$, then $f$ has inverse monodromy, and we have  
\[
s > -\frac{1}{2},\quad t > -\frac{1}{2},\quad s + t < -\frac{1}{2}.
\] 

For $n>1$, note that attaching two hemispheres on the side of $L_2$ or $L_3$ will translate the parameter $s$ by $\frac{1}{2}$, so in total $s$ is translated by $\frac{\theta_1-1}{2}$. A similar argument holds for $t$. The result follows.
\end{proof}

The domain for projective monodromy parameters $2s, 2t$ is simply determined by doubling the ordinary monodromy parameters. Translating these regions into $(0,1)\times (0,1)$, we obtain the distribution of monodromy parameters shown in Figure~\ref{Figure:Number_deve_maps} and Figure~\ref{Figure:Number_ansatz_sol}. 

When the basic triangle has edge lengths being rational multiples of $\pi$, we have $s,t\in\mathbb Q$ and the resulting spherical torus will have finite monodromy. Now we are ready to prove Dahmen's formulae by the above proposition.


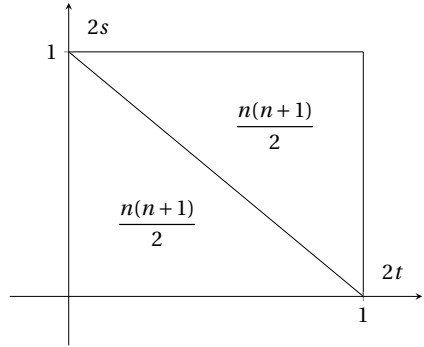
\begin{figure}
\begin{tikzpicture}[scale=0.8]
\begin{axis}[
    xmin=-0.2, xmax=1.2,
    xtick={0,1},
    ymin=-0.2, ymax=1.2,
    ytick={0,1},
    axis lines=center,
    axis on top=true,
    domain=0:1,
    ]
    \addplot [mark=none] coordinates {(1, 0) (0, 1)};
    \addplot [mark=none] coordinates {(0, 1) (1, 1)};
    \addplot [mark=none] coordinates {(1, 0) (1, 1)};
    \node at (axis cs: 0.7,0.7) {$\dfrac{n(n+1)}{2}$};
    \node at (axis cs: 0.3,0.3) {$\dfrac{n(n+1)}{2}$};
    \node at (axis cs: 1.1,0.1) {$2t$};
    \node at (axis cs: 0.1,1.1) {$2s$};
\end{axis}
\end{tikzpicture}
\caption{The number of Lam\'e equations with unitary monodromy and with given parameter $2s, 2t \pmod 1$ for $n\in\mathbb Z$.}
\label{Figure:Number_deve_maps}
\end{figure}

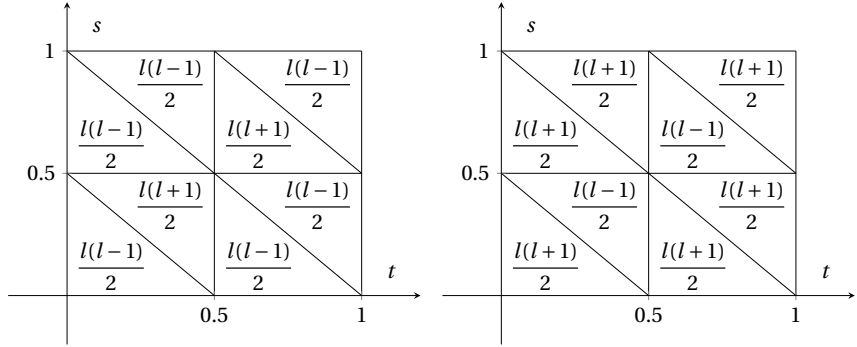
\begin{figure}
\begin{subfigure}{0.45\textwidth}
\begin{tikzpicture}[scale=0.8]
\begin{axis}[
    xmin=-0.2, xmax=1.2,
    xtick={0,0.5,1},
    ymin=-0.2, ymax=1.2,
    ytick={0,0.5,1},
    axis lines=center,
    axis on top=true,
    domain=0:1,
    ]
    \addplot [mark=none] coordinates {(0.5, 0) (0, 0.5)};
    \addplot [mark=none] coordinates {(1, 0) (0, 1)};
    \addplot [mark=none] coordinates {(1, 0.5) (0.5, 1)};
    \addplot [mark=none] coordinates {(0, 0.5) (1, 0.5)};
    \addplot [mark=none] coordinates {(0, 1) (1, 1)};
    \addplot [mark=none] coordinates {(0.5, 0) (0.5, 1)};
    \addplot [mark=none] coordinates {(1, 0) (1, 1)};
    \node at (axis cs: 0.15,0.12) {$\dfrac{l(l-1)}{2}$};
    \node at (axis cs: 0.65,0.12) {$\dfrac{l(l-1)}{2}$};
    \node at (axis cs: 0.15,0.62) {$\dfrac{l(l-1)}{2}$};
    \node at (axis cs: 0.65,0.62) {$\dfrac{l(l+1)}{2}$};
    \node at (axis cs: 0.35,0.37) {$\dfrac{l(l+1)}{2}$};
    \node at (axis cs: 0.85,0.37) {$\dfrac{l(l-1)}{2}$};
    \node at (axis cs: 0.35,0.87) {$\dfrac{l(l-1)}{2}$};
    \node at (axis cs: 0.85,0.87) {$\dfrac{l(l-1)}{2}$};
    \node at (axis cs: 1.1,0.1) {$t$};
    \node at (axis cs: 0.1,1.1) {$s$};
\end{axis}
\end{tikzpicture}
\end{subfigure}
\begin{subfigure}{0.45\textwidth}
\begin{tikzpicture}[scale=0.8]
\begin{axis}[
    xmin=-0.2, xmax=1.2,
    xtick={0,0.5,1},
    ymin=-0.2, ymax=1.2,
    ytick={0,0.5,1},
    axis lines=center,
    axis on top=true,
    domain=0:1,
    ]
    \addplot [mark=none] coordinates {(0.5, 0) (0, 0.5)};
    \addplot [mark=none] coordinates {(1, 0) (0, 1)};
    \addplot [mark=none] coordinates {(1, 0.5) (0.5, 1)};
    \addplot [mark=none] coordinates {(0, 0.5) (1, 0.5)};
    \addplot [mark=none] coordinates {(0, 1) (1, 1)};
    \addplot [mark=none] coordinates {(0.5, 0) (0.5, 1)};
    \addplot [mark=none] coordinates {(1, 0) (1, 1)};
    \node at (axis cs: 0.15,0.12) {$\dfrac{l(l+1)}{2}$};
    \node at (axis cs: 0.65,0.12) {$\dfrac{l(l+1)}{2}$};
    \node at (axis cs: 0.15,0.62) {$\dfrac{l(l+1)}{2}$};
    \node at (axis cs: 0.65,0.62) {$\dfrac{l(l-1)}{2}$};
    \node at (axis cs: 0.35,0.37) {$\dfrac{l(l-1)}{2}$};
    \node at (axis cs: 0.85,0.37) {$\dfrac{l(l+1)}{2}$};
    \node at (axis cs: 0.35,0.87) {$\dfrac{l(l+1)}{2}$};
    \node at (axis cs: 0.85,0.87) {$\dfrac{l(l+1)}{2}$};
    \node at (axis cs: 1.1,0.1) {$t$};
    \node at (axis cs: 0.1,1.1) {$s$};
\end{axis}
\end{tikzpicture}
\end{subfigure}
\caption{The number of Lam\'e equations with unitary monodromy with given parameters $s, t \pmod 1$, and with odd $n = 2l - 1$ (left) or even $n = 2l$ (right).}
\label{Figure:Number_ansatz_sol}
\end{figure}

\begin{proof} 
Note that if $N=3$ and $3|(n-1)$, then there is a unique spherical torus fixed by the $\mathbb Z/3$-action on the labels, thus $3P\widetilde{L}_n(N)-2\epsilon(n,N)$ counts the number of Lam\'e equations with parameter 
\[
(2s, 2t) = (\dfrac{k_1}{N}, \dfrac{k_2}{N}), \quad 0 < k_1 < N - k_2 < N, \quad \mbox{$\gcd (k_1, k_2, N) = 1$ and $N\geq 3$}.
\] 

For convenience we may also assume that this holds for $N=1, 2$. By counting the lattice points in Figure~\ref{Figure:Number_deve_maps} we have
\[
\sum_{d|N}(3 P\widetilde{L}_n(d) - 2\epsilon(n, d)) = \dfrac{n(n + 1)}{2}(N^2 - 3N + 2).
\]
The formula (\ref{thm_alg_dahmen}) then follows from the M\"obius inversion.

Similarly we have that $3\widetilde{L}_n(N) - 2\epsilon(n,N)$ counts the number of Lam\'e equations with parameter 
\[
(s, t) = (\dfrac{k_1}{N}, \dfrac{k_2}{N}), \quad 0 < k_1 < N - k_2 < N, \quad  \mbox{$\gcd(k_1, k_2, N) = 1$ and $N\geq 3$},
\] 
and we may assume that this holds for $N=1, 2$ as well. 

By counting the lattice points in Figure~\ref{Figure:Number_ansatz_sol} we have
\allowdisplaybreaks
\begin{align*}&\sum_{d|N}(3\widetilde{L}_n(d)-2\epsilon(n,d))\\&=\begin{cases}
a_n\dfrac{3(l-1)(l-2)}{2}+(b_n-a_n)\dfrac{l(l-1)}{2} \qquad\textnormal{if }N=2l-1,\\
a_n\dfrac{3(l-1)(l-2)}{2}+(b_n-a_n)\dfrac{(l-1)(l-2)}{2} \qquad\textnormal{if }N=2l,\end{cases}
\\&=\begin{cases}
\dfrac{n(n+1)}{2}\dfrac{l(l-1)}{2}-3a_n l\qquad\textnormal{if }N=2l-1,\\
\dfrac{n(n+1)}{2}\dfrac{l^2}{2}-(2a_n+b_n)\dfrac{3l-2}{2}\qquad\textnormal{if }N=2l,
\end{cases}
\\&=\begin{cases}
\dfrac{n(n+1)}{16}(N^2-1)-\dfrac{3}{2} a_n (N+1) \qquad\textnormal{if }N=2l-1,\\
\dfrac{n(n+1)}{16}N^2-(2a_n+b_n)(\dfrac{3N}{4}-1) \qquad\textnormal{if }N=2l.\end{cases}
\end{align*}

Again the formula (\ref{thm_alg_dahmen_II}) follows from the M\"obius inversion. The proofs of both Theorem \ref{thm_alg_dahmen} and Theorem \ref{thm_alg_dahmen_II} are now complete. 
\end{proof}

\section{Dessin d'enfants and conformal structures} \label{s:dessin}

Our next goal is to find general algorithms such that, for given $n\in \mathbb{Q}$ and a monodromy group $M$ from the list in Theorem~\ref{Thm:B.-W.}, we may compute all possible values for $g_2$, $g_3$, and $B$. In other words, assume $B=0$ or $1$ after rescaling, compute all possible conformal structures $j(\tau)$, the $j$-invariant of $E_{\tau}$.

\subsection{Spherical tori and dessin d'enfants} \label{subsec:tori_dessin}

In this subsection, we first recall the correspondence between Belyi maps and dessins (or dessin d'enfants) for the reader's convenience. Then, we explain their connection with spherical tori via an example.

A Belyi map is a holomorphic map from a Riemann surface $X$ to $\mathbb{CP}^1$ ramified only at three points. By Belyi's theorem, this is equivalent to saying that $X$ can be defined over an algebraic number field. A dessin is a connected graph with bicolored vertices (i.e. the two end points of an edge are colored differently) equipped with a cyclic ordering of the edges around each vertex. 

Given a Belyi map $F: X \rightarrow \mathbb{CP}^1$ ramified at $\{0,1,\infty\}$, the preimage $F^{-1}([0,1])$ can be given a structure of dessin by placing black points at the preimage of $0$ and white points at the preimage of $1$. Similar construction can be given on $F^{-1}([1,\infty])$, and $F^{-1}([\infty,0])$. Conversely, a dessin determines the corresponding Belyi map uniquely up to isomorphism over $\overline{\mathbb{Q}}$.

This correspondence is useful to construct 2nd order ODEs with finite (projective) monodromy by Theorem~\ref{t:KBD}.
\begin{example} \label{eg:LameDessin}
Let $k\in\mathbb{N}$. By Theorem~\ref{t:KBD}, $L_{3/10+k,B}$ has finite projective monodromy ($PM \cong A_5$), if it can be obtained by the pullback of $H_{1/2,1/3,1/5}$. From \cite[Lemma~1.5]{Baldassarri_Dwork_1979}, the degree of the pullback map equals $60k+18$ and with the following ramification table:
\begin{table}[H]
\begin{tabular}{|l|l|l|l|l|l|l|}
\hline
 & \ $0 \in E$ & \\ \hline
0  &  & $30k+9$ points of ramification index 2\\  \hline
1 &  &  $20k+6$ points of ramification index 3\\ \hline
$\infty$ & $10k+8$ & $10k+2$ points of ramification index 5\\\hline
\end{tabular}
\caption{ramification table for $L_{3/10 +k,B}$}
\label{Table:L3/10+k}
\end{table}
Table~\ref{Table:L3/10+k} means that the pullback function has $30k+9$ points mapped to $0$ of multiplicity $2$,  $20k+6$ points mapped to $1$ of multiplicity $3$, the point $0 \in E$ mapped to $\infty$ of multiplicity $10k+8$, and $10k+2$ points mapped to $\infty$ of multiplicity $5$. We simply denote it as $[ 2^{30k+9}, 3^{20k+6}, (10k+8)^1 5^{10k+2}  ]$. 

The map ramifies at three points: 0, 1, and $\infty$, making it a Belyi map. One could construct the map by drawing the corresponding dessin. See Figure~\ref{Figure:Dessin13/10} for a complete list of 6 dessins when $k=1$.
\end{example}

\begin{remark}
In \cite{Kazarian_Zograf_2015}, given a genus $g$ and degree $d$, a recursive formula for the weighted sum of the number of dessins is given. For a more refined counting, given a ramification table, it is in general hard to enumerate the number of corresponding dessins, needless to say on the constructions for all of them. See \cite{Musty_Schiavone_Sijsling_Voight_2019} for the database of dessins up to genus 4 and degree 9. See also \cite{Shabat_2016} for a general discussion.


For ramification tables coming from Lam\'e equations, Table~\ref{Table:Counting_sphe_tri} gives a complete answer of the enumerative question. For the rest of this section, we further give a construction by relating spherical triangles with finite non-Klein-four monodromy groups to dessins. 
\end{remark}

Given a spherical triangle $\triangle$, construct the corresponding spherical torus $T(\triangle)$ following \cite[Construction~1.4]{Eremenko_Mondello_Panov_2023}. When $T(\triangle)$ has finite non-Klein-four monodromy, we superimpose the graph of a Platonic solid onto the triangle, and, consequently, onto the corresponding torus. Then we mark the points in $\V$ (vertices) as black points and points in $\E$ (midpoints of edges) as white points. This gives the dessin $D(\triangle)$ on $T(\triangle)$. We apply the double cover $\wp: E \rightarrow \mathbb{CP}^1$ to descent the dessin to $\mathbb{CP}^1$, which we denoted $\widetilde{D}(\triangle)$.

We give an explicit example for the above construction.
\begin{example} \label{eg:spherical_dessin}
For $n=3/10$, consider the following spherical triangle $\triangle$ in Figure~\ref{Figure:Spherical3/10}.
\begin{figure}[H]
\includegraphics[scale=0.09]{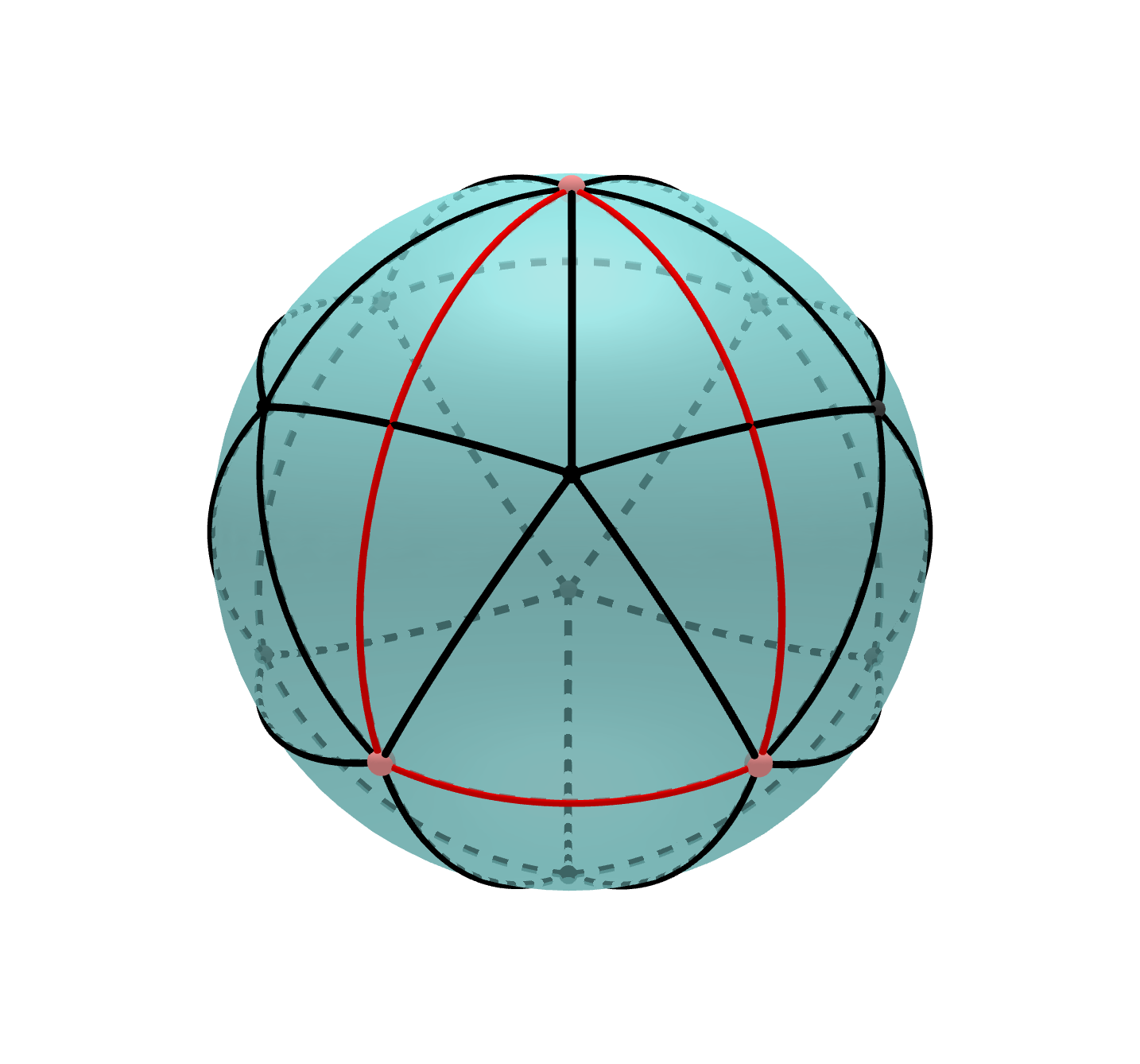} 
\caption{spherical triangle for $n=3/10$ with graph distance (1, 2, 2)}
\label{Figure:Spherical3/10}
\end{figure}
We obtain $D(\triangle)$ on E, and $\widetilde{D}(\triangle)$ on $\mathbb{CP}^1$ in Figure~\ref{Figure:Dessin3/10}.
\begin{figure}[H]
\begin{tikzpicture}[scale=0.8]
    \node[circle, fill, inner sep=2pt] (v1) at (0,0) {};
    \node[circle, fill, inner sep=2pt] (v2) at (4,0) {};
    \node[circle, fill, inner sep=2pt] (v3) at (2,4) {};
    \node[circle, fill, inner sep=2pt] (v4) at (-2,4) {};
    \node[circle, fill, inner sep=2pt] (v5) at (2,4/3) {};
    \node[circle, fill, inner sep=2pt] (v6) at (0,8/3) {};

    \node at (-0.28, 0.1) {\tiny $v_1$};
    \node at (4.2, 0.2) {\tiny $v_1$};
    \node at (2.2, 4.2) {\tiny $v_1$};
    \node at (-1.8, 4.2) {\tiny $v_1$};
    \node at (2.17, 1.67) {\tiny $v_2$};
    \node at (0, 3) {\tiny $v_3$};
    \node at (2,0.5) {\small $\frac{w_1}{2}$};
    \node at (1.45, 2.1) {\small $\frac{w_3}{2}$};
    \node at (3.4,2) {\small $\frac{w_2}{2}$};
    
    \node[circle, draw, inner sep=2pt] at (2,0) {};
    \node[circle, draw, inner sep=2pt] at (3,2) {};
    \node[circle, draw, inner sep=2pt] at (0,4) {};
    \node[circle, draw, inner sep=2pt] at (-1,2) {};
    \node[circle, draw, inner sep=2pt] at (1,2) {};
    \node[circle, draw, inner sep=2pt] at (1,0.65) {};
    \node[circle, draw, inner sep=2pt] at (3,0.65) {};
    \node[circle, draw, inner sep=2pt] at (0,1.35) {};
    \node[circle, draw, inner sep=2pt] at (-1,3.35) {};
    \node[circle, draw, inner sep=2pt] at (1,3.35) {};
    \node[circle, draw, inner sep=2pt] at (2,2.65) {};

    \draw[dashed, purple] (0,0.05) -- (4,0.05);
    \draw[dashed, purple] (v2) -- (v3);
    \draw[dashed, purple] (-2,4.05) -- (2,4.05);
    \draw[dashed, purple] (v1) -- (v4);
    \draw[dashed, purple] (v1) -- (v3);
    
    \draw[black] (v1) -- (v5);
    \draw[black] (v2) -- (v5);
    \draw[black] (v3) -- (v5);
    \draw[black] (v1) -- (v6);
    \draw[black] (v4) -- (v6);
    \draw[black] (v3) -- (v6);
    \draw[black] (v5) -- (v6);
    \draw[black] (v5) -- (3,2);
    \draw[black] (v6) -- (-1,2);
    \draw[black] (v1) -- (v2);
    \draw[black] (v3) -- (v4);

    \draw[->] (4.5,2) -- (5.5,2);
    \node at (5, 2.3) {$\wp$};

    \node[circle, fill, inner sep=2pt] (v7) at (7.4,1) {};
    \node[circle, fill, inner sep=2pt] (v8) at (10.4,2.5) {};

    \node[circle, draw, inner sep=2pt] (v9) at (8.15,3.25) {};
    \node[circle, draw, inner sep=2pt] (v10) at (8.9,1.75) {};
    \node[circle, draw, inner sep=2pt] (v11) at (9.65,0.25) {};
    \node[circle, draw, inner sep=2pt] (v12) at (8.4,1) {};
    \node[circle, draw, inner sep=2pt] (v13) at (9.4,2.5) {};
    \node[circle, draw, inner sep=2pt] (v14) at (10.9,3.366) {};

    \draw (8.9,1.75) circle ({ sqrt(11.25)/2 });

    \draw[black] (v7) -- (v8);
    \draw[black] (v7) -- (v12);
    \draw[black] (v8) -- (v13);
    \draw[black] (v8) -- (v14);

    \node at (8.8,1) {\small $e_1$};
    \node at (9,2.5) {\small $e_3$};
    \node at (11.3,3.366) {\small $e_2$};
    \node at (7,1) {\small $\infty$};
    \node at (12,2.5) {\small $\wp(v_2)=\wp(v_3)$};

\end{tikzpicture}
\caption{Dessins on $E$ (left) and on $\mathbb{P}^1$ (right)}
\label{Figure:Dessin3/10}
\end{figure}
The double cover $\wp$ maps $v_1$ to $\infty$, $v_2$ and $v_3$ to the same point, and half periods $w_i/2$ to $e_i$.

Figure~\ref{Figure:Spherical13/10} demonstrates the change of spherical triangle after attaching half sphere.

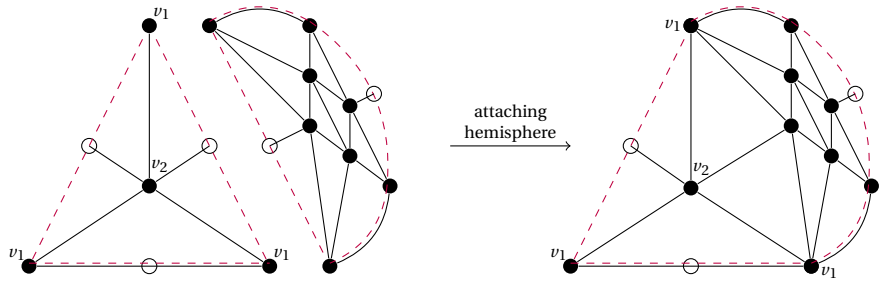
\begin{figure}[H]
\begin{tikzpicture}[scale=0.8]
    \node[circle, fill, inner sep=2pt] (v9) at (0,0) {};
    \node[circle, fill, inner sep=2pt] (v10) at (4,0) {};
    \node[circle, fill, inner sep=2pt] (v11) at (2,4) {};
    \node[circle, fill, inner sep=2pt] (v12) at (2,4/3) {};

    \node at (-0.2, 0.2) {\tiny $v_1$};
    \node at (4.2, 0.2) {\tiny $v_1$};
    \node at (2.2, 4.2) {\tiny $v_1$};
    \node at (2.17, 1.67) {\tiny $v_2$};
    
    \node[circle, draw, inner sep=2pt] at (2,0) {};
    \node[circle, draw, inner sep=2pt] at (3,2) {};
    \node[circle, draw, inner sep=2pt] at (1,2) {};

    \draw[dashed, purple] (0,0.05) -- (4,0.05);
    \draw[dashed, purple] (v10) -- (v11);
    \draw[dashed, purple] (v9) -- (v11);
    
    \draw[black] (v9) -- (v12);
    \draw[black] (v10) -- (v12);
    \draw[black] (v11) -- (v12);
    \draw[black] (v12) -- (3,2);
    \draw[black] (v9) -- (v10);
    \draw[black] (v12) -- (1,2);

    \node[circle, fill, inner sep=2pt] (v1) at (3,4) {};
    \node[circle, fill, inner sep=2pt] (v2) at (5,0) {};
    \node[circle, fill, inner sep=2pt] (v3) at (14/3,4) {};
    \node[circle, fill, inner sep=2pt] (v4) at (6,4/3) {};

    \node[circle, draw, inner sep=2pt] at (4,2) {};
    \node[circle, draw, inner sep=2pt] at (86/15,43/15) {};

    \draw[dashed, purple] (v1) -- (v2);
    \draw[dashed, purple] (3,4.07) to[bend left] (14/3,4.07);
    \draw[dashed, purple] (5.95,4/3) to[bend left] (5,0.05);
    \draw[dashed, purple] (14/3,4.07) to[bend left] (5.95,4/3);

    \node[circle, fill, inner sep=2pt] (v5) at (14/3,7/3) {};
    \node[circle, fill, inner sep=2pt] (v6) at (14/3,19/6) {};
    \node[circle, fill, inner sep=2pt] (v7) at (16/3,8/3) {};
    \node[circle, fill, inner sep=2pt] (v8) at (16/3,11/6) {};

    \draw (v1) to[bend left] (v3);
    \draw (v4) to[bend left] (v2);
    \draw (v1) -- (v5);
    \draw (v2) -- (v5);
    \draw (v3) -- (v6);
    \draw (v5) -- (v6);
    \draw (v5) -- (v8);
    \draw (v6) -- (v8);
    \draw (v6) -- (v7);
    \draw (v7) -- (v8);
    \draw (v2) -- (v8);
    \draw (v4) -- (v8);
    \draw (v1) -- (v6);
    \draw (v3) -- (v7);
    \draw (v4) -- (v7);
    \draw (v5) -- (4,2);
    \draw (v7) -- (86/15,43/15);

    \draw[->] (7,2) -- (9,2);
    \node at (8, 2.2) {\tiny hemisphere};
    \node at (8, 2.55) {\tiny attaching};
    
    \node[circle, fill, inner sep=2pt] (v29) at (9,0) {};
    \node[circle, fill, inner sep=2pt] (v30) at (13,0) {};
    \node[circle, fill, inner sep=2pt] (v31) at (11,4) {};
    \node[circle, fill, inner sep=2pt] (v32) at (11,1.3) {};

    \node at (8.8, 0.2) {\tiny $v_1$};
    \node at (13.3, -0.1) {\tiny $v_1$};
    \node at (10.7, 4) {\tiny $v_1$};
    \node at (11.17, 1.62) {\tiny $v_2$};
    
    \node[circle, draw, inner sep=2pt] at (10,2) {};
    \node[circle, draw, inner sep=2pt] at (11,0) {};

    \draw[dashed, purple] (9,0.05) -- (13,0.05);
    \draw[dashed, purple] (v29) -- (v31);
    
    \draw[black] (v29) -- (v32);
    \draw[black] (v30) -- (v32);
    \draw[black] (v31) -- (v32);
    \draw[black] (v29) -- (v30);
    \draw[black] (v32) -- (10,2);

        \node[circle, fill, inner sep=2pt] (v21) at (11,4) {};
        \node[circle, fill, inner sep=2pt] (v22) at (13,0) {};
        \node[circle, fill, inner sep=2pt] (v23) at (38/3,4) {};
        \node[circle, fill, inner sep=2pt] (v24) at (14,4/3) {};

        \node[circle, draw, inner sep=2pt] at (206/15,43/15) {};

        \draw[dashed, purple] (11,4.07) to[bend left] (38/3,4.07);
        \draw[dashed, purple] (13.95,4/3) to[bend left] (13,0.05);
        \draw[dashed, purple] (38/3,4.07) to[bend left] (13.95,4/3);

        \node[circle, fill, inner sep=2pt] (v25) at (38/3,7/3) {};
        \node[circle, fill, inner sep=2pt] (v26) at (38/3,19/6) {};
        \node[circle, fill, inner sep=2pt] (v27) at (40/3,8/3) {};
        \node[circle, fill, inner sep=2pt] (v28) at (40/3,11/6) {};

        \draw (v21) to[bend left] (v23);
        \draw (v24) to[bend left] (v22);
        \draw (v21) -- (v25);
        \draw (v22) -- (v25);
        \draw (v23) -- (v26);
        \draw (v25) -- (v26);
        \draw (v25) -- (v28);
        \draw (v26) -- (v28);
        \draw (v26) -- (v27);
        \draw (v27) -- (v28);
        \draw (v22) -- (v28);
        \draw (v24) -- (v28);
        \draw (v21) -- (v26);
        \draw (v23) -- (v27);
        \draw (v24) -- (v27);
        \draw (v27) -- (206/15,43/15);
        \draw (v32) -- (v25);
\end{tikzpicture}
\caption{spherical triangle for $n=13/10$}
\label{Figure:Spherical13/10}
\end{figure}
\quad\\
The dessin of the corresponding Platonic hemisphere is given in Figure~\ref{Figure:Hemisphere_Platonic}. For simplicity, we draw $\bullet - \bullet$ to represent $\bullet - \circ - \bullet$ if no confusion may occur.

With the same construction, we obtain $D(\triangle)$ on $E$, and $\widetilde{D}(\triangle)$ on $\mathbb{P}^1$ as in Figure~\ref{Figure:Dessin13/10}.
\begin{figure}[H]
\begin{tikzpicture}[scale=0.8]
    \node[circle, fill, inner sep=2pt] (v29) at (-1,-2) {};
    \node[circle, fill, inner sep=2pt] (v30) at (3,-2) {};
    \node[circle, fill, inner sep=2pt] (v31) at (1,2) {};
    \node[circle, fill, inner sep=2pt] (v32) at (1,-0.7) {};

    \node at (3.35, -2.1) {\tiny $v_1$};
    \node at (1.18, -0.37) {\tiny $v_2$};
    \node at (-0.7, -2.2) {\tiny $v_1$};
    \node at (0.8, 2.2) {\tiny $v_1$};
    \node at (-0.83, 0.35) {\tiny $v_3$};
    
    \node[circle, draw, inner sep=2pt] at (1,-2) {};
    \node[circle, draw, inner sep=2pt] at (0,0) {};
    \node at (0, -0.5) {\small $\frac{w_3}{2}$};
    \node at (1, -2.5) {\small $\frac{w_1}{2}$};
    \node at (4.1,1.1) {\small $\frac{w_2}{2}$};

    \draw[dashed, purple] (-1,-1.95) -- (3,-1.95);    
    \draw[black] (v29) -- (v32);
    \draw[black] (v30) -- (v32);
    \draw[black] (v31) -- (v32);
    \draw[black] (v29) -- (v30);
    \draw[black] (v32) -- (1,2);

    \node[circle, fill, inner sep=2pt] (v21) at (1,2) {};
    \node[circle, fill, inner sep=2pt] (v22) at (3,-2) {};
    \node[circle, fill, inner sep=2pt] (v23) at (8/3,2) {};
    \node[circle, fill, inner sep=2pt] (v24) at (4,-2/3) {};

    \node[circle, draw, inner sep=2pt] at (56/15,13/15) {};

    \draw[dashed, purple] (1,2.07) to[bend left] (8/3,2.07);
    \draw[dashed, purple] (3.95,-2/3) to[bend left] (3,-1.95);
    \draw[dashed, purple] (8/3,2.07) to[bend left] (3.95,-2/3);

    \node[circle, fill, inner sep=2pt] (v25) at (8/3,1/3) {};
    \node[circle, fill, inner sep=2pt] (v26) at (8/3,7/6) {};
    \node[circle, fill, inner sep=2pt] (v27) at (10/3,2/3) {};
    \node[circle, fill, inner sep=2pt] (v28) at (10/3,-1/6) {};

    \draw (v21) to[bend left] (v23);
    \draw (v24) to[bend left] (v22);
    \draw (v21) -- (v25);
    \draw (v22) -- (v25);
    \draw (v23) -- (v26);
    \draw (v25) -- (v26);
    \draw (v25) -- (v28);
    \draw (v26) -- (v28);
    \draw (v26) -- (v27);
    \draw (v27) -- (v28);
    \draw (v22) -- (v28);
    \draw (v24) -- (v28);
    \draw (v21) -- (v26);
    \draw (v23) -- (v27);
    \draw (v24) -- (v27);
    \draw (v27) -- (56/15,13/15);
    \draw (v32) -- (v25);

    \node[circle, fill, inner sep=2pt] (v9) at (1,2) {};
    \node[circle, fill, inner sep=2pt] (v10) at (-3,2) {};
    \node[circle, fill, inner sep=2pt] (v11) at (-1,-2) {};
    \node[circle, fill, inner sep=2pt] (v12) at (-1,0.7) {};

    \node at (-3.3, 2.1) {\tiny $v_1$};
    
    \node[circle, draw, inner sep=2pt] at (-1,2) {};

    \draw[dashed, purple] (1,1.95) -- (-3,1.95);
    
    \draw[black] (v9) -- (v12);
    \draw[black] (v10) -- (v12);
    \draw[black] (v11) -- (v12);
    \draw[black] (v9) -- (v10);
    \draw[black] (v12) -- (1,2);

    \node[circle, fill, inner sep=2pt] (v1) at (-1,-2) {};
    \node[circle, fill, inner sep=2pt] (v2) at (-3,2) {};
    \node[circle, fill, inner sep=2pt] (v3) at (-8/3,-2) {};
    \node[circle, fill, inner sep=2pt] (v4) at (-4,2/3) {};

    \node[circle, draw, inner sep=2pt] at (-56/15,-13/15) {};

    \draw[dashed, purple] (-1,-2.07) to[bend left] (-8/3,-2.07);
    \draw[dashed, purple] (-3.95,2/3) to[bend left] (-3,1.95);
    \draw[dashed, purple] (-8/3,-2.07) to[bend left] (-3.95,2/3);

    \node[circle, fill, inner sep=2pt] (v5) at (-8/3,-1/3) {};
    \node[circle, fill, inner sep=2pt] (v6) at (-8/3,-7/6) {};
    \node[circle, fill, inner sep=2pt] (v7) at (-10/3,-2/3) {};
    \node[circle, fill, inner sep=2pt] (v8) at (-10/3,1/6) {};

    \draw (v1) to[bend left] (v3);
    \draw (v4) to[bend left] (v2);
    \draw (v1) -- (v5);
    \draw (v2) -- (v5);
    \draw (v3) -- (v6);
    \draw (v5) -- (v6);
    \draw (v5) -- (v8);
    \draw (v6) -- (v8);
    \draw (v6) -- (v7);
    \draw (v7) -- (v8);
    \draw (v2) -- (v8);
    \draw (v4) -- (v8);
    \draw (v1) -- (v6);
    \draw (v3) -- (v7);
    \draw (v4) -- (v7);
    \draw (v7) -- (-56/15,-13/15);
    \draw (v12) -- (v5);

    \draw (v32) -- (v12);

     \draw[->] (4.6,0) -- (5.6,0);
    \node at (5.1, 0.3) {$\wp$};

    \node[circle, fill, inner sep=2pt] (v31) at (7,-2) {};
    \node[circle, fill, inner sep=2pt] (v32) at (8.5,-0.5) {};
    \node[circle, fill, inner sep=2pt] (v33) at (9,0) {};
    \node[circle, fill, inner sep=2pt] (v34) at (10,1) {};
    \node[circle, fill, inner sep=2pt] (v35) at (11,2) {};
    \node[circle, fill, inner sep=2pt] (v39) at (10.609,-0.276) {};
    \node[circle, fill, inner sep=2pt] (v40) at (8.724,1.609) {};

    \node[circle, draw, inner sep=2pt] (v36) at (7.5,-1) {};
    \node[circle, draw, inner sep=2pt] (v37) at (8,-1.5) {};
    \node[circle, draw, inner sep=2pt] (v38) at (10.5,1.5) {};

     \draw (v31) -- (v33);
     \draw (v32) -- (v36);
     \draw (v31) -- (v37);
     \draw (v39) -- (v40);
     \draw (v33) -- (v39);
     \draw (v33) -- (v40);
     \draw (v34) -- (v39);
     \draw (v34) -- (v40);
     \draw (v34) -- (v38);

    \draw (7.75,-1.25) circle ({sqrt{18}/4});   
    \draw (8,-1) circle ({sqrt(2)});
    \draw (10.5,1.5) circle ({sqrt(2)/2}); 
    \draw (10.609,-0.276) arc ( -64.486:154.486:{sqrt(2)});
    \draw (8.724,1.609) arc ( 83.937:366.063:{sqrt(18)/2} );
    \draw ( 7,-2 ) arc ( 225:405:{sqrt(8)} );
    
    \node at (8.35,-1.5) {\small $e_1$};
    \node at (7.5,-0.7) {\small $e_3$};
    \node at (10.7,1.75) {\small $e_2$};
    \node at (6.5,-2) {\small $\infty$};
    \node at (8.3,0) {\tiny $\wp(v_2)$};
\end{tikzpicture}
\caption{Dessins on $E$ (left) and on $\mathbb{P}^1$ (right)}
\label{Figure:Dessin13/10}
\end{figure}
The double cover $\wp$ maps $v_1$ to $\infty$, $v_2$ and $v_3$ to the same point, and half periods $w_i/2$ to $e_i$. 
\end{example}

Given any spherical triangle with finite non-Klein-four monodromy, one obtains the corresponding dessins on $E$ and on $\mathbb{CP}^1$ following the process in Example~\ref{eg:spherical_dessin}.
The construction problem then reduces to the long standing question on computing the corresponding Belyi map with a given dessin.

\subsection{Conformal structures} \label{subsec:conformal}
Given a dessin, computing the corresponding Belyi map has high computational complexity. An explicit computation is reduced to a system of polynomial equations. For genus 0 case, many methods (including Gr\"obner basis method, complex analytic method, and $p$-adic method) have been developed. We explain the general idea of Gr\"ober basis and complex analytic methods below. The later one, though being approximative in nature, is more applicable when the degree is large. We refer to \cite{Sijsling_Voight_2014} for the detail.

Following the notation in \cite{Beffaras_2015}, consider a (genus 0) ramification table
\[
[ z_1^{d^0_1}\dots z_{n_0}^{d^0_{n_0}}, o_1^{d^1_1}\dots o_{n_1}^{d^1_{n_1}}, p_1^{d^{\infty}_1} \dots p_{n_{\infty}}^{d^{\infty}_{n_{\infty}}}   ]
\]
with 
\[
-2 = \sum_{i=1}^{n_0} (d^0_{n_i}-1) + \sum_{j=1}^{n_1} (d^1_{n_j}-1) + \sum_{k=1}^{n_{\infty}} (d^{\infty}_{n_{\infty}}-1)
\]
by Riemann-Hurwitz formula.

Up to a M\"obius transformation, we may assume $z_1=0$, $o_1=1$, and $p_1 = \infty$.
The corresponding Belyi map takes the form  
\[
F(z) = \lambda \frac{ \prod_{i=1}^{n_0} (z-z_i)^{d^0_i} }{ \prod_{j=2}^{n_{\infty}} (z-p_j)^{d^{\infty}_j} }
\]
with 
\[
F(z) -1 = \lambda \frac{ \prod_{i=1}^{n_1} (z-o_i)^{d^1_{n_i}} }{ \prod \prod_{j=2}^{n_{\infty}} (z-p_j)^{d^{\infty}_j} }.
\]
By eliminating $F$, we obtain a system of $n_0+n_1+n_{\infty}-2$ polynomial equations with $n_0+n_1+n_{\infty}-2$ unknowns. In principle, it can be computed via a Gr\"obner basis calculation.

Complex analytic methods for determining Belyi maps are approximative. The approach involves computing a high precision solution over $\mathbb{C}$, from which one can then construct an exact solution over $\overline{\mathbb{Q}}$. 

As an example, let $\mathfrak{C}^*$ be the set of triples $(Z,O,P)$, where $Z=\{z_i\}_{i=1}^{n_0}$, $O=\{o_i\}_{i=1}^{n_1}$, and $P=\{p_i\}_{i=1}^{n_{\infty}}$ are three finite sets in $\mathbb{CP}^1$ with $z_1=0$, $o_1=1$, and $p_1=\infty$. Given $\mathcal{C} \in \mathfrak{C}^*$, we define
\[
F_{\lambda, \mathcal{C}} := \lambda \frac{ \prod_{i=1}^{n_0} (z-z_i)^{d^0_i} }{ \prod_{j=2}^{n_{\infty}} (z-p_j)^{d^{\infty}_j} }.
\]
Consider the map $\Phi: \mathbb{C} \times \mathfrak{C}^* \rightarrow \mathbb{C}^{n_0+n_1+n_{\infty}-2}$ given by
\[
\begin{split}
&\Phi(\lambda,\mathcal{C}) :=
\\
&\quad ( F(o_1)-1, F'(o_1),\dots, F^{(d^1_1-1)}(o_1),F(o_2)-1,F'(o_2),\dots , F^{(d^1_{n_1}-1)}(o_{n_1})  ),
\end{split}
\]
where $F = F_{\lambda,\mathcal{C}}$.
The set of preimages of $(0,0,\dots,0)$ by $\Phi$ will be the Belyi maps. We simply apply the Newton's iteration 
\[
(\lambda_{n+1}, \mathcal{C}_{n+1}) := (\lambda_n, \mathcal{C}_n) - J_{\Phi}(\lambda_n,\mathcal{C}_n)^{-1} \Phi(\lambda_n, \mathcal{C}_n)
\]
with a suitable choice of initial approximation (c.f. Remark~\ref{Rmk:algorithm}). 

Given a high precision solution over $\mathbb{C}$, we may use the LLL lattice-reduction algorithm \cite{Lenstra_Lenstra_Lovasz_1982} to recognize the algebraic one. Finally, we need to verity the correctness of the resulting solution. In contrast to the construction, it does admit satisfactory methods \cite[\S 8]{Sijsling_Voight_2014}.

\begin{remark}\label{Rmk:algorithm}
As discussed in \cite{Beffaras_2015, Sijsling_Voight_2014}, random choices for initial approximations fail to converge. 
In practice, one can always obtain a suitable choice using the circle packing method.
More precisely, a dessin naturally gives rise to a triangulation and hence a circle packing. When the triangulation is iteratively hexagonally refined, Bowers and Stephenson \cite{Bowers_Stephenson_2004} proved that the resulting circle packing will converge to a correct solution. 

\end{remark}

Due to the high computational complexity involved in computing Belyi maps, our understanding of the structural correspondence between spherical triangles $\triangle$ and conformal structures of $D(\triangle)$ remains limited. Below we list some applications and an open question.

\begin{remark} \label{Rmk:specialtau}
When the dessin has $\mathbb{Z}/ 4\mathbb{Z}$ (resp. $\mathbb{Z}/ 6\mathbb{Z}$) automorphism, the resulting $j$-invariant equals $1728$ (resp. $0$). We list all corresponding cases except Klein-four from Table~\ref{Table:Sph_triangle}.
\begin{itemize}
    \item For $j=1728$, all cases are
    
    Cubical: $n=\frac{1}{6}$, (1,1,2); $n=\frac{5}{6}$, (1,1,2').
    \item For $j=0$, all cases are

    Octahedral: $n\in \{ \frac{1}{4}, \frac{7}{4} \}$, (1,1,1);

    Icosahedral: $n\in \{\frac{1}{10}, \frac{19}{10} \}$, (1,1,1); $n\in\{ \frac{7}{10}, \frac{13}{10} \}$, (2,2,2);

    Dihedral: $n=1$, $(k_1,k_2,k_3)=(1,1,1)$.
\end{itemize}
By attaching hemispheres which preserves the symmetry, we observe that 
\begin{itemize}
    \item for $n\in \{ \frac{1}{6}, \frac{5}{6} \} + 2\mathbb{N} $, there exists $L_{n,B}$ with finite monodromy and with $j=1728$. The construction is given by attaching same number of hemispheres on equal sides of the spherical triangles. 
    \item for $n\in \{1, \frac{1}{4}, \frac{7}{4}, \frac{1}{10}, \frac{7}{10}, \frac{13}{10}, \frac{19}{10}  \} + 3 \mathbb{N}$, there exists $L_{n,B}$ with finite monodromy and with $j=0$. Here we attach same number of hemispheres on all sides.
\end{itemize}
For the non-dihedral case, this observation is consistent with the construction given by Maier \cite[Proposition~3.4]{Maier_2004}.
\end{remark}

\begin{remark} \label{Rmk:tau}
The bijection between the triangle $\triangle$ and $\Omega_5$ is given in \cite[Theorem~1.2]{Chen_Kuo_Lin_Wang_2018}. Here $\triangle$ represents all possible monodromy parameters $s$, and $t$ (mod 1) of equation~\eqref{eqn_Lame_ellipic}. $\Omega_5$ be the set of tori such that equation~\eqref{eqn_Lame_ellipic} has unitary monodromy. See Figure~\ref{Figure:triangle_omega}.

\begin{figure}[h]
\hspace{10mm}
\adjincludegraphics[scale=0.15,trim={{.25\width} 0 {.25\width} 0}]{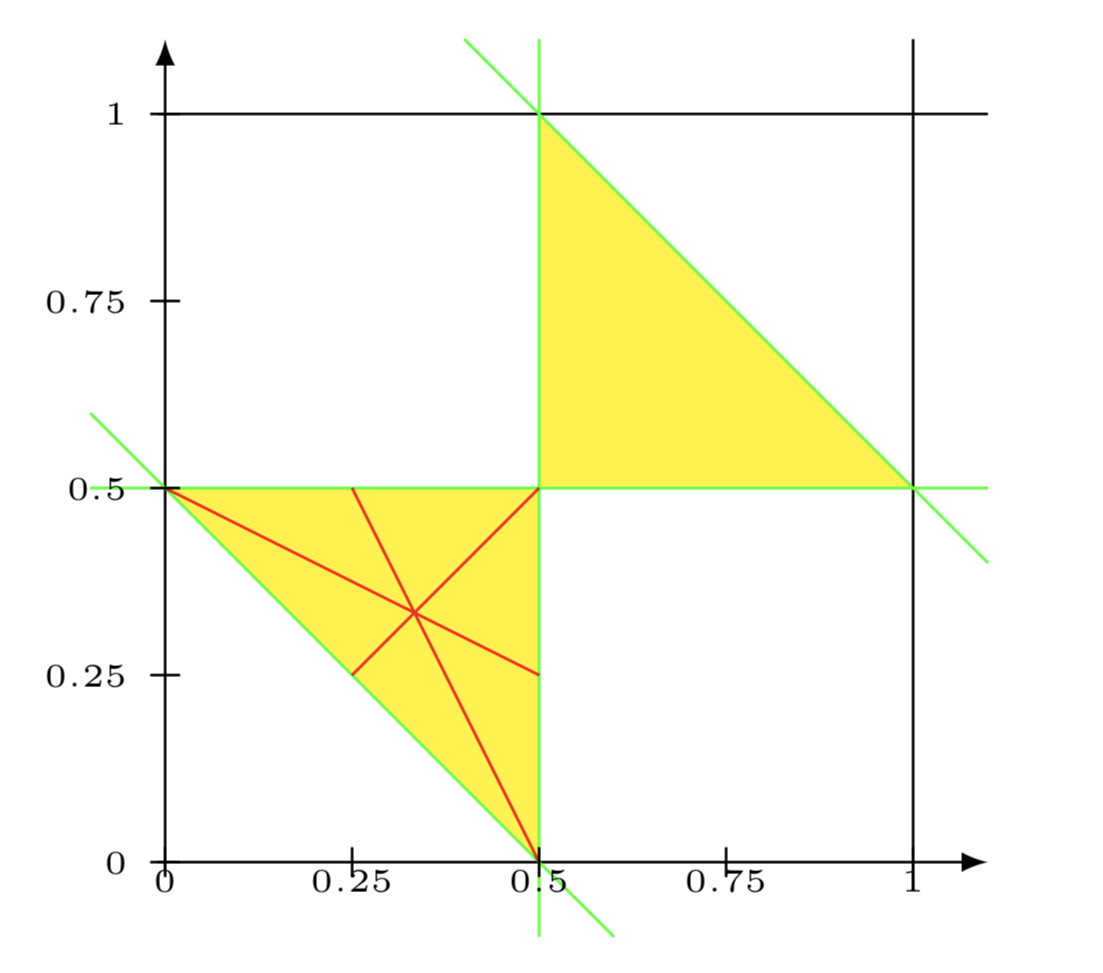}
\hfill
\adjincludegraphics[scale=0.16,trim={{.25\width} 0 {.25\width} 0}]{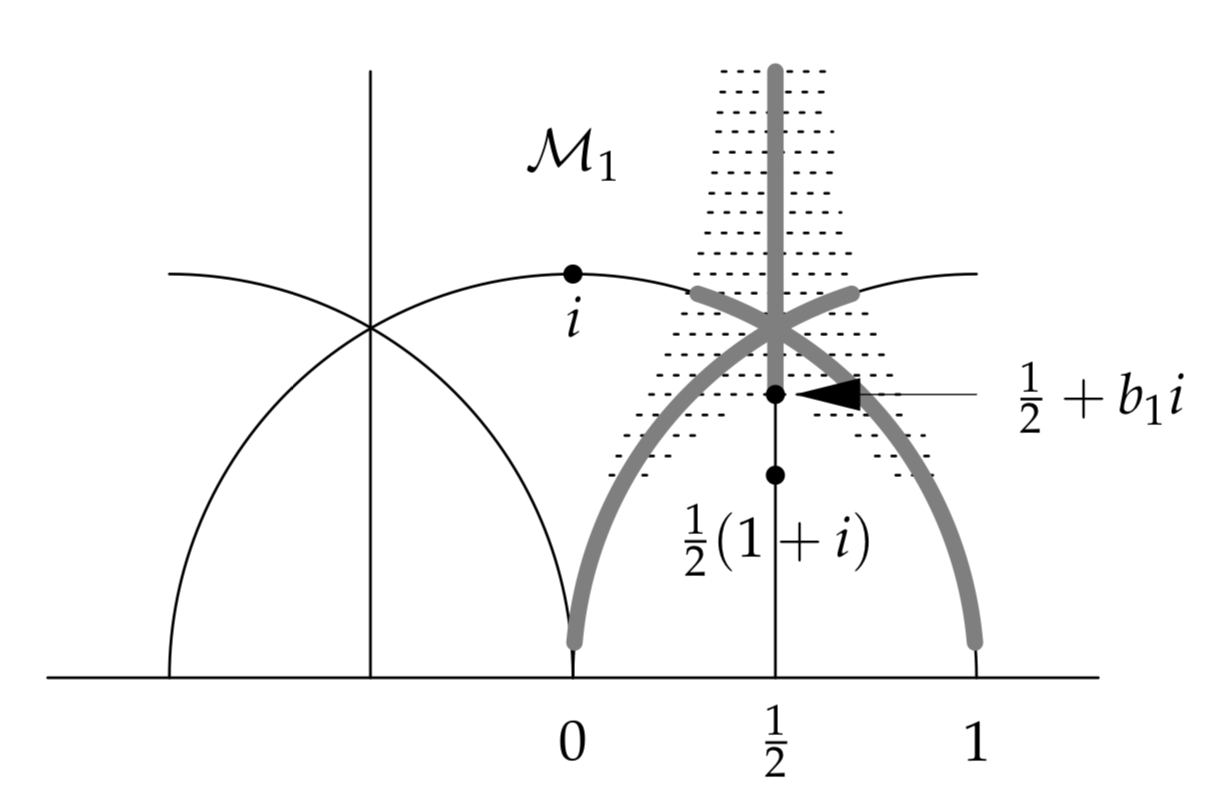}
\hspace{10mm}
\caption{(\cite[Figure~3, 4]{Chen_Kuo_Lin_Wang_2018}) The triangle $\triangle$ which is bijection to $\Omega_5$ (left), $\Omega_5$ contains a neighborhood of $e^{2\pi i /3}$ (right)}
\label{Figure:triangle_omega}
\end{figure}

When $L_{1,B}$ has finite monodromy group, the underlying torus $E_{\tau}$ is defined over $\overline{\mathbb{Q}}$, i.e. $j(\tau) \in \overline{\mathbb{Q}}$. The converse is not true. We claim that there exist infinitely many $\tau \in \Omega_5$ with $j(\tau) \in \overline{\mathbb{Q}}$ such that the corresponding $L_{1,B}$ has no finite monodromy. Note that the absolute Galois group, ${\rm Gal}(\overline{\mathbb{Q}}/\mathbb{Q})$, acts on dessins, and on the corresponding defining field, $\mathbb{Q}(j(\tau))$, of the torus. A Galois action preserves the ramification table and gives rise to a corresponding Lam\'e equation with finite monodromy. To find such a $\tau$ in the claim, it suffices to find $\tau\in \Omega_5$ and $\tau' \notin \Omega_5$ such that $j(\tau)$ and $j(\tau') $ are algebraic numbers and lie in the same Galois orbit. There are infinitely many, see Example~\ref{Eg:tau}.
\end{remark}
\begin{example} \label{Eg:tau} Note that $b_1 \approx 0.705$, and hence $j(\frac{1}{2} + b_1i) \approx 243.797$. Note also that $j$-invariant is monotone decreasing on 
\[
(+\infty i, i] \sqcup \widearc{i \ e^{2 \pi i /3}}   \sqcup [ e^{2\pi i /3}, \frac{1}{2} + \infty i ).
\]
Here $\widearc{ i \ e^{e^{2\pi i/3}} }$ is the arc starting from $i$ and moving clockwise to $e^{2\pi i /3}$ with center at the origin.

There exists infinitely many pairs $(p,q) \in \mathbb{Q}^2$ with $q$ not a rational perfect square such that $p+\sqrt{q} > j(\frac{1}{2}+b_1 i)$ and $p-\sqrt{q} \leq j(\frac{1}{2}+b_1i)$. The corresponding $\tau$ with $j(\tau) = p+\sqrt{q}$ will be what we claimed in Remark~\ref{Rmk:tau}.
\end{example}
We finish this section with a nature and widely open question. 
\medskip

\centerline{\emph{ Given $\tau$ with $j(\tau) \in \overline{\mathbb{Q}}$, classify all equations~(\ref{eqn_Lame_ellipic}) with finite monodromy.}}
\medskip
For the special cases when $j(\tau) = 0$ or $1728$, all equations $L_{n,B}$ with $n\in \{ \frac{1}{2} \} + \mathbb{Z}_{\geq 0}$ and finite projective monodromy ($\cong K_4$) are characterized by Theorem~\ref{t:BHC}. For $n\notin \frac{1}{2}+\mathbb{Z}_{\geq 0}$, a list is given in Remark~\ref{Rmk:specialtau}. Since the automorphism of a dessin is in general a subgroup of the automorphism of the resulting Riemann surface, it is not clear if we have exhausted all cases. 

\section{Appendix: an exhaustive list of basic spherical triangles}

The followings are graphics of the boundaries of the spherical triangles of finite monodromy in the exhaustive list (except dihedral). They are generated by the spherical geometry tool developed by Heather Pierce in Geogebra. \\\\

\noindent Octahedral:

\noindent
\hfill\adjincludegraphics[scale=0.08,trim={{.1\width} 0 {.1\width} 0}]{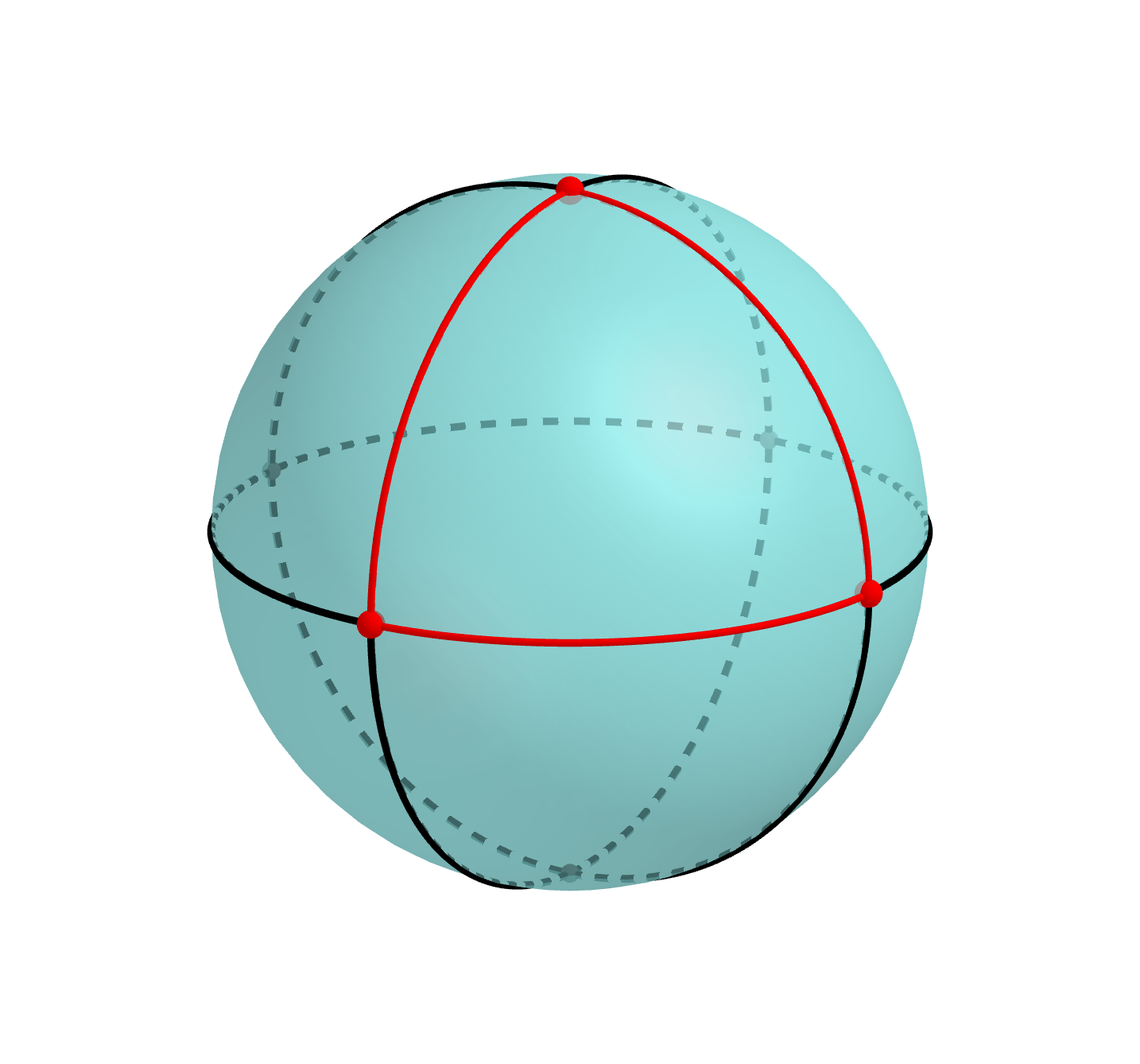}\hfill
\adjincludegraphics[scale=0.08,trim={{.1\width} 0 {.1\width} 0}]{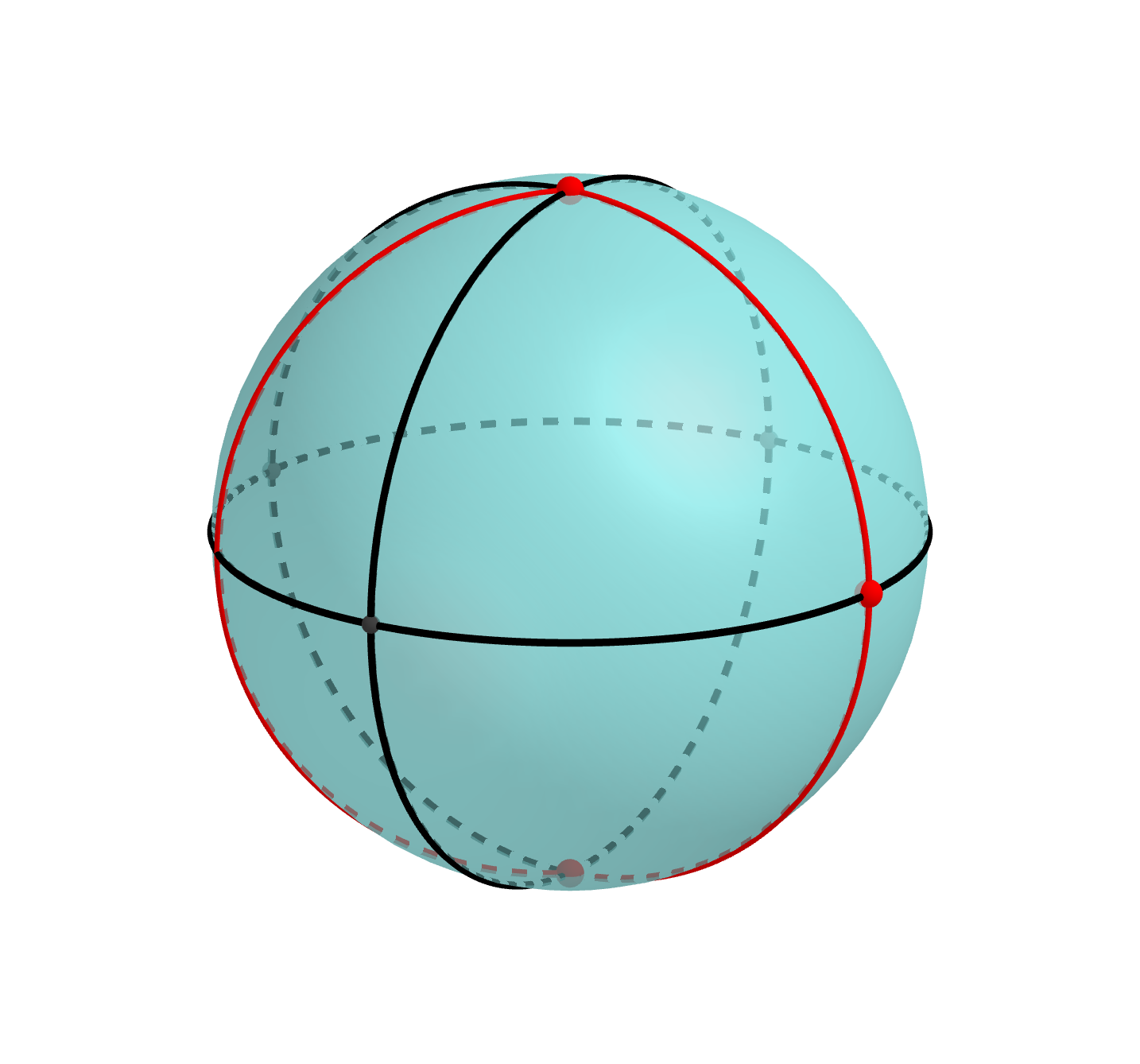} \hfill  \

\noindent
\hspace{25mm} $1\pm \frac{3}{4}$, (1,1,1) \hspace{32mm} $1\pm \frac{1}{4}$,  (1,1,2) 

\noindent Cubical:

\noindent
\adjincludegraphics[scale=0.08,trim={{.1\width} 0 {.1\width} 0}]{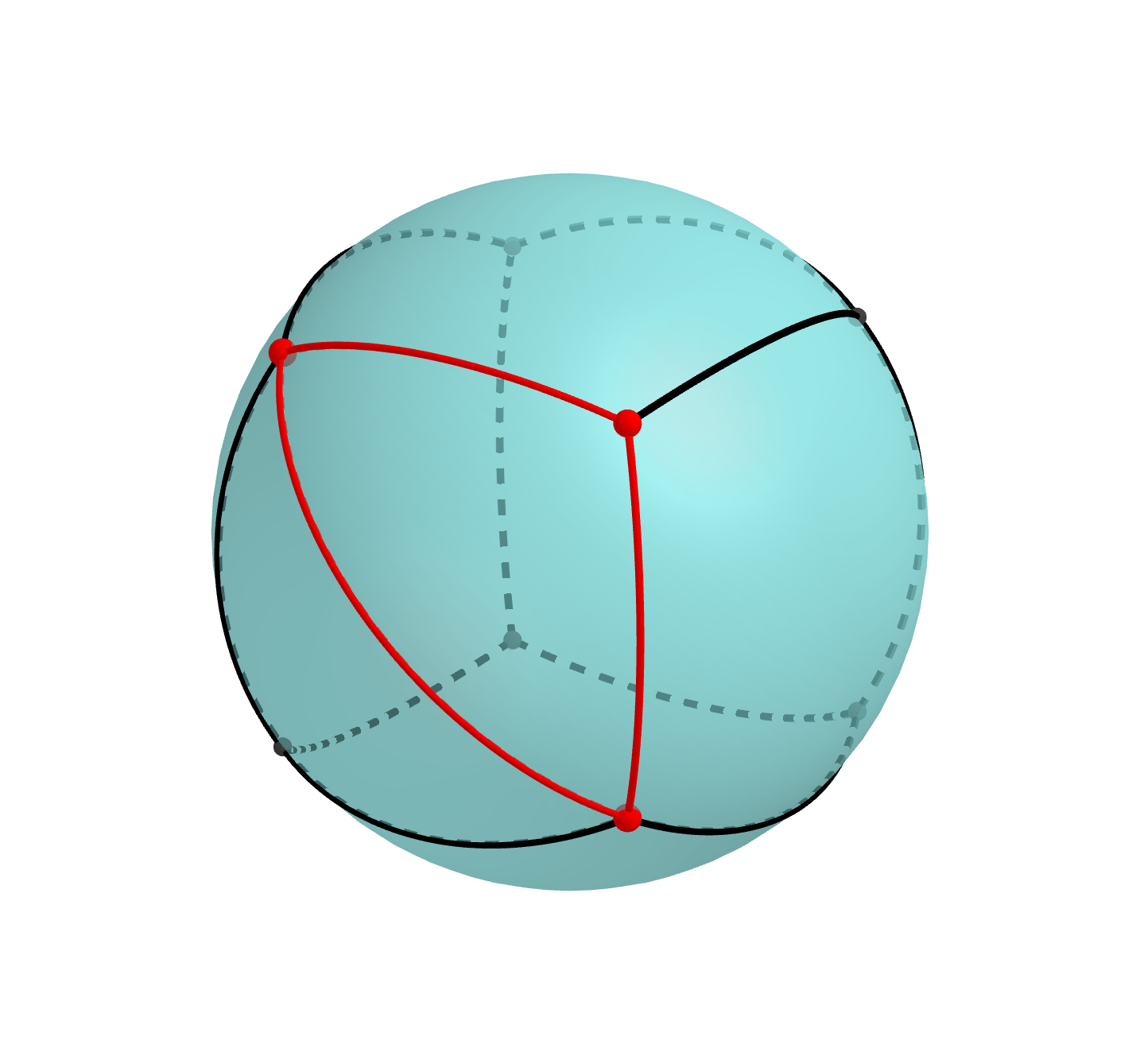}
\adjincludegraphics[scale=0.08,trim={{.1\width} 0 {.1\width} 0}]{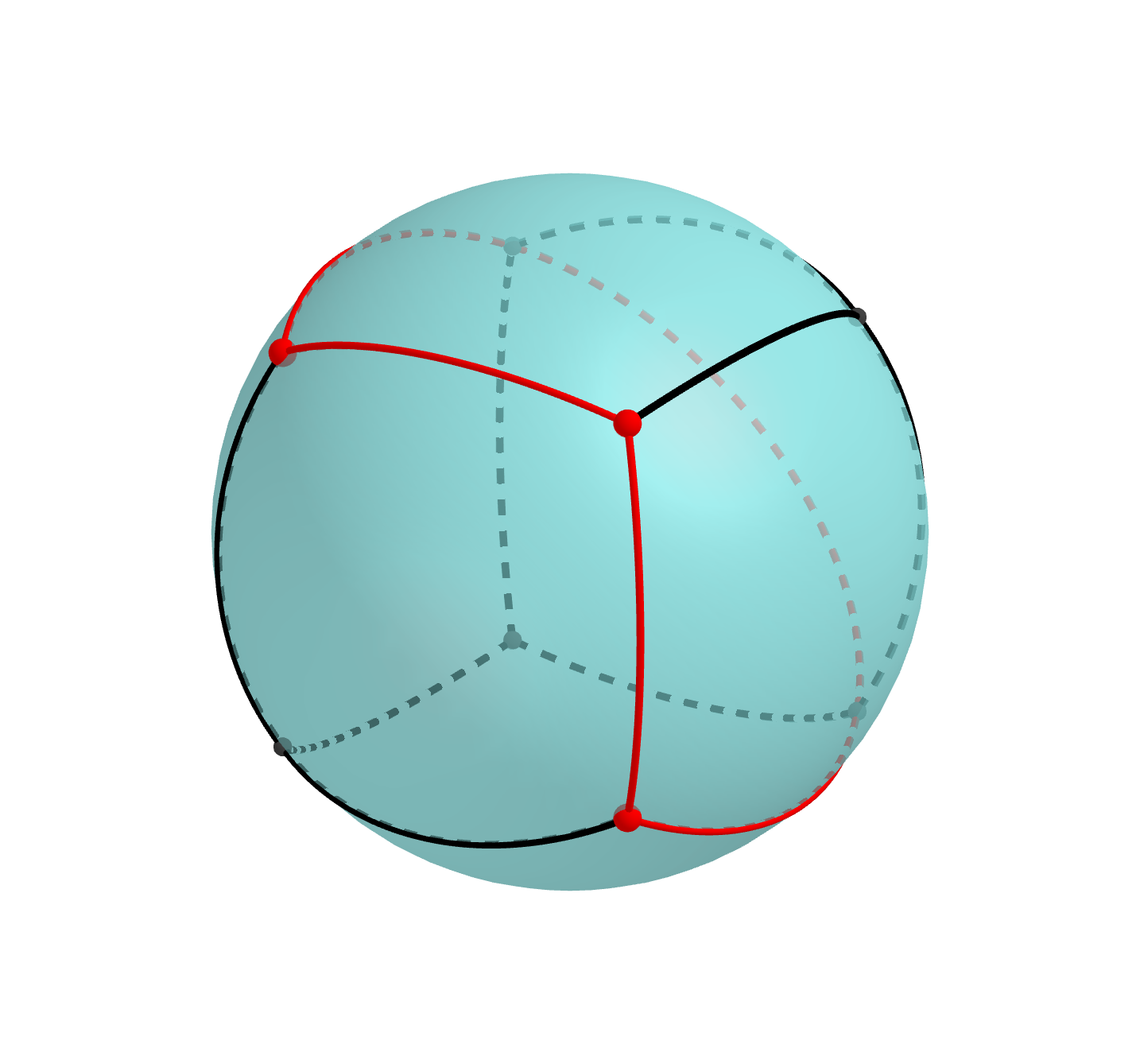}
\adjincludegraphics[scale=0.08,trim={{.1\width} 0 {.1\width} 0}]{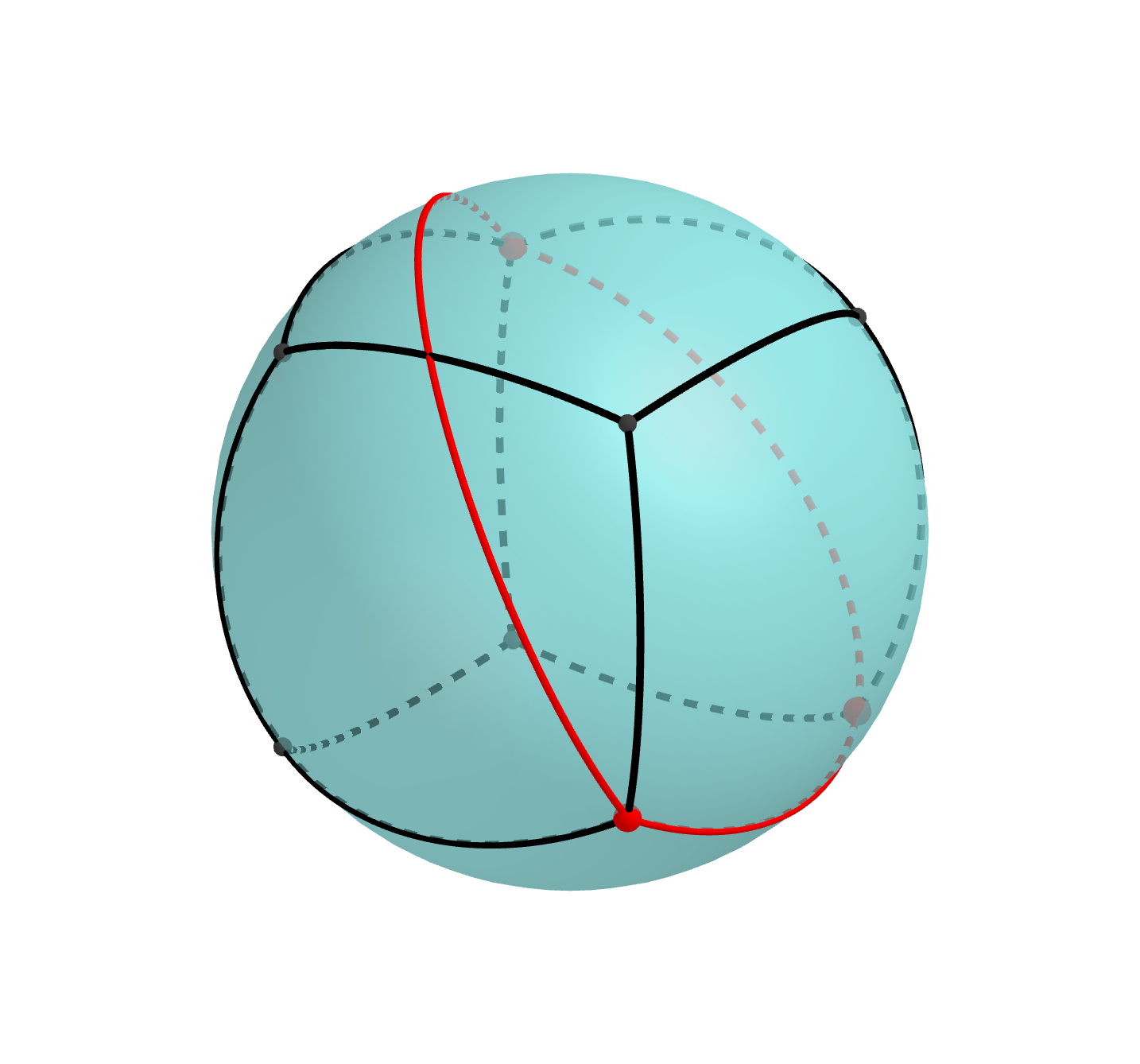}
\adjincludegraphics[scale=0.08,trim={{.1\width} 0 {.1\width} 0}]{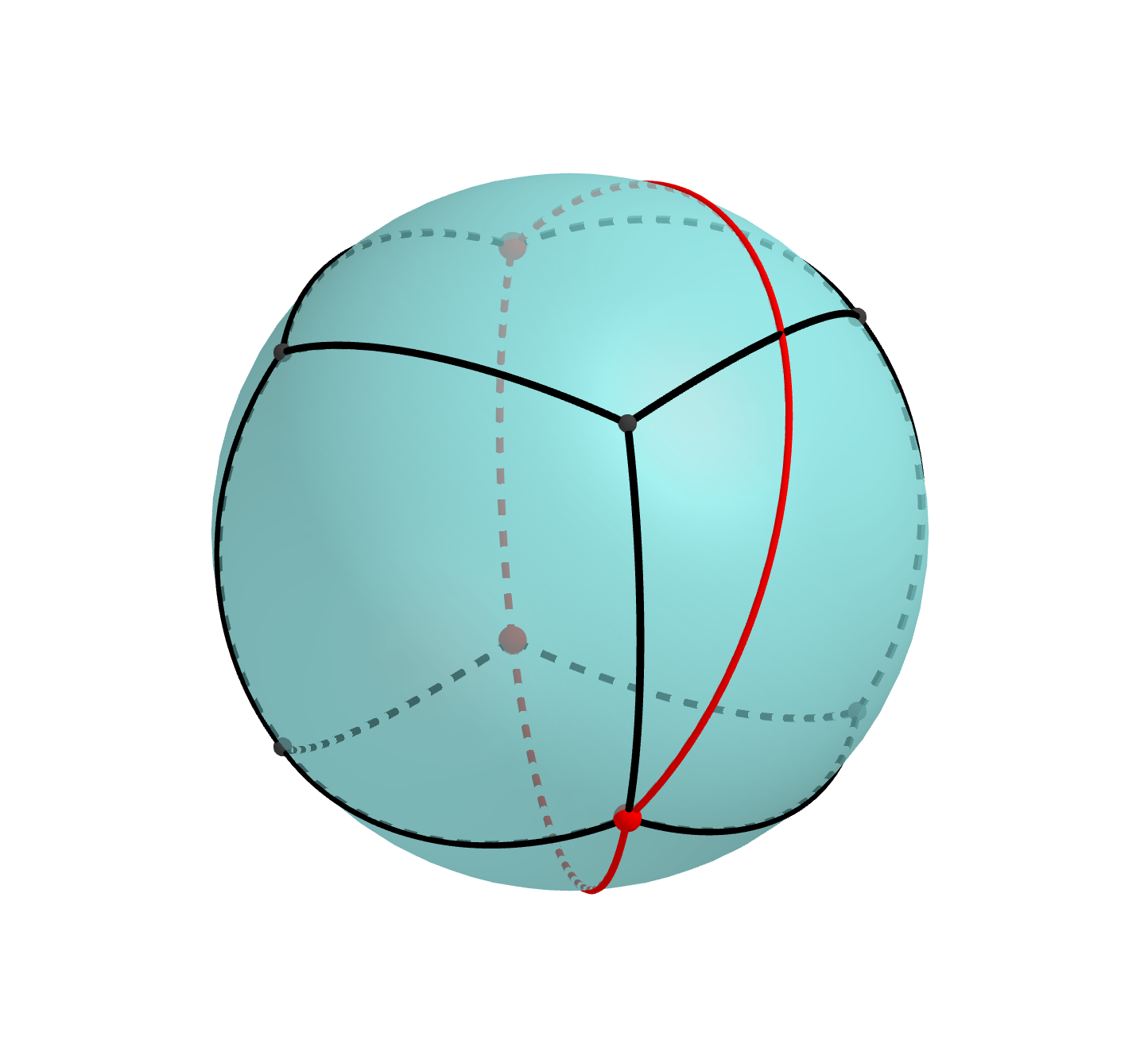}

\noindent
\hspace{7mm} $\frac{1}{6}$, (1,1,2) \hspace{16mm} $\frac{5}{6}$, (1,1,2') \hspace{14mm} $1\pm \frac{1}{6}$, (1,2,3) \hspace{12mm} $1\pm \frac{1}{6}$,  (1,3,2) \\\\

\noindent Icosahedral:

\noindent
\adjincludegraphics[scale=0.08,trim={{.1\width} 0 {.1\width} 0}]{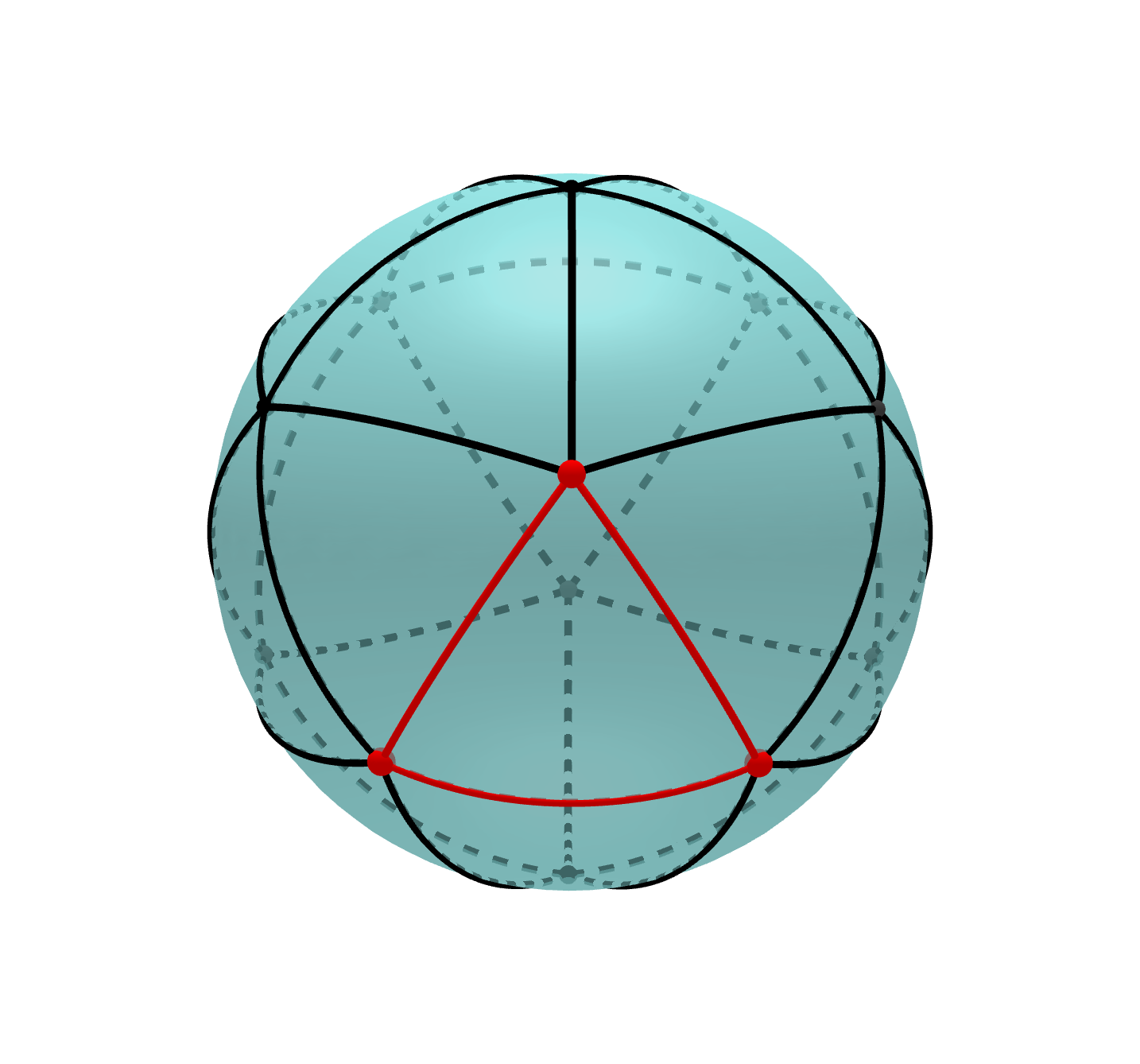}
\adjincludegraphics[scale=0.08,trim={{.1\width} 0 {.1\width} 0}]{spheres_big/icosa2_big.png}
\adjincludegraphics[scale=0.08,trim={{.1\width} 0 {.1\width} 0}]{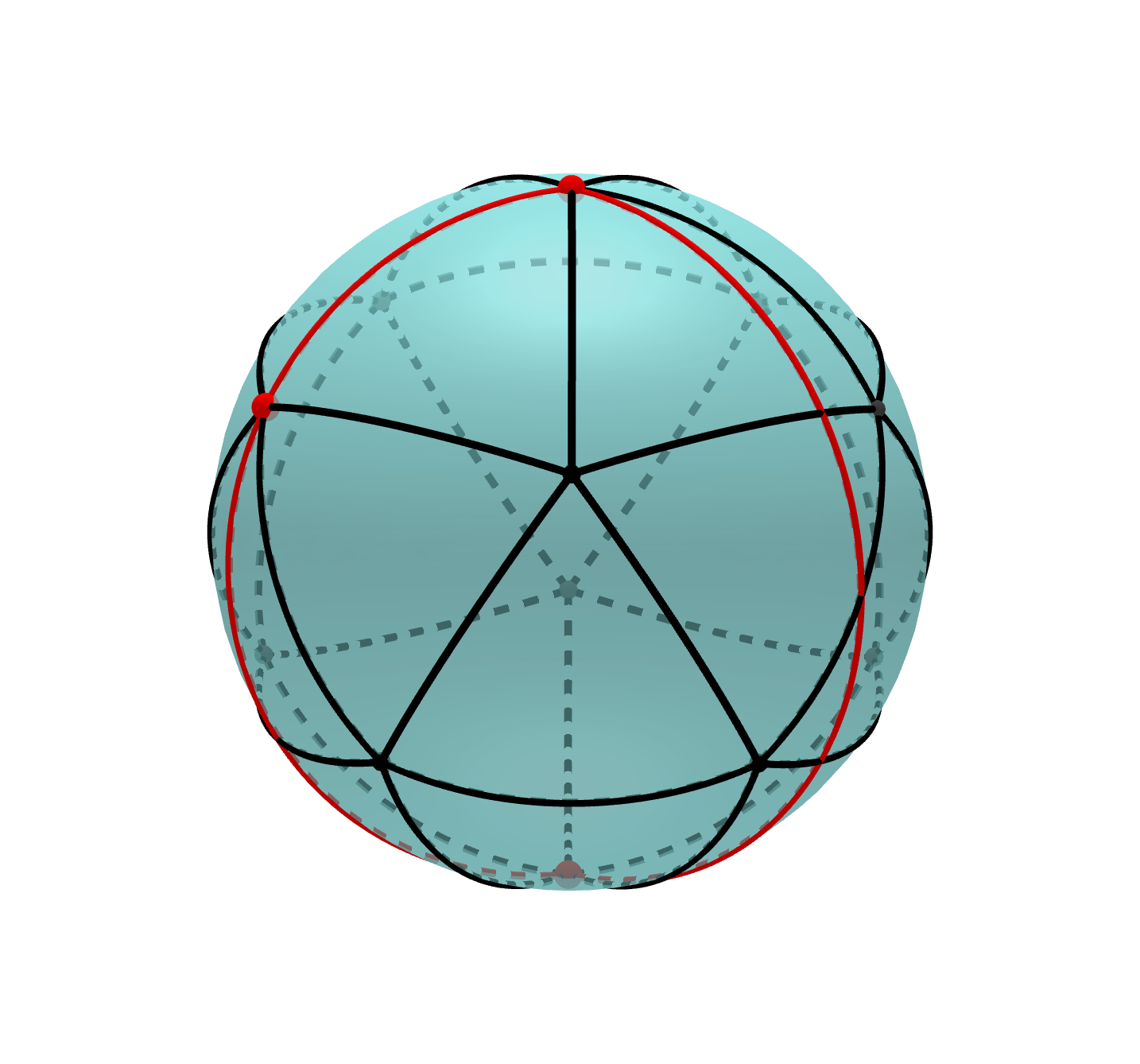}
\adjincludegraphics[scale=0.08,trim={{.1\width} 0 {.1\width} 0}]{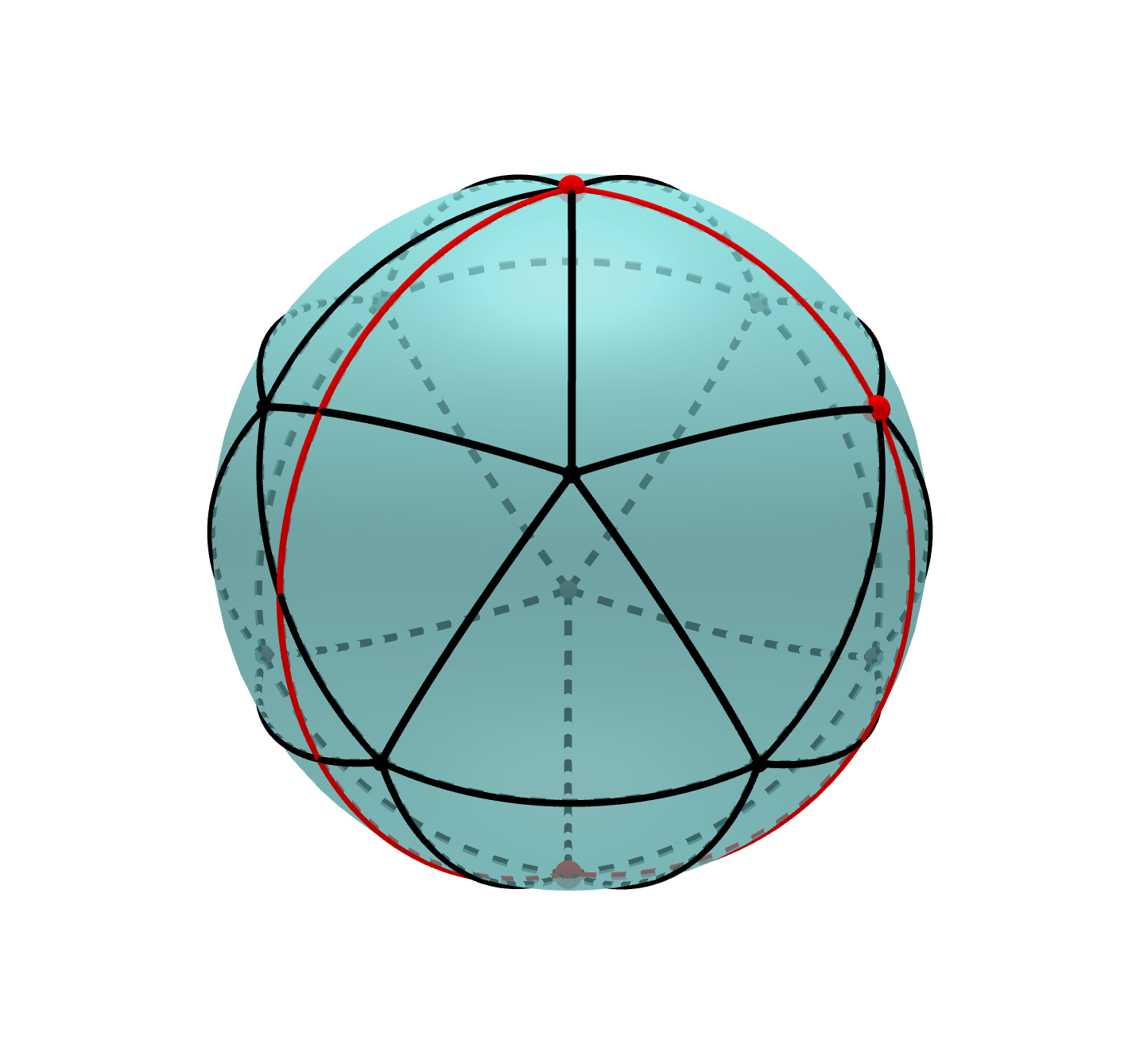}

\noindent
\hspace{4.5mm} $1 \pm \frac{9}{10}$, (1,1,1)\hspace{10mm} $1\pm \frac{7}{10}$,  (1,2,2)  \hspace{10mm} $1\pm \frac{3}{10}$, (1,2,3) \hspace{10mm} $1\pm \frac{3}{10}$, (1,3,2)

\noindent
\adjincludegraphics[scale=0.08,trim={{.1\width} 0 {.1\width} 0}]{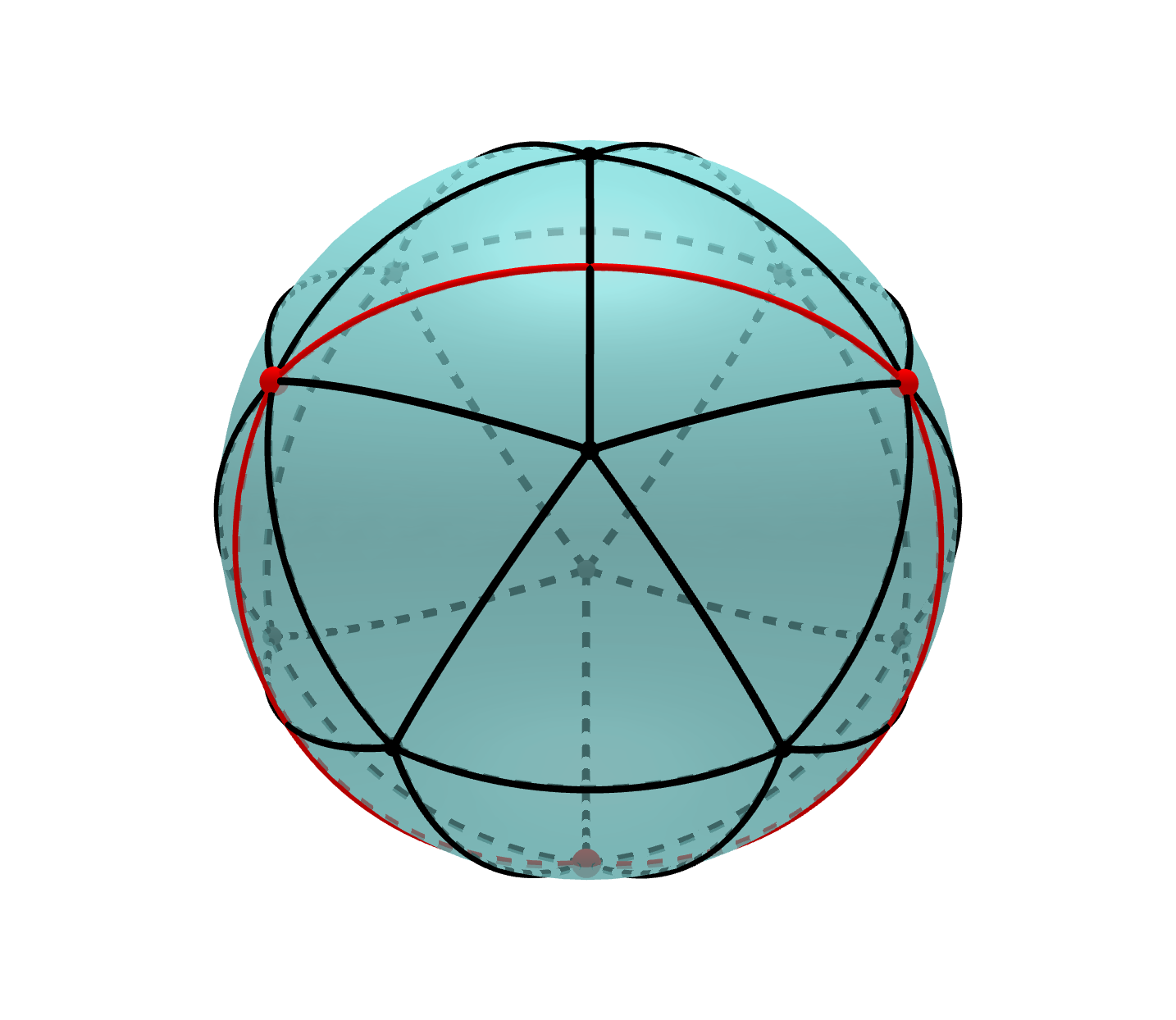}
\adjincludegraphics[scale=0.08,trim={{.1\width} 0 {.1\width} 0}]{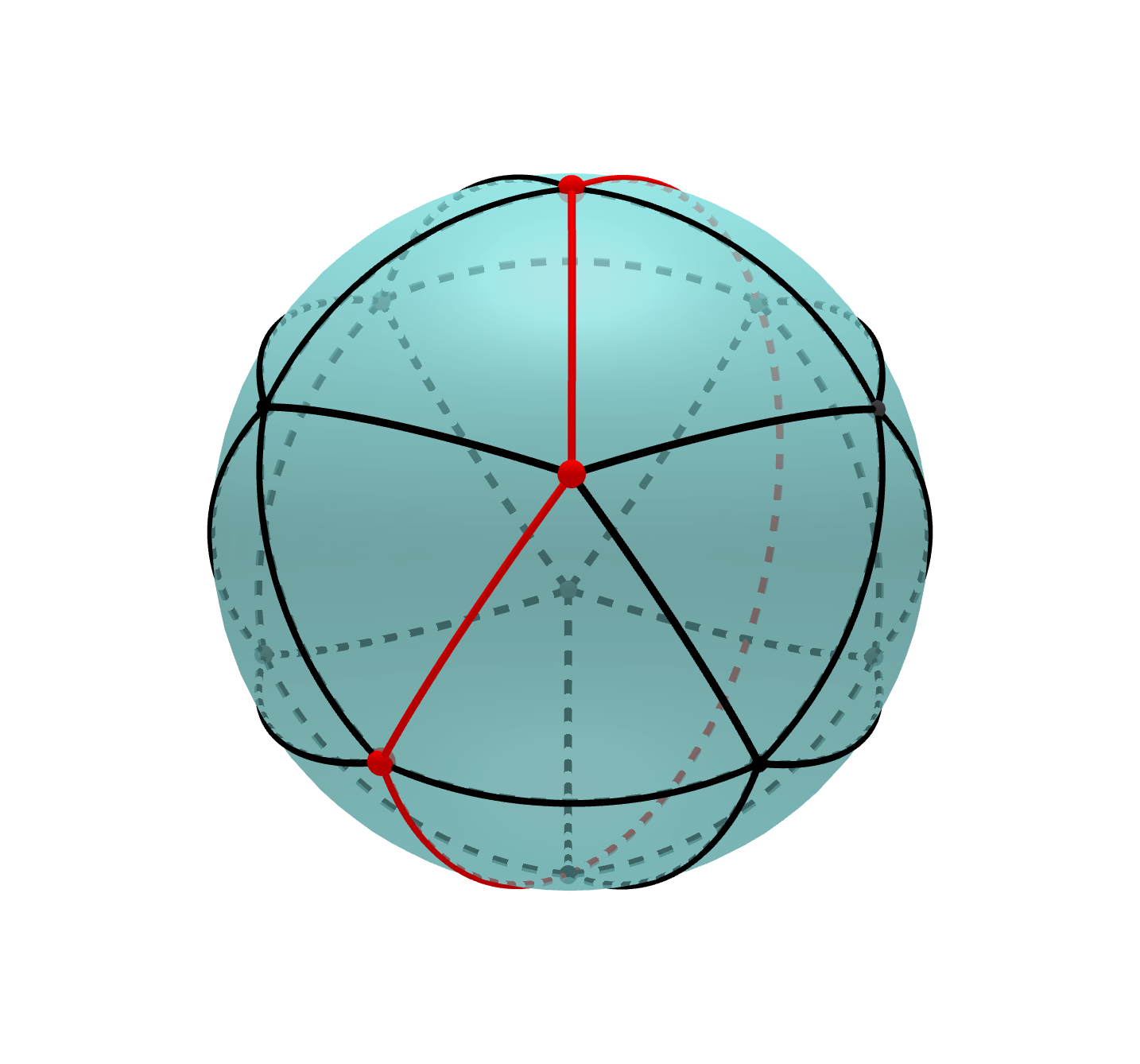}
\adjincludegraphics[scale=0.08,trim={{.1\width} 0 {.1\width} 0}]{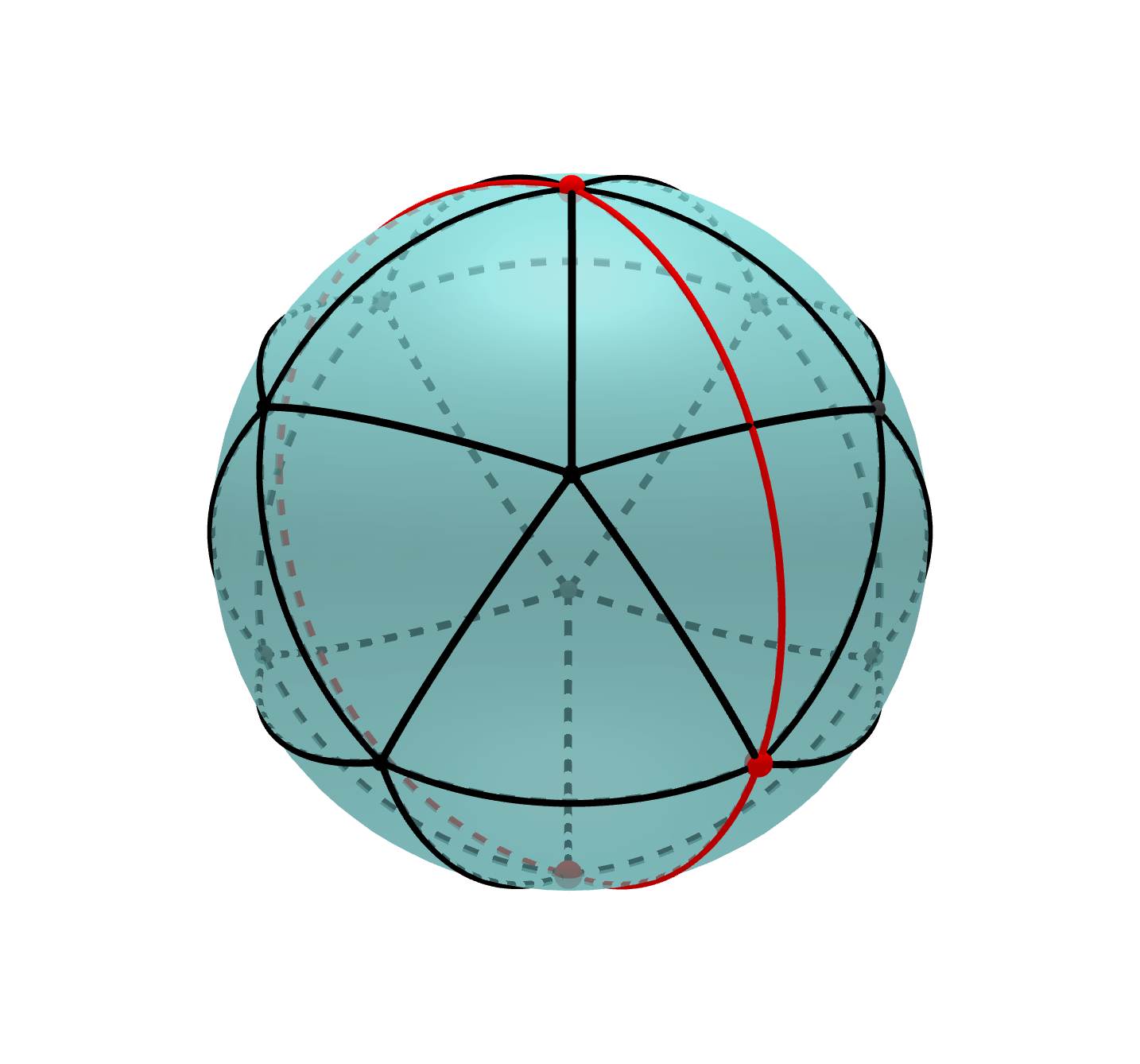}
\adjincludegraphics[scale=0.08,trim={{.1\width} 0 {.1\width} 0}]{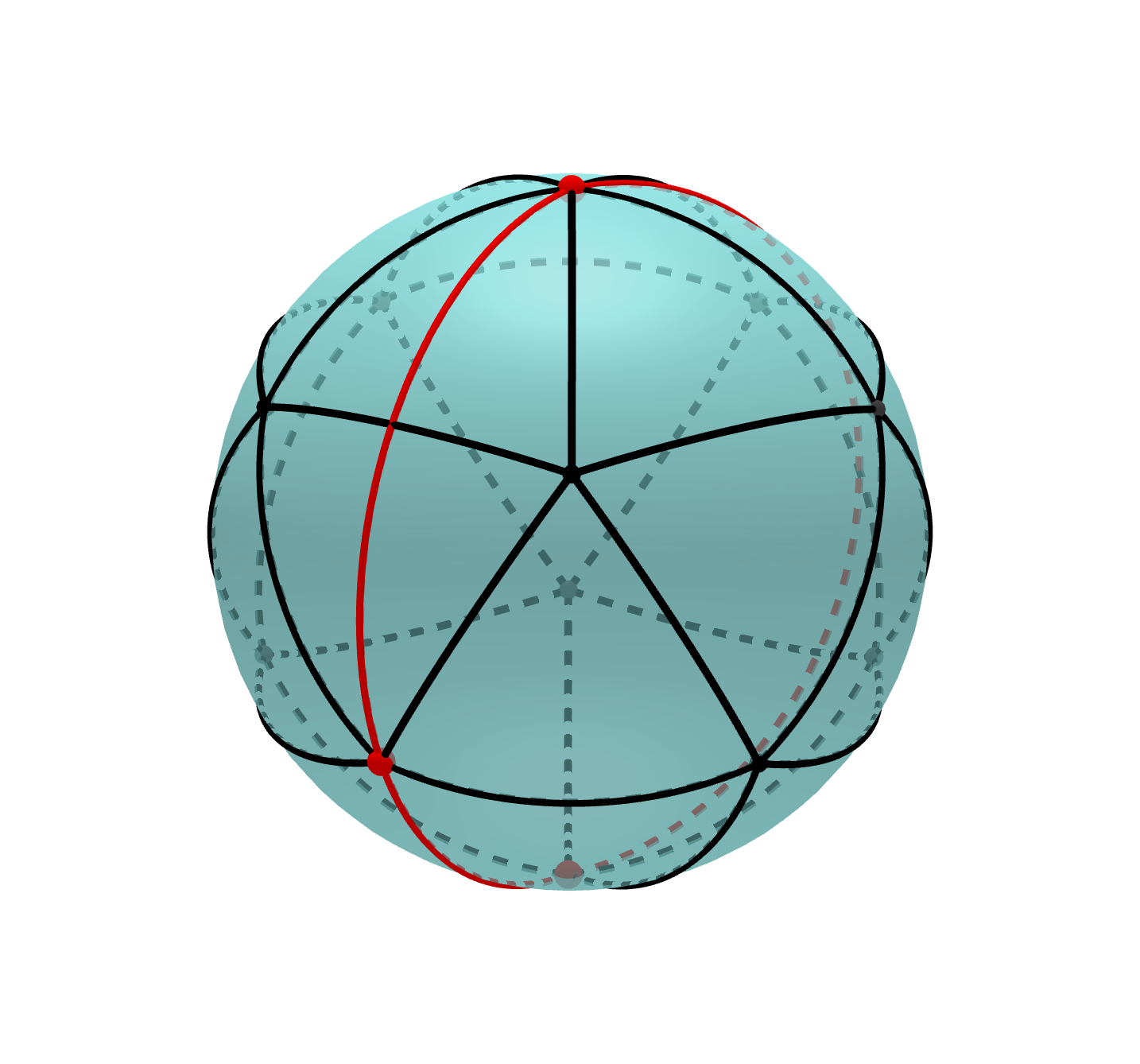}

\noindent
\hspace{4mm} $1\pm \frac{3}{10}$, (2,2,2)\hspace{10mm} $1\pm \frac{1}{10}$, (1,1,2')  \hspace{9.5mm} $1\pm \frac{1}{10}$, (1,2,3) \hspace{10mm} $1\pm \frac{3}{10}$, (1,3,2) \\\\

\noindent Dodecahedral:

\noindent
\hfill\adjincludegraphics[scale=0.08,trim={{.1\width} 0 {.1\width} 0}]{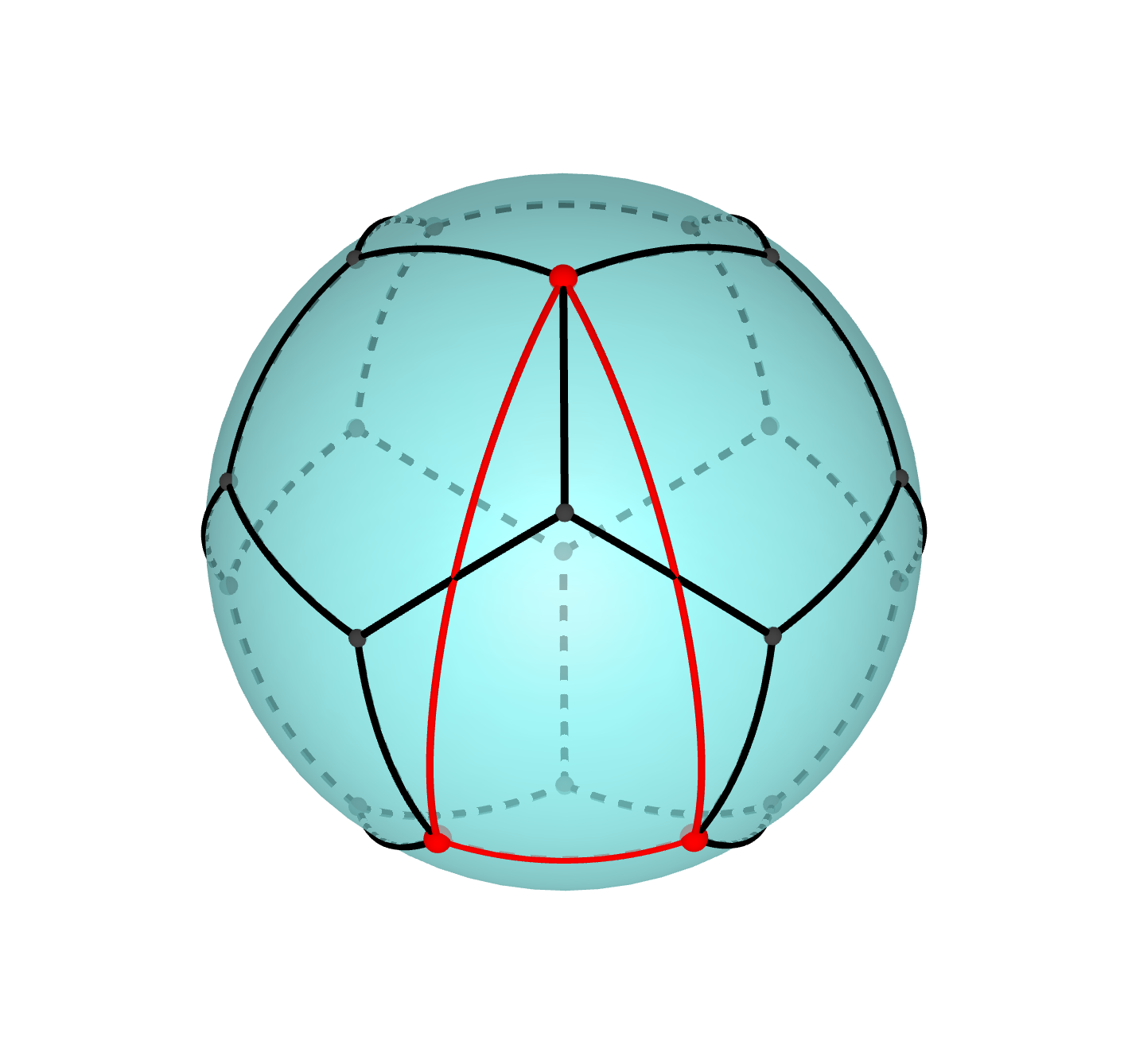}
\hfill\adjincludegraphics[scale=0.08,trim={{.1\width} 0 {.1\width} 0}]{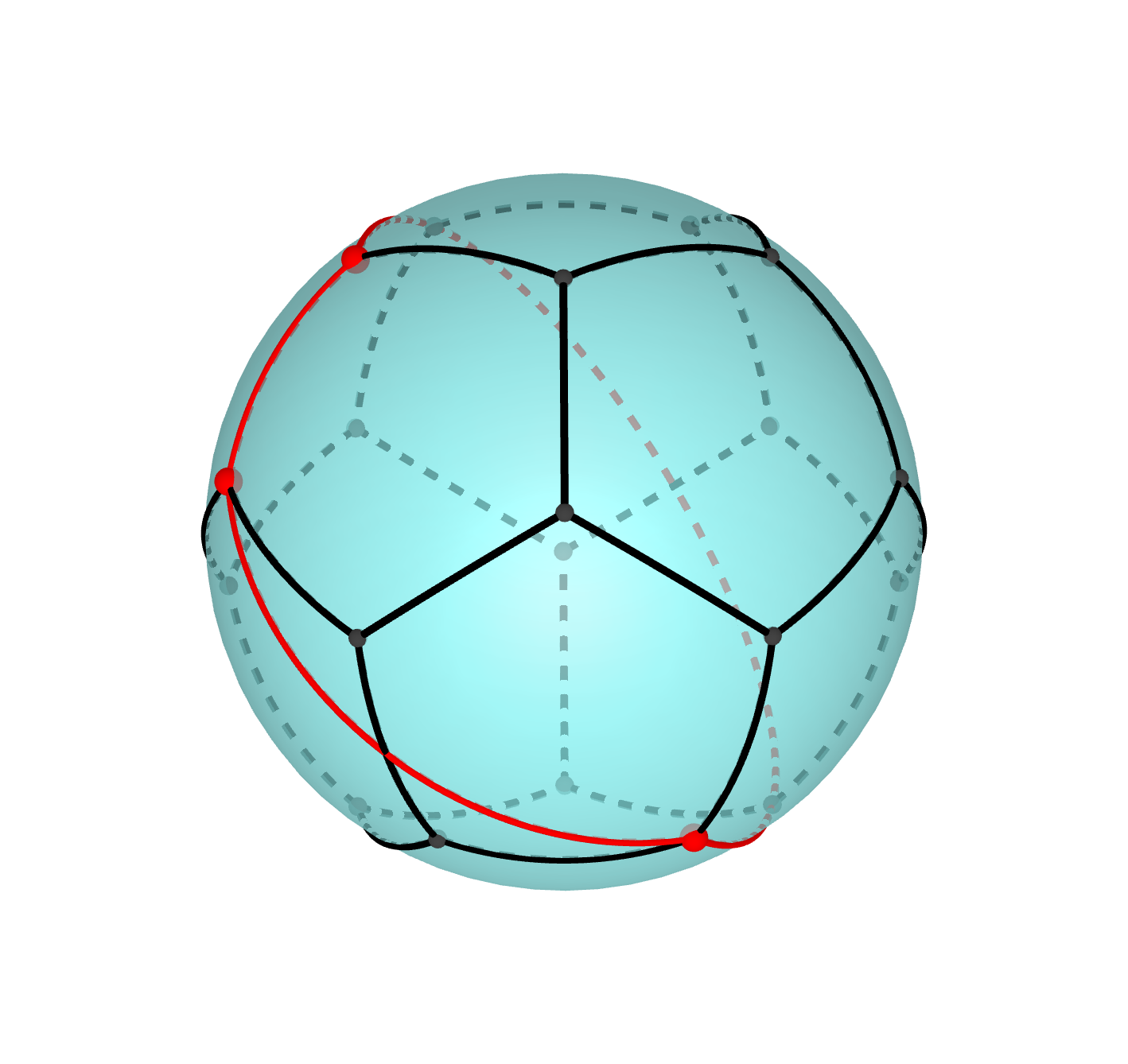}
\hfill\adjincludegraphics[scale=0.08,trim={{.1\width} 0 {.1\width} 0}]{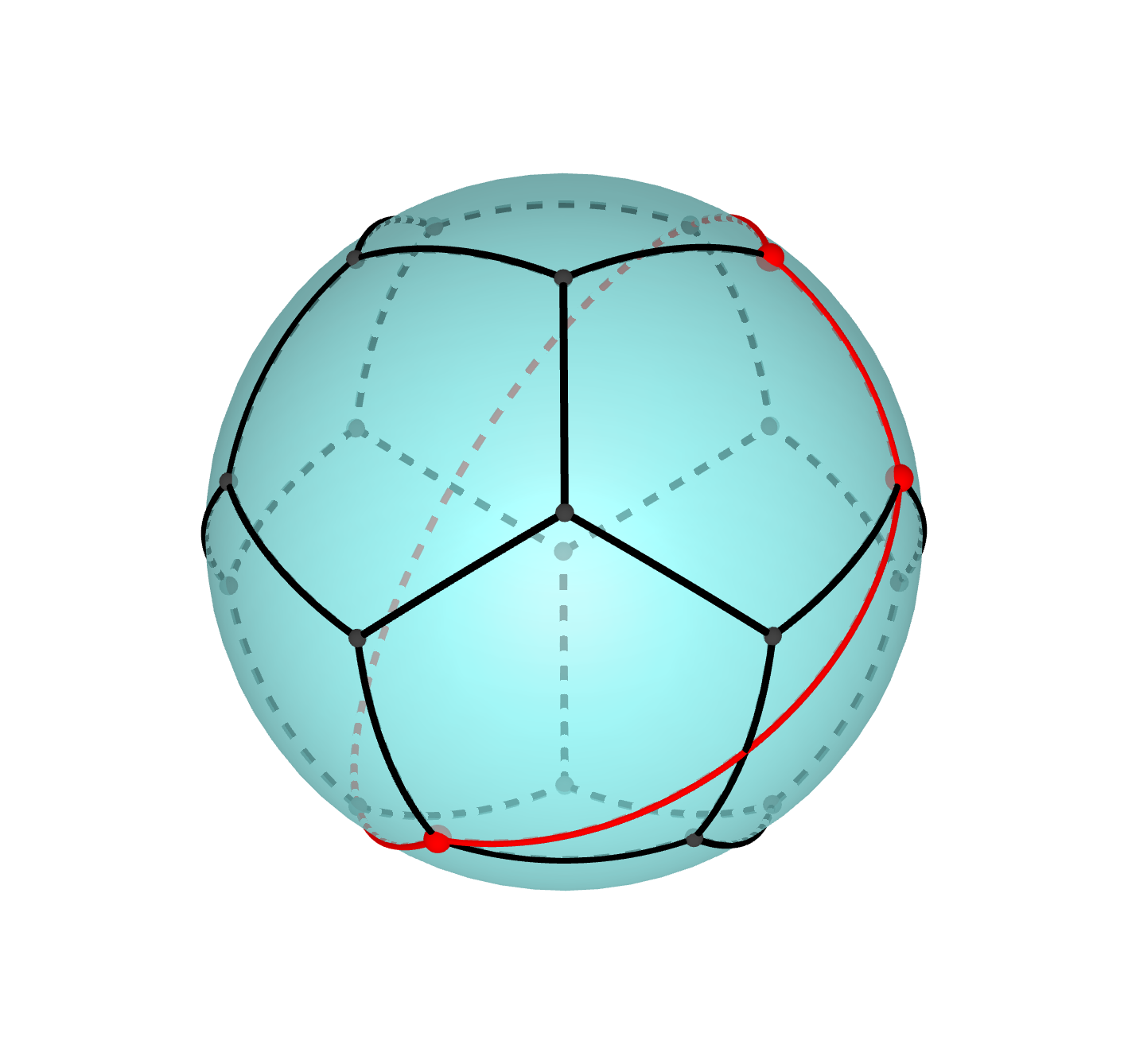}
\hfill \

\noindent 
\hspace{12mm} $1\pm \frac{5}{6}$, (1,3,3) \hspace{16mm} $1\pm \frac{1}{6}$, (1,3,4') \hspace{16mm} $1\pm \frac{1}{6}$, (1,4',3)

\noindent
\hfill\adjincludegraphics[scale=0.08,trim={{.1\width} 0 {.1\width} 0}]{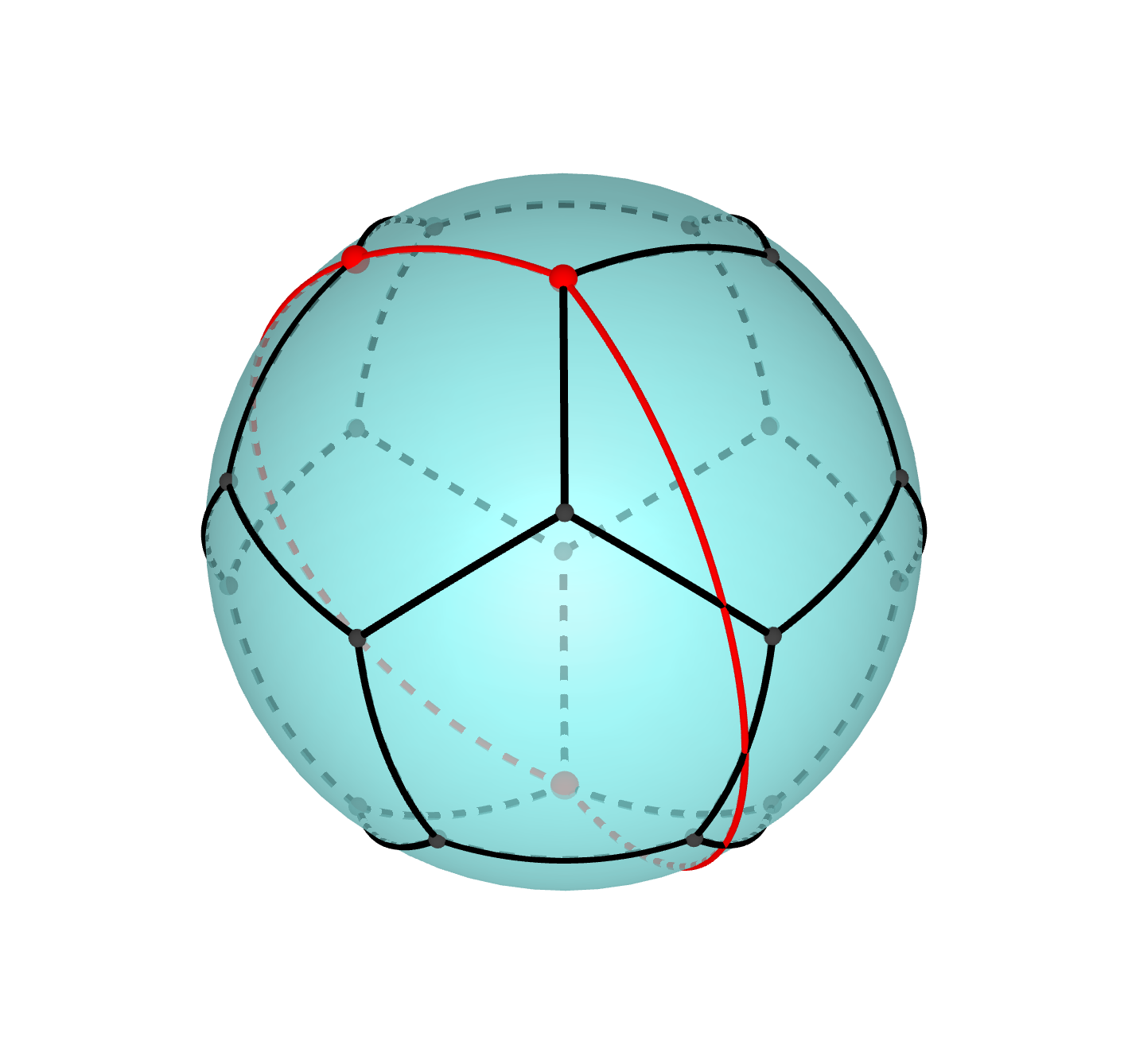}
\hfill\adjincludegraphics[scale=0.08,trim={{.1\width} 0 {.1\width} 0}]{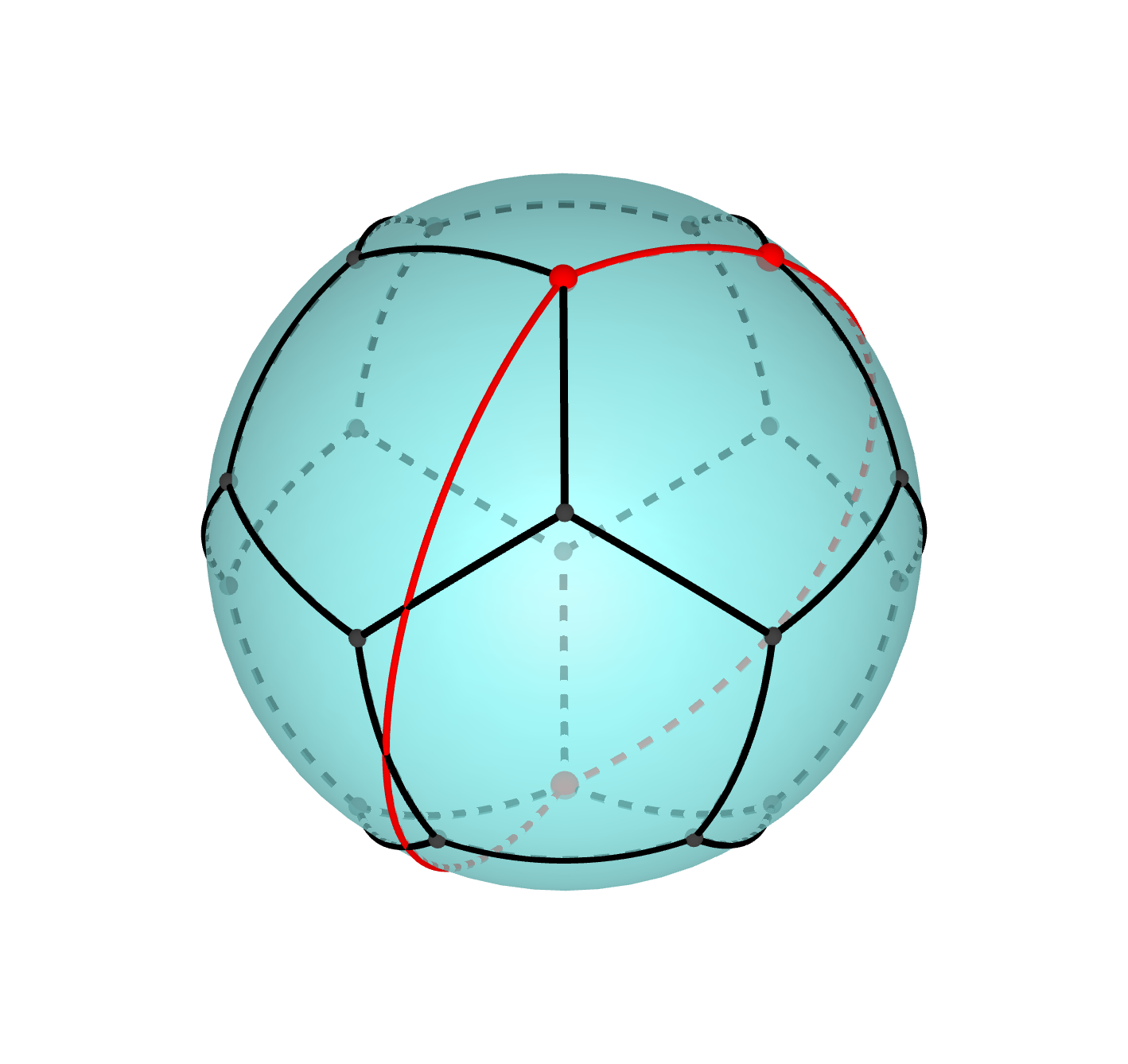}
\hfill\adjincludegraphics[scale=0.08,trim={{.1\width} 0 {.1\width} 0}]{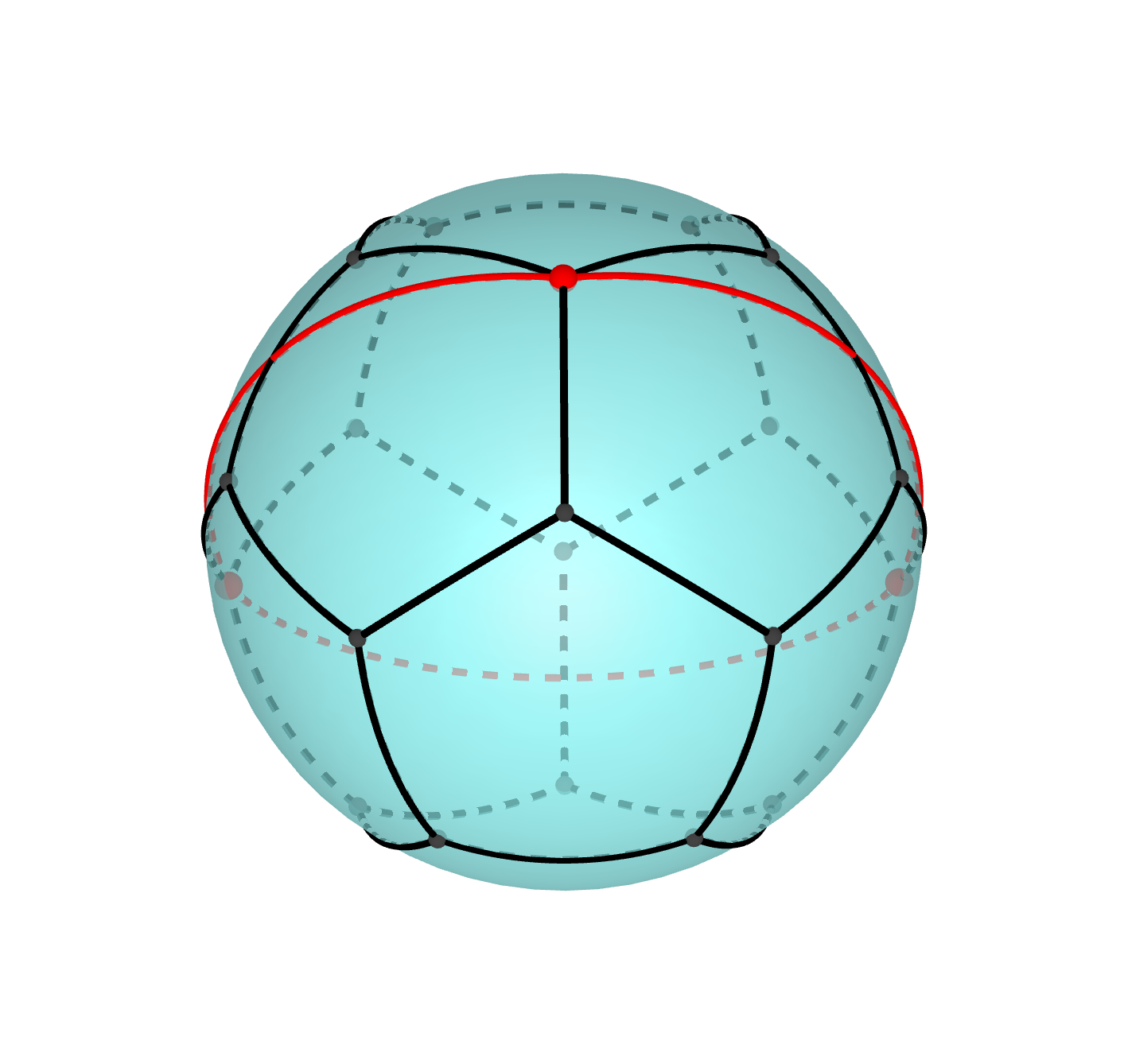}
\hfill \

\noindent 
\hspace{12mm} $1\pm \frac{1}{6}$, (1,4,5) \hspace{16mm} $1\pm \frac{1}{6}$, (1,5,4) \hspace{16mm} $1\pm \frac{1}{6}$, (3,3,4) \\\\


\begin{figure}[H]
    \centering
    \hfill
    \includegraphics[scale=0.18]{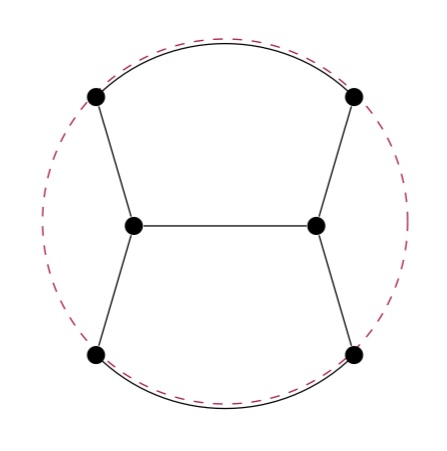}
    \hfill
    \includegraphics[scale=0.18]{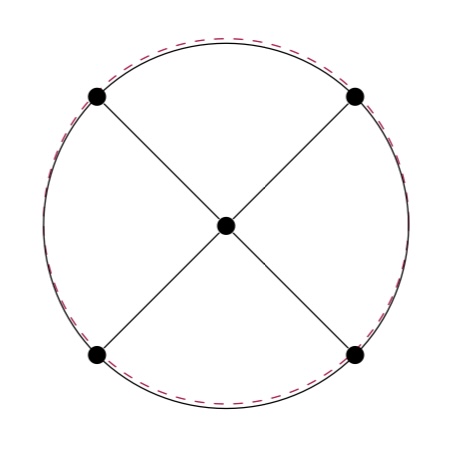}
    \hfill
    \includegraphics[scale=0.18]{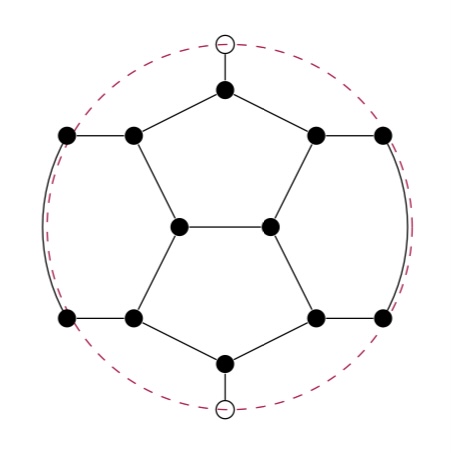}
    \hfill
    \includegraphics[scale=0.18]{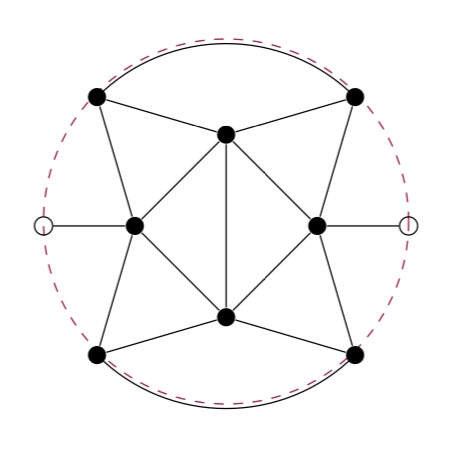}
    \hfill
    
    \noindent
    \hspace{10.3mm} Cube \hspace{12.7mm} Octahedron  \hspace{9mm} Dodecahedron \hspace{7mm} Icosahedron \hfill
    
    \caption{Dessins of the Platonic hemisphere}
    \label{Figure:Hemisphere_Platonic}
\end{figure}
For simplicity, we draw $\bullet - \bullet$ to represent $\bullet - \circ - \bullet$ if no confusion may occur. The dashed line represents the boundary of the hemisphere.

\begin{figure}[H]
    \centering 
    \adjincludegraphics[scale=0.22,trim={{.1\width} 0 {.1\width} 0}]{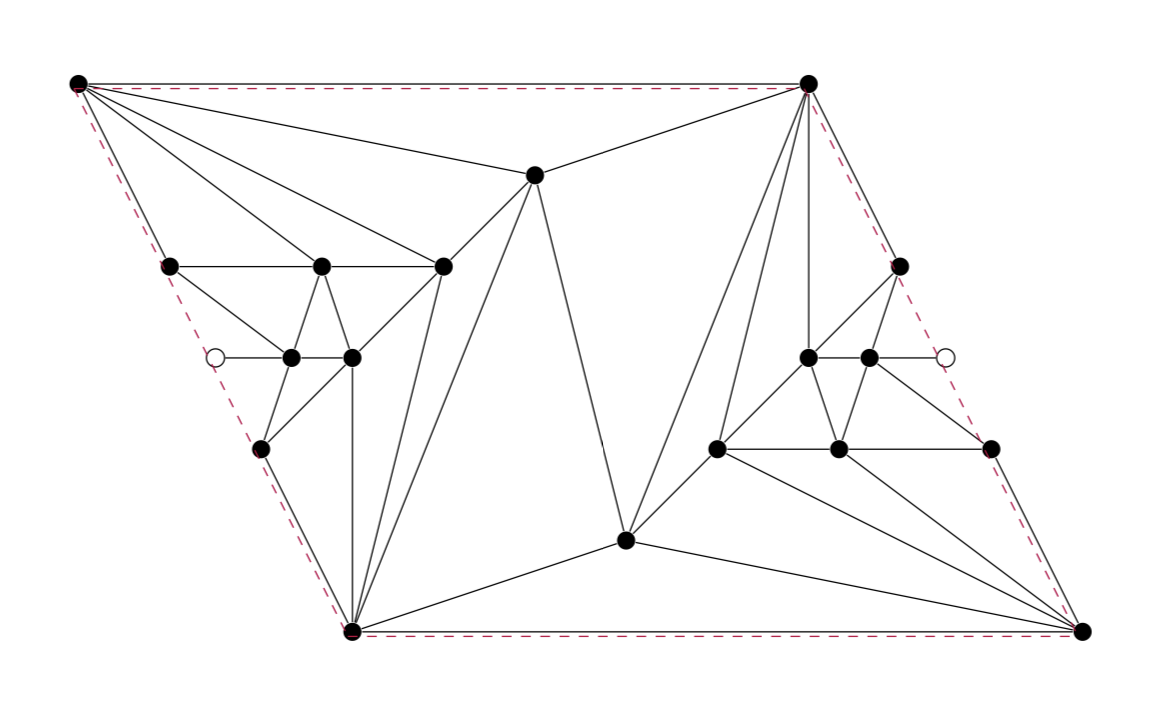}
    \hfill
    \adjincludegraphics[scale=0.11,trim={{.1\width} 0 {.1\width} 0}]{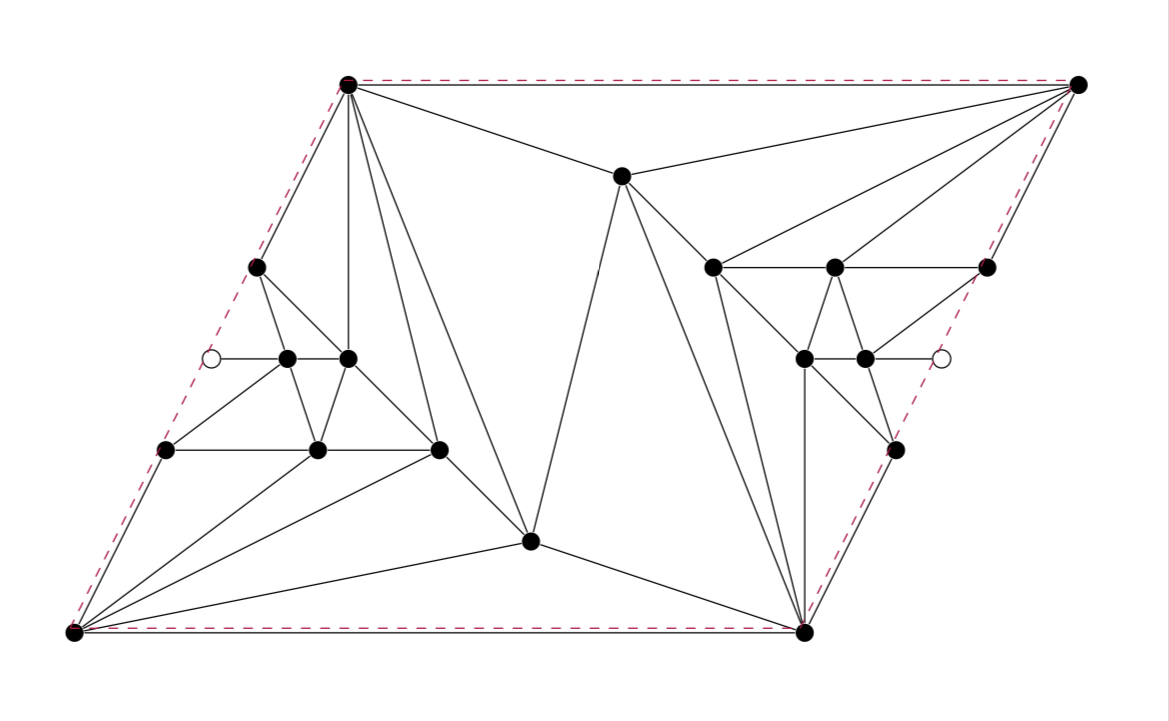}
    \hfill
    \adjincludegraphics[scale=0.22,trim={{.1\width} 0 {.1\width} 0}]{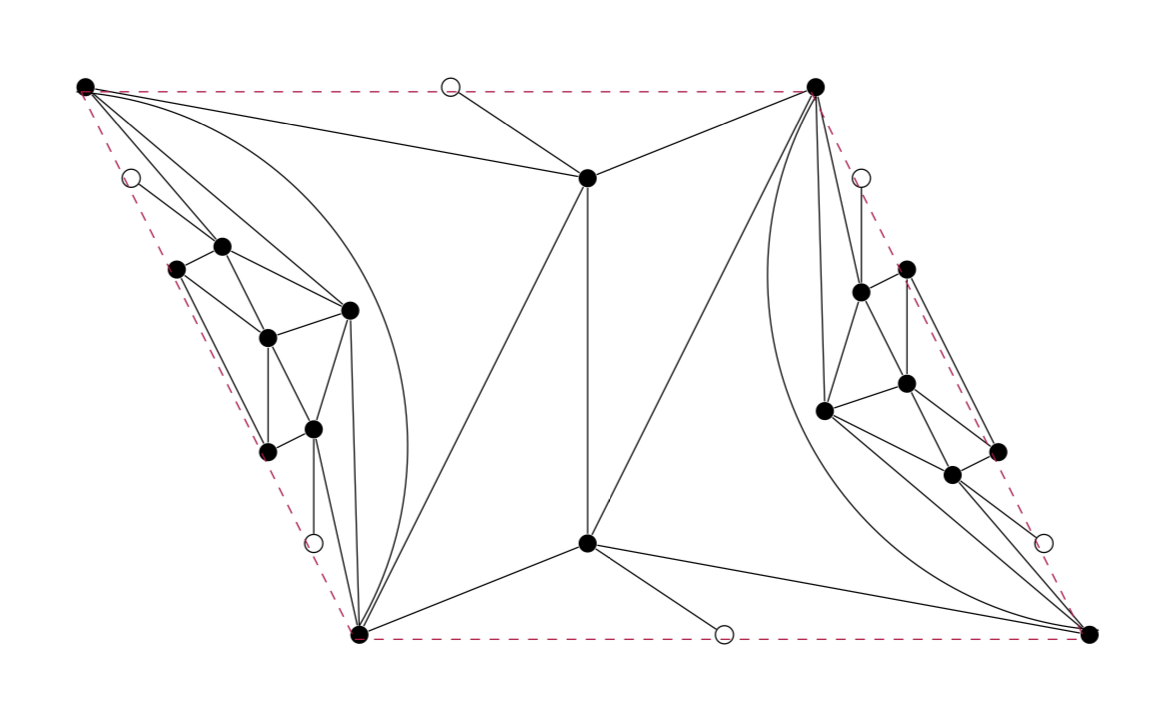}
    \\
      $3/10$, (1,2,2), R \hspace{18mm}   $3/10$, (1,2,2), L  \hspace{18mm} $3/10$, (1,2,2), B 
    \\
    \adjincludegraphics[scale=0.22,trim={{.1\width} 0 {.1\width} 0}]{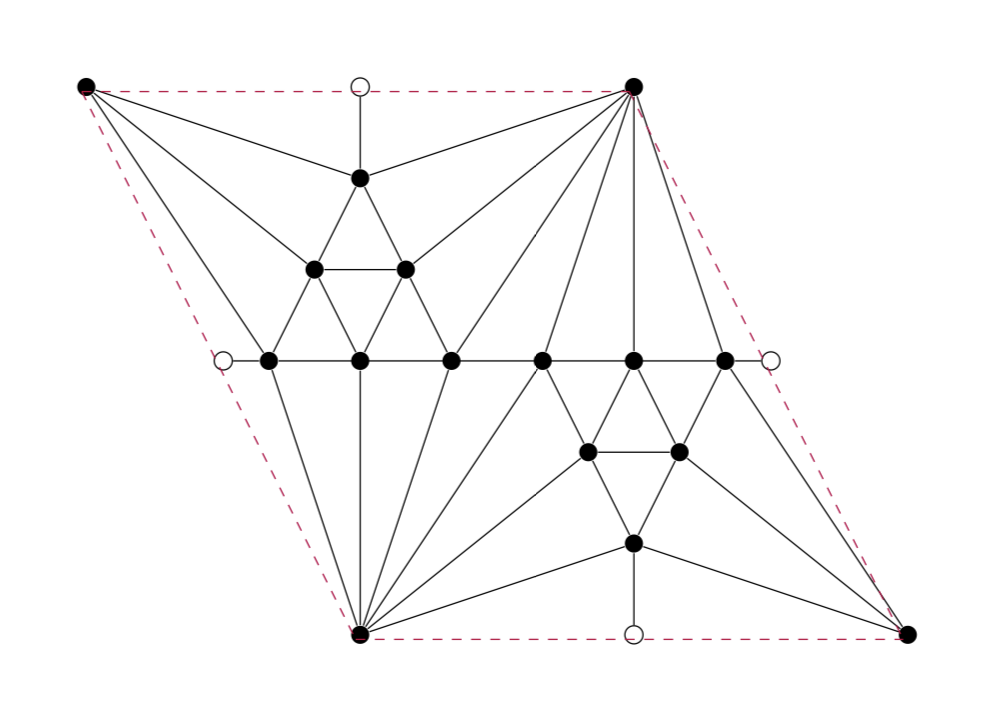}
    \hfill
    \adjincludegraphics[scale=0.22,trim={{.1\width} 0 {.1\width} 0}]{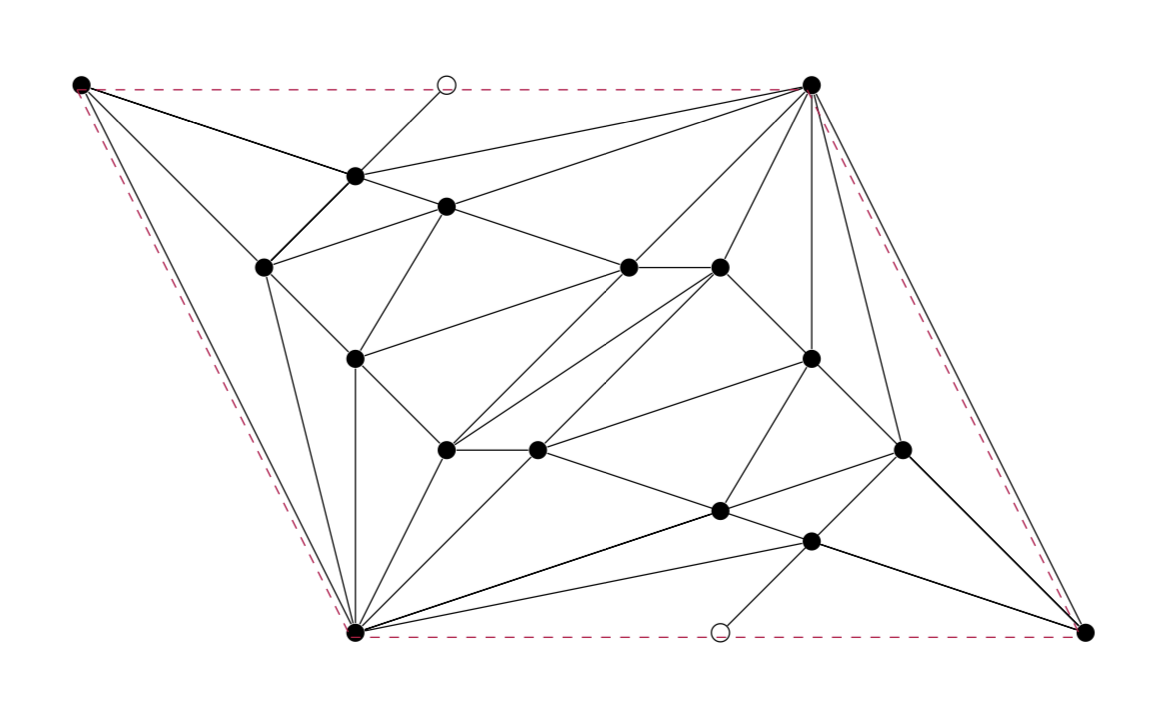}
    \hfill
    \adjincludegraphics[scale=0.11,trim={{.1\width} 0 {.1\width} 0}]{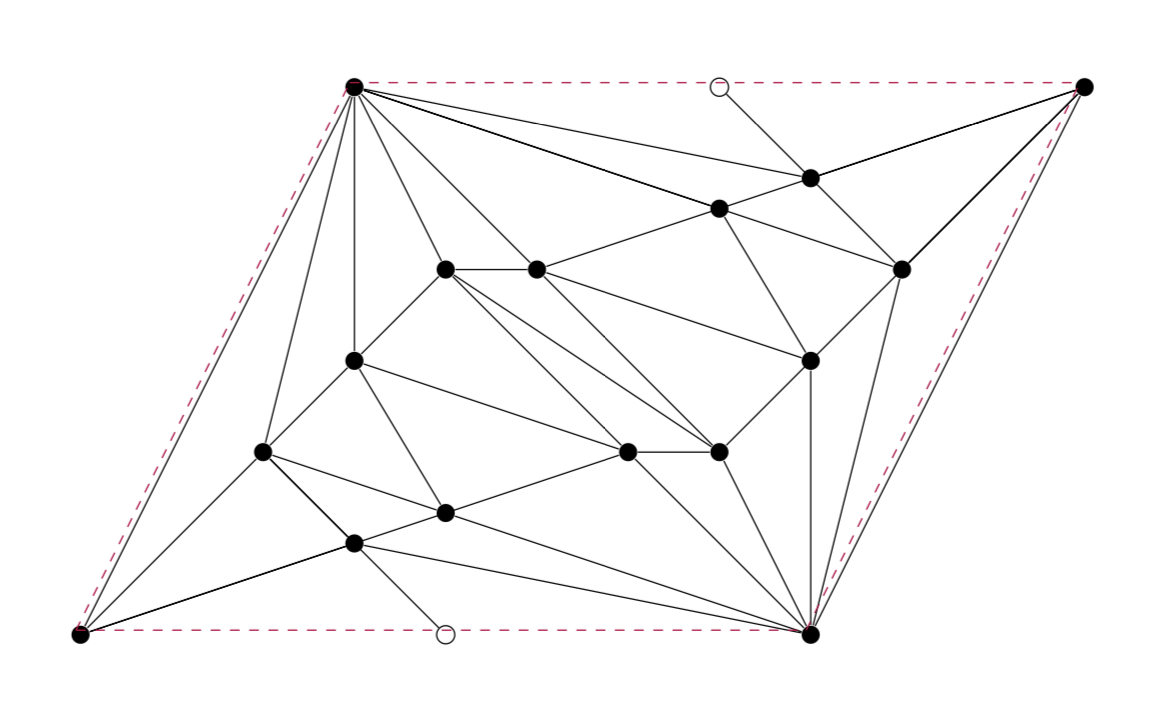}
    \\
    $13/10$, (2,2,2) \hspace{19mm}   $13/10$, (1,2,3)  \hspace{19mm} $13/10$, (1,3,2)
    \caption{Dessins for $L_{13/10}$}
    \label{Figure:Dessins13/10}
\end{figure}

In Figure~\ref{Figure:Dessins13/10}, the graph 3/10, (1,2,2), R (resp. L, B) corresponds to the case ``Icosahedral, $\frac{3}{10}$, (1,2,2)'' with hemisphere attached to the right edge (resp. left edge, bottom edge).

\bibliographystyle{plain}
    
\bibliography{Finite_monodromy}

\end{document}